\documentclass{article}

\usepackage{amsmath}
\usepackage{amssymb,amsthm,mathtools}
\usepackage{array}
\usepackage[utf8]{inputenc}
\usepackage{tikz}
\usetikzlibrary{shapes,backgrounds,matrix,arrows,automata,positioning}
\usepackage{tikz-cd}
\usepackage{scalefnt}
\usepackage{enumitem}
\usepackage{verbatim}
\usepackage[hidelinks]{hyperref}
\usepackage{mathrsfs}
\usepackage{stmaryrd} 
\usepackage{hhline}
\usepackage{enumitem}
\usepackage[skip=10pt plus1pt]{parskip}
\usepackage{newunicodechar}

\newunicodechar{ń}{{\'n}}
\newunicodechar{ě}{{\v e}}
\newunicodechar{Č}{{\v C}}
\newunicodechar{œ}{\oe}

\tikzset{shorten <>/.style={shorten >=#1,shorten <=#1}}

\setenumerate{align=left}

\usepackage{setspace}
\usepackage[top=3cm,bottom=3cm,left=3cm,right=3cm]{geometry}
\usepackage{titlesec}
\titleformat{\paragraph}[runin]{\bfseries}{\theparagraph}{7em plus .2em minus .3em}{}[.]

\theoremstyle{plain}
\newtheorem{theorem}{Theorem}

\newtheorem{proposition}[theorem]{Proposition}
\newtheorem{lemma}[theorem]{Lemma}
\newtheorem{corollary}[theorem]{Corollary}

\theoremstyle{definition}
\newtheorem{definition}[theorem]{Definition}

\newtheorem*{convention}{Convention}

\theoremstyle{remark}
\newtheorem{example}[theorem]{Example}
\newtheorem{remark}[theorem]{Remark}

\numberwithin{theorem}{section}

\DeclarePairedDelimiterX{\set}[1]{\{}{\}}{#1}
\newcommand{\midvert}{\,\delimsize|\,\mathopen{}}
\DeclarePairedDelimiterX{\setst}[2]{\{}{\}}{#1\ \midvert\ #2}
\DeclarePairedDelimiterX{\segcc}[2]{[}{]}{#1, #2}
\DeclarePairedDelimiterX{\segoo}[2]{]}{[}{#1, #2}
\DeclarePairedDelimiterX{\segco}[2]{[}{[}{#1, #2}
\DeclarePairedDelimiterX{\segoc}[2]{]}{]}{#1, #2}
\DeclarePairedDelimiterX{\sem}[1]{\llbracket}{\rrbracket}{#1} 

\makeatletter
\newcommand{\optionaldesc}[2]{%
	\phantomsection
	#1\protected@edef\@currentlabel{#1}\label{#2}%
}
\makeatother

\newcommand{\N}{\mathbb{N}}
\DeclareMathOperator{\powerset}{\mathscr{P}}
\newcommand*{\restr}[2]{\left.#1\right\rvert_{#2}} 
\newcommand{\ovl}{\overline}

\newcommand{\into}{\hookrightarrow}
\newcommand{\from}{\leftarrow}

\newcommand{\uset}{{\uparrow}}
\newcommand{\dset}{{\downarrow}}
\DeclareMathOperator{\Up}{Up} 
\DeclareMathOperator{\ClU}{Clp^{\uparrow}}
\DeclareMathOperator{\Spec}{Spec}
\renewcommand{\hat}{\widehat}

\newcommand{\cat}[1]{\ensuremath{\mathbf{#1}}}
\newcommand{\cBA}{\cat{BA}}

\newcommand{\cPoset}{\cat{Poset}}
\newcommand{\DL}{\cat{DL}}
\newcommand{\cDL}{\cat{DL}}
\newcommand{\cHeyt}{\cat{HA}}
\newcommand{\cPriestley}{\cat{Priestley}}
\newcommand{\cFinSet}{\cat{FinSet}}
\newcommand{\cFinSetInj}{\cat{FinSetInj}}
\newcommand{\cKHaus}{\cat{KHaus}}
\newcommand{\cKOrd}{\cat{KOrd}}
\newcommand{\cSet}{\cat{Set}}

\newcommand{\cBoolSp}{\cat{BoolSp}}
\newcommand{\cEsakia}{\cat{Esakia}}

\newcommand{\cC}{\cat{C}}
\newcommand{\cD}{\cat{D}}
\newcommand{\cE}{\cat{E}}
\newcommand{\cT}{\cat{T}}
\newcommand{\cI}{\cat{I}}
\newcommand{\cJ}{\cat{J}}

\DeclareMathOperator{\cInd}{\cat{Ind}}
\DeclareMathOperator{\colim}{colim}

\newcommand{\op}{\mathrm{op}} 

\newcommand{\ladj}[1]{{#1}^*} 
\newcommand{\radj}[1]{{#1}^{\#}} 

\DeclareMathOperator{\Hom}{Hom}

\DeclareMathOperator{\Sub}{Sub}

\DeclareMathOperator{\Mod}{Mod} 
\DeclareMathOperator{\Tp}{Tp} 

\newcommand{\hD}{\mathcal{D}} 
\newcommand{\sP}{\mathcal{P}} 
\newcommand{\sQ}{\mathcal{Q}} 
\newcommand{\sS}{\mathcal{S}} 

\usepackage[style=numeric,maxnames=50]{biblatex}
\title{On duality and model theory for polyadic spaces}
\author{Sam~v.~Gool and Jérémie~Marquès}
\date{\today}

\begin{document}

\maketitle
\abstract{This paper is a study of first-order coherent logic from the point of view of duality and categorical logic. We prove a duality theorem between coherent hyperdoctrines and open polyadic Priestley spaces, which we subsequently apply to prove completeness, omitting types, and Craig interpolation theorems for coherent or intuitionistic logic. Our approach emphasizes the role of interpolation and openness properties, and allows for a modular, syntax-free treatment of these model-theoretic results. As further applications of the same method, we prove completeness theorems for constant domain and Gödel-Dummett intuitionistic predicate logics.}


\section{Introduction}
The aim of this paper is to show how the point of view of duality and categorical logic can be used to gain insight into, and generalize, some classical theorems of model theory. Our main object of study is \emph{polyadic spaces}.%
\footnote{The term \emph{polyadic space}, as we use it in this paper, comes from Joyal's note \cite{joyalpoly}, in reference to Halmos' \emph{polyadic algebras}. There exists another, entirely different, use of the term ``polyadic space'' \cite{Mrowka1970}, as a generalization of ``dyadic space,'' introduced independently from and around the same time as Joyal's. In case confusion between the two notions might arise, one could use the slightly longer name ``polyadic type space'' for the notion we study in this paper.}
They are the pointwise Priestley duals of hyperdoctrines, which were introduced by Lawvere as an algebraization of first-order logic in \cite{LawvereAdjoint,LawvereEqHyp}.

Let us give a quick description of hyperdoctrines and polyadic spaces over the base category $\cFinSet$; the general definition is given in Sections~\ref{sec:hyperdoctrines} and \ref{sec:duality}. A hyperdoctrine may be thought of as an algebra representing a first-order theory, cf., e.g., \cite[Ch.~5]{CoumansThesis} or \cite{See83}. Elements of this algebra are formulas modulo equivalence, and the operations of the algebra model conjunction, disjunction and quantification. Hyperdoctrines are multi-sorted algebras, having one sort for each finite cardinal $n$. This distinguishes them from the more classical polyadic and cylindric algebras \cite{Hal62, HMT71, HMT85}, which are single-sorted. In a hyperdoctrine representing a theory $T$, the elements of the $n^{\text{th}}$ sort are the formulas whose free variables are taken among $x_1,...,x_n$, modulo equivalence in the theory $T$.

The base logic that we consider in this paper is \emph{coherent (first-order) logic}. This logic may be understood in two equivalent ways: first, as the fragment of classical first-order logic that only uses equality, existential quantification, and finitary conjunctions and disjunctions; second, as the fragment of intuitionistic first-order logic that does not have implication nor universal quantification. The propositional fragment of coherent logic is algebraized by distributive lattices, so that a coherent hyperdoctrine is a collection of distributive lattices, organized into a functor $\cFinSet \to \cDL$ from finite sets to distributive lattices.

Applying Priestley duality, we get a collection of Priestley spaces indexed by finite sets.
The $n^{\text{th}}$ space is known as the \emph{space of $n$-types} in model theory: its points can be thought of as $n$-pointed models of the theory modulo equivalence in coherent logic, i.e., models equipped with $n$ distinguished points and two such models are considered equal if they satisfy the same coherent formulas. What makes this intuition correct is the associated completeness theorem: each complete $n$-type is realized by a model. These spaces of types are organized into a functor $\sS \colon \cFinSet^\op \to \cPriestley$ from finite sets to Priestley spaces; this is the prime example of a \emph{polyadic space}. The operations of existential quantification and equality of coherent logic are then realized topologically by taking direct images under maps $\sS(f)$, for appropriate choices of $f$; see Proposition~\ref{prop:left-adjoint-dually}. Note that, while in a classical (Boolean) first-order theory, the order on the space of $n$-types is discrete, proper inclusions between types can occur in coherent logic.

Our first main result, Theorem~\ref{thm:duality}, gives a dual equivalence between coherent hyperdoctrines and open polyadic Priestley spaces.
Our axiomatization of these spaces relies on two essential aspects: \emph{interpolation} and \emph{openness}.
Correspondingly, we need two duality-theoretic results: first, interpolation is auto-dual (Proposition~\ref{prop:dual-interpolation}), and thus appears both on the topological and algebraic sides of the duality; second, openness dually corresponds to the existence of adjoints (Proposition~\ref{prop:left-adjoint-dually}).
The duality theorem allows us to exclusively work on the topological side in the remainder of the paper, in which we use polyadic spaces to give proofs of completeness, omitting types and Craig interpolation theorems.

We briefly note the origins of our approach to hyperdoctrines in this paper. The first source of the idea can be found in Joyal's short 1971 announcement~\cite{joyalpoly}, while the authors initially learned of this approach through the 2019 lecture course~\cite{joyalcourse}. Joyal's work in particular already contained the formulation of the dual of a Beck-Chevalley condition as amalgamation and the notion of model that we use in this paper. The results announced in~\cite{joyalpoly} were never fully published, but our results here are much indebted to Joyal's view. Our contributions in this paper expand and extend this view, and show how it is naturally placed in the context of Priestley duality and compact ordered spaces.

Polyadic spaces are closely related to the \emph{type space functors} of \cite{EAGLE2021102907,HAYKAZYAN_2019,Kamsma_2022}, \emph{type categories} of \cite{Knight2007}, and \emph{compact abstract theories} of \cite{Ben-Yaacov03}; we point out two distinguishing features of our work here. First, in these works, the type spaces are endowed with the \emph{spectral} topology.
In our work here, on the other hand, the specialization order on these spaces will play a crucial role through the interpolation properties, and for this reason we will work with \emph{Priestley} spaces instead.
This in particular allows us to give an order-topological characterization of type space functors as the open polyadic Priestley spaces over the base category $\cFinSet$ (Theorem~\ref{thm:duality}).
Second, instead of working exclusively over $\cFinSet$, our definition of polyadic space is relative to a more general base category $\cC$ of ``small objects.'' For example, by taking $\cC$ to be the category of finite graphs or of finite linear orders, polyadic spaces over $\cC$ represent first-order theories extending respectively the theories of graphs or linear orders. The only condition that we require on $\cC$ is that it has pushouts, or see Remark~\ref{rmk:fw-colim-butterflies} for a weaker condition.

Since what we do in this paper works not only for Priestley spaces but for compact ordered spaces, we will often place ourselves in this more general context on the topological side. An algebraic dual and logical interpretation of this more general notion of polyadic space will be given in a forthcoming paper, based on Abbadini and Reggio's duality for compact ordered spaces \cite{AbbadiniThesis2021,AbbadiniReggio_2020}. This yields a variation on the continuous syntactic categories of \cite{AlbHar16}, and seems related to the link between compact abstract theories and continuous logic noted in \cite{yaacov_berenstein_henson_usvyatsov_2008}.

The paper falls into two parts: in Sections~\ref{sec:prelim}--\ref{sec:main-lemma} we introduce the duality between hyperdoctrines and polyadic spaces, using first-order coherent logic and its model theory as the guiding example. In the remaining Sections~\ref{sec:complete-coherent}--\ref{sec:variations-int} we apply this point of view to prove various results in the model theory of  first-order logics in the more general setting of compact ordered spaces. These results in particular generalize known results in coherent, intuitionistic and classical first-order logics.

In particular, after we set up some basic notation and recall preliminaries in Section~\ref{sec:prelim}, we define in Sections~\ref{sec:hyperdoctrines}~and~\ref{sec:duality} coherent hyperdoctrines and characterize their Priestley duals, open polyadic spaces, leading to our first main result, the dual equivalence of Theorem~\ref{thm:duality}. In Section~\ref{sec:models}, we explain how models and types can be viewed through the lens of hyperdoctrines and polyadic spaces. Section~\ref{sec:main-lemma} is devoted to proving a technical result that we call the interpolation extension principle (Proposition~\ref{prop:main}). This principle is central to our approach, as it allows us to extend the interpolation properties of polyadic spaces to their inductive completions.

Interpolation properties allow us to use what is usually called the method of diagrams in model theory to prove various completeness theorems (Sections~\ref{sec:complete-coherent}, \ref{sec:complete-intuitionistic} and \ref{sec:variations-int}).
In Section~\ref{sec:access}, we see how to compute filtered colimits in categories of models, which is needed in Section~\ref{sec:FO-interpol}.
Under the additional hypothesis of openness, we can prove an omitting types theorem and Craig interpolation for intuitionistic logic (Sections~\ref{sec:omitting-types} and \ref{sec:FO-interpol}, respectively).

\section{Lattices, Priestley duality, and categories}
\label{sec:prelim}

In this section, we collect preliminaries and notation that will be used throughout the paper. For basic definitions and notations of lattice theory and category theory, we follow the conventions in \cite{GG22} unless noted otherwise. One notable difference with \cite{GG22} is that we will represent distributive lattices as clopen \emph{up}-sets, rather than down-sets, of their Priestley dual space, see further in Section~\ref{sec:duality}. This will fit better with existing literature when we consider Kripke semantics for intuitionistic logics below.

\paragraph{Basic order theory}
Throughout the paper, we use the word \emph{order} to mean a reflexive, transitive and anti-symmetric relation. Linearity is always explicitly mentioned when it is assumed. Given a subset $U$ of a poset $X$, we write ${\uparrow} U \coloneqq \setst{x \in X}{\exists y \in U\colon y \leq x}$ for its \emph{upward closure}, and symmetrically ${\downarrow} U$ for its \emph{downward closure}. We say that $U$ is an \emph{up-set} if $U = {\uparrow}U$ and a \emph{down-set} if $U = {\downarrow} U$. An order-preserving function $f \colon X\to Y$ is called \emph{bounded} if the direct image $f[U]$ of each up-set $U \subseteq X$ is an up-set. When $f \colon A \to B$ is an order-preserving map between ordered sets, we denote its left adjoint, if it exists, by $\ladj{f}$, and its right adjoint, if it exists, by $\radj{f}$.

A distributive lattice is an ordered set in which finite meets and joins exist and distribute over each other; in particular, all distributive lattices in this paper are assumed to have a least and greatest element, denoted $\bot$ and $\top$, respectively, and homomorphisms are required to preserve them. The category of distributive lattices with homomorphisms is denoted $\cDL$.
When $A$ is a distributive lattice and $a \in A$, the principal down-set ${\downarrow} a$ is a distributive lattice itself. We introduce the following notation for the \emph{projection map}
\[ p_a \colon A \to {\downarrow} a, \quad p_a(b) \coloneqq a \wedge b,\]
and we note that this is a surjective homomorphism. It may also be described as the quotient of $A$ by the congruence generated by the relation $a = \top$.

A distributive lattice $A$ is a \emph{Heyting algebra} if, and only if, for all $a,b \in A$, there exists a necessarily unique element $a\to b \in A$ such that $c \leq a\to b \iff c \land a \leq b$ for all $c \in A$. Equivalently, $p_a$ has a right adjoint $q_a$ for each $a \in A$ and $a\to b = q_a(p_a(b))$. In a Heyting algebra, we write $\neg a \coloneqq a \to \bot$. The category of Heyting algebras with Heyting homomorphisms, i.e., lattice homomorphisms that moreover preserve $\to$, is denoted $\cHeyt$. A \emph{Boolean algebra} is a Heyting algebra in which $a \vee \neg a = \top$ for all $a$; the full subcategory is denoted $\cBA$. We write $\powerset(X)$ for the Boolean algebra of subsets of a set $X$.

Given two posets $A$ and $B$, an \emph{order relation} from $A$ to $B$ is an up-set of $A^{\op}\times B$. If $f \colon A \to B$ is a monotone function, then the upward closure of its graph in $A^{\op} \times B$, i.e., the relation $\setst{(a,b) \in A^{\op} \times B}{f(a) \leq b}$, is an order relation. We will say an order relation $R \subseteq A^\op \times B$ is \emph{represented} by $f \colon A \to B$ if $R$ is the upward closure of the graph of $f$. The composition of two order relations $R \subseteq A^\op \times B$ and $S \subseteq B^\op \times C$ is
\[ R \cdot S \coloneqq \setst{(a,c) \in A^\op \times C}{\exists b\in B \colon R(a,b) \land S(b,c)}\text{.} \]
Note that if $R$ and $S$ are represented respectively by $f$ and $g$, then $R\cdot S$ is represented by the composite $fg$, which is our notation for `first $f$, then $g$', also see ``Categories'' below.

\paragraph{Compact ordered spaces}
A \emph{compact ordered space} is a a compact topological space $X$ equipped
with an order $\leq$ that is closed as a subset of $X^2$ with respect to the product topology. These spaces generalize compact
Hausdorff spaces to the ordered setting: compact Hausdorff spaces are the compact ordered
spaces whose order is discrete (the equality order), and any compact ordered space is necessarily Hausdorff. A \emph{morphism} between compact ordered spaces
is a continuous order-preserving function.
We write $\cKOrd$ for the category of compact ordered spaces and $\cKHaus$ for the full
subcategory of compact Hausdorff spaces.

We recall the following basic fact about compact ordered spaces, that we will use in Section~\ref{sec:main-lemma}.

\begin{lemma}\label{lem:cofiltered-limit-non-empty}
	For any cofiltered diagram $D \colon \cI \to \cKOrd$, if $D(i) \neq \emptyset$ for all $i \in \cI$, then $\lim D$ is non-empty.
\end{lemma}

\begin{proof}
	Limits in $\cKOrd$ can be computed as in $\cSet$, since arbitrary products and closed subspaces of compact ordered spaces are compact ordered. Hence, $\lim D$ can be computed as a filtered intersection of non-empty closed subsets of $\prod_i D(i)$, which is non-empty because $\prod_i D(i)$ is a compact space.
\end{proof}

An alternative, purely topological description of a compact ordered space $X$ can be given by considering the topological space $X^{\uparrow}$, defined on the same set of points as $X$, but equipped with the subtopology of \emph{open up-sets} on $X$.  The original topology and the order of $X$ can be recovered from the space $X^{\uparrow}$. A function $X \to Y$ between compact ordered spaces is called \emph{lower semi-continuous} if it is continuous as a function $X^{\uparrow} \to Y^{\uparrow}$ (this terminology is standard in the literature). Any morphism of compact ordered spaces is lower semi-continuous, but the converse is not true. Analogously, we may define the topological space $X^{\downarrow}$ on the same set of points as $X$ with the topology of \emph{open down-sets}.
The topological spaces of the form $X^{\uparrow}$ or  $X^{\downarrow}$, for $X$ a compact ordered space, are called \emph{stably compact spaces}. For more on the general theory of compact ordered and stably compact spaces, the reader may refer to, e.g., \cite{Lawson11} or \cite[Sec.~2.3]{GG22}.

\paragraph{Priestley duality}
Priestley duality \cite{Pri1970} is a dual equivalence of categories
\begin{equation}\label{eq:Priestley-duality}
	\Spec \ \colon \cat{DL}^\op \leftrightarrows \cat{Priestley} \ \colon \ClU.
\end{equation}
We briefly recall the definitions and some basic properties of $\cat{Priestley}$, $\ClU$ and $\Spec$ that we will rely on; see e.g. \cite[Ch.~3]{GG22} or \cite[Sec.~1.5]{DST2019} for more detailed accounts.

A \emph{Priestley space} is a compact ordered space that is
moreover \emph{totally order disconnected}, i.e., for any $x, y \in X$, if $x
\nleq y$, then there exists a clopen up-set $U$ of $X$ such that $x \in U$ and
$y \not\in U$; note that total order disconnectedness in particular implies that
$\leq$ is closed as a subset of $X^2$. The category $\cat{Priestley}$ is the full subcategory of $\cKOrd$ on the Priestley spaces. The correspondence between compact ordered spaces and stably compact spaces mentioned in the previous subsection restricts to an isomorphism between \cat{Priestley} and the category of \emph{spectral} spaces with functions whose inverse image preserves compact-open sets.
For later use in the paper, we note two consequences of total order disconnectedness.

\begin{lemma}\label{lem:adh-spec}
	Let $X$ be a Priestley space and $A \subseteq X$ a subset. Then ${\uparrow}\ovl{A}$ is the intersection of all the clopen up-sets containing $A$.
\end{lemma}

\begin{proof}
	Let $x \not\in {\uparrow}\ovl{A}$. By total order disconnectedness, for each $a \in A$, pick a clopen up-set $K_a$ containing $a$ and not $x$. Since $A$ is closed, it is compact; pick a finite subcover $(K_a)_{a \in F}$ of $A$. Then $\bigcup_{a \in F} K_a$ is a clopen up-set containing $A$ and not $x$.
\end{proof}

\begin{lemma}[Strong Priestley separation property]\label{lem:Priest-sep}
	Let $X$ be a Priestley space. Let $A,B\subseteq X$ be respectively a closed up-set and a closed down-set. If $A$ and $B$ are disjoint, then they are separated by a clopen up-set $U \subseteq X$ containing $A$ and disjoint from $B$.
\end{lemma}

\begin{proof}
	By Lemma~\ref{lem:adh-spec}, for each $b \in B$, pick a clopen up-set $K_b$ that contains $A$ and not $b$. Then $(K_b^c)_{b \in B}$ covers $B$, so since $B$ is closed, hence compact, pick a finite subcover $(K_b^c)_{b \in F}$ of $B$. Then $U \coloneqq \bigcap_{b \in B} K_b$ is a clopen up-set containing $A$ and disjoint from $B$.
\end{proof}

We now describe the functors $\ClU$ and $\Spec$ in the dual equivalence \eqref{eq:Priestley-duality}. The clopen up-sets of a Priestley space $X$ form a distributive lattice, that we
denote by $\ClU(X)$, and if $f \colon X \to Y$ is a morphism between Priestley
spaces then $f^{-1} \colon \ClU(Y) \to \ClU(X)$ is a lattice homomorphism. The functor in the
other direction, $\Spec \colon \cat{DL} \to \cat{Priestley}$, takes a
distributive lattice $A$ and equips the set $\Spec(A) \coloneqq \Hom_\cat{DL}(A, 2)$
with the pointwise order{\footnotemark} and topology inherited from the product topology on
$2^L$, where $2 = \{0, 1\}$ is ordered by $0 \leq 1$ and equipped with the
discrete topology. In what follows, when $A$ is a lattice, we write $A_*$
for the Priestley space $\Spec A$. A lattice homomorphism $h \colon A \to B$ is
sent to the function $h_* \colon B_* \to A_*$ that is defined by sending $x
\colon B \to 2$ to $x \circ h \colon A \to 2$. Priestley duality is actually a poset-enriched equivalence, in the sense that,
for any pair of lattice homomorphisms $h, k \colon A \rightrightarrows B$, we
have that $h \leq k$ pointwise if, and only if, $h_* \leq k_*$ pointwise.
\footnotetext{The definition of the
partial order on $\Spec A$ is subject to some discussion (and confusion) in the
literature. In this paper, we use the convention that, for $x, y \in \Spec A$,
$x \leq y$ iff for all $a \in A$, $x(a) = 1$ implies $y(a) = 1$; that is, $x
\leq y$ means that the prime filter $x^{-1}(1)$ is a subset of the prime filter
$y^{-1}(1)$. Some references, e.g. \cite[Ch.~3]{GG22}, use the reverse of this
order, and there are good arguments for this, also see that chapter.}

The unit of the dual equivalence \eqref{eq:Priestley-duality} is an isomorphism between a distributive lattice $A$ and $\ClU(A_*)$; for any $a \in A$, we write $\hat{a} \coloneqq \setst{x \in A_*}{x(a) = 1}$ for the corresponding clopen up-set of $A_*$. The fact that the map $a \mapsto \widehat{a}$ is injective relies on the prime filter-ideal theorem, a weak form of the axiom of choice, see, e.g. \cite[Thm.~3.10]{GG22}.

A Priestley space $A_*$ whose dual $A$ is a Heyting algebra is called an \emph{Esakia space}.
A well-known equivalent condition is that for every open $U \subseteq A_*$, the set ${\downarrow}U \subseteq A_*$ is open. Moreover, if $f_* \colon A_* \to B_*$ is a morphism between Esakia spaces, then its dual $f$ is a morphism of Heyting algebras
if and only if $f$ is bounded; in Remark~\ref{rmk:Heyt-dual} below, we explain how these well-known facts can be deduced as corollaries to our results in Section~\ref{sec:duality}. We write
$\cEsakia$ for the category of Esakia spaces and continuous order-preserving morphisms that are \emph{bounded}, where we recall that a function $f$ between posets is called \emph{bounded} if the direct image under $f$ of any up-set is an up-set.
Priestley duality restricts to an equivalence between $\cHeyt^\op$ and
$\cEsakia$ that is known as \emph{Esakia duality} \cite{Esakia2019}.

A distributive lattice is a Boolean algebra if, and only if, the partial order on its spectrum is trivial. In
this case (only), the Priestley topology, open up-set topology, and open
down-set topology all coincide. We thus also call a Priestley space
\emph{Boolean} if its partial order is trivial, and in this case the definition
requires exactly that the topology is compact, Hausdorff, and zero-dimensional.%
\footnote{Boolean spaces have also been referred to as \emph{Stone} spaces in
the literature.} The further restriction of the dual equivalence is called (Boolean) \emph{Stone duality} \cite{Sto1938BA}.

\paragraph{Coherent first-order logic} A formula of classical first-order logic is \emph{coherent} if it uses only finitary conjunctions and disjunctions, the existential quantifier and equality. A first-order theory is coherent if all its axioms are of the form $\forall x_1,\dots,x_n\colon \varphi(x_1,\dots,x_n) \to \psi(x_1,\dots,x_n)$ with $\varphi$ and $\psi$ coherent formulas. We write this axiom $\varphi(x_1,\dots,x_n) \vdash \psi(x_1,\dots,x_n)$ in the context of coherent logic.

\paragraph{Ordered Stone--Čech compactification} The forgetful functor $\cKOrd \to \cPoset$ has a left adjoint $\beta \colon \cPoset \to \cKOrd$ that we call the \emph{ordered Stone-Čech compactification}, generalizing the well known Stone-Čech compactification which is left adjoint to the forgetful functor $\cKHaus \to \cSet$. If $X$ is a poset, we can describe $\beta X$ as the Priestley dual of the lattice of up-sets of $X$. The poset $X$ is a dense subspace of $\beta X$. If $u \colon X \to A$ is an order-preserving function from a poset to a compact ordered space, we write $\ovl{u}$ for its unique extension by continuity to $\beta X$.

\paragraph{Categories}
Given two morphisms $f \colon A\to B$ and $g \colon B\to C$ in a category, we write either $g \circ f$ or $fg$ for their composite; note the change of order according to whether or not the symbol $\circ$ is used. We will mostly work in categories whose $\Hom$-sets are equipped with an order, and we often use \emph{lax commutative diagrams}, also simply called \emph{lax diagrams}, which in this setting means that some faces in the diagram represent inequalities instead of equalities. Concretely, a lax diagram
\begin{equation*}
\begin{tikzcd}
			A \ar[r,"g"] \ar[d,"f"'] & C \ar[d,"v"]\\
			B \ar[r,"u"'] \ar[ru,phantom,"\leq"{description,sloped}] & D
	\end{tikzcd}\end{equation*}
expresses the property that $u \circ f \leq v \circ g$, which we may also write as $fu \leq gv$.

When $\cat{C}$ and $\cat{D}$ are categories, we denote by $[\cat{C}, \cat{D}]$ the category of functors from $\cat{C}$ to $\cat{D}$, with natural
transformations between them. A functor $\cat{C}^\op \to \cat{D}$ is sometimes called a \emph{$\cat{D}$-valued presheaf on $\cat{C}$}, and a functor $\cat{C} \to \cat{D}$ is sometimes called a \emph{$\cat{D}$-valued copresheaf on $\cat{C}$}. In case $\cat{D} = \cSet$, the adjective ``$\cSet$-valued'' is often omitted.

\paragraph{Oplax cocones and oplax colimits} Let $\cC$ be a small category. The category $[\cC^\op,\cPoset]$ is order-enriched: given two natural transformations $\alpha, \beta : K \rightrightarrows L$, we say that $\alpha \leq \beta$ if $\alpha_c(x) \leq \beta_c(x)$ for all $c \in \cC$ and all $x \in K(c)$. We recall some $2$-categorical terminology in this special case, see \cite[Sec.~6]{Borceux94T1} and \cite{Lack2010} for more about lax limits in general. Let $\sS \in [\cC^\op,\cPoset]$ and let $F \colon \cI \to [\cC^\op,\cPoset]$ be a diagram indexed by a poset $\cI$. An \emph{oplax cocone} $(F(i) \to \sS)_i$ is a family of morphisms $F(i) \to \sS$ such that $F(i) \to \sS \leq F(i) \to F(j) \to \sS$ for all $i \leq j$. Oplax cocones of this kind correspond to natural transformations $\tilde{F} \to \sS$ where $\tilde{F} \colon \cC^\op\to\cPoset$ is the \emph{oplax colimit} of $F$. Elements of $\tilde{F}(c)$ are pairs $(i \in \cI, x \in F(i)(c))$ and $(i,x) \leq (j,y)$ means that $i \leq j$ and $F(i \to j)(x) \leq y$.

Given a presheaf $F \colon \cC^\op \to \cSet$, we write $\int F$ for the \emph{category of elements} of $F$, equipped with a forgetful functor $\int F \to \cC$. Its objects are the pairs $(c \in \cC,x \in F(c))$ and the morphisms $(c,x) \to (d,y)$ are the morphisms $c \to d$ in $\cC$ such that $x = F(c\to d)(y)$. We will exclusively use this notation $\int F$ when $F$ is a $\cSet$-valued presheaf.

\paragraph{Inductive completion and ind-objects} A category is called \emph{filtered} if every finite diagram admits a cocone. A colimit is called filtered if it is indexed by a filtered category, and a limit is called cofiltered if it is indexed by the opposite of a filtered category. Any essentially small category $\cC$ has a universal cocompletion with respect to filtered colimits, denoted $\cInd(\cC)$, and called the \emph{inductive completion} of $\cC$ \cite[Cor.~2.1.9${}^\prime$]{MakkaiPare}; objects of $\cInd(\cC)$ are called \emph{ind-objects of $\cC$}.
Given a functor $F \colon \cC \to \cD$ whose codomain admits all filtered colimits, we write $\ovl{F} \colon \cInd(\cC) \to \cD$ for its left Kan extension along $\cC \into \cInd(\cC)$, which is the essentially unique extension of $F$ to $\cInd(\cC)$ preserving filtered colimits. We call $\ovl{F}$ the \emph{extension by continuity} of $F$. An object $X$ of a category is called \emph{$\omega$-presentable} if $\Hom(X,-)$ preserves filtered colimits. More generally, $X$ is called $\kappa$-presentable if $\Hom(X,-)$ preserves $\kappa$-filtered colimits (i.e., colimits indexed by categories in which every diagram of cardinality less than $\kappa$ admits a cocone). An important property is that the canonical embedding $\cC \hookrightarrow \cInd(\cC)$ sends every object of $\cC$ to an $\omega$-presentable object in $\cInd(\cC)$. The category of $\omega$-presentable objects of $\cInd(\cC)$ is the \emph{Cauchy completion} of the category $\cC$.

We briefly recall how $\cInd(\cC)$ may be realized. Given an essentially small category $\cC$, the Yoneda embedding associates to any object $n \in \cC$ the so-called \emph{representable presheaf} $\cC(-,n) \in [\cC^\op, \cSet]$. The inductive completion of $\cC$ may be realized as the full subcategory of $[\cC^\op,\cSet]$ on presheaves $X$ such that $\int X$ is filtered. An equivalent condition is that $X$ is a filtered colimit of representables, see \cite[Thm.~1.2.2]{MakkaiPare} or \cite[Thm.~8.3.3]{SGA4-1}.
For such a presheaf $X$ and $n \in \cC$, the set $X(n)$ is naturally isomorphic to the set of morphisms $n \to X$ in $\cInd(\cC)$; here and in what follows, we suppress notation for the canonical embedding $\cC \hookrightarrow \cInd(\cC)$. For instance, $\cInd(\cFinSet) \simeq \cSet$, and in this case the set of morphisms $n \to X$ is $X^n$. Even in the case of a general category $\cC$, when $X \in \cInd(\cC)$ and $n \in \cC$, we will sometimes use the notation $X^n$ for the set of morphisms $n \to X$.

\begin{convention}
	Unless noted otherwise, $\cC$ will always denote an essentially small category with pushouts. The existence of pushouts can be relaxed (Remarks~\ref{rmk:no-pushout}, \ref{rmk:theories-presheaf-type} and \ref{rmk:fw-colim-butterflies}), but most of our proofs will be made under this hypothesis.
\end{convention}

\section{Hyperdoctrines and interpolation}
\label{sec:hyperdoctrines}

Propositional logics are commonly studied using lattice-based algebraic structures, such as distributive lattices, Boolean algebras, and frames. Lawvere \cite{LawvereAdjoint,LawvereEqHyp} defined \emph{hyperdoctrines} to extend these algebraizations of propositional logics to the first-order case, making use of the insight that quantifiers can be modeled using adjoints. The aim of this section is to recall the definitions of coherent, intuitionistic and Boolean hyperdoctrines, while emphasizing our perspective that, in addition to adjunction, two \emph{interpolation} properties are used. We refer the reader to \cite[Ch.~5]{CoumansThesis} and \cite[Sec.~4.4 and 7]{MarquisReyes2011} for much more background than we can give here.

\paragraph{Definition of hyperdoctrines via interpolation}
The following interpolation property for a lax square makes sense in any category of ordered structures. It seems to have been first defined explicitly, in the context of Heyting algebras, in \cite{PittsAM:amaich}.

\begin{definition}\label{dfn:interp}
	Let
	\begin{equation}\label{ABCDinterp}\begin{tikzcd}
			A \ar[r,"g"] \ar[d,"f"'] & C \ar[d,"v"]\\
			B \ar[r,"u"'] \ar[ru,phantom,"\leq"{description,sloped}] & D
	\end{tikzcd}\end{equation}
	be a lax square of ordered structures.
	We say that the square \eqref{ABCDinterp} has the \emph{interpolation property} if for all $b \in B$ and all $c \in C$ such that $u(b) \leq v(c)$, there is an interpolant $a \in A$ verifying $b \leq f(a)$ and $g(a) \leq c$. When the orders are discrete, we say that the square has the \emph{amalgamation property}.
\end{definition}

Note that the interpolation property is not invariant under transposition of the lax square. We will often need this property for squares that are actually commutative, but in which the interpolation property only holds in one direction; see, e.g., the definition of morphism between coherent hyperdoctrines (Definition~\ref{dfn:morphism-c-hyp}). Considering a commutative square as a lax square then allows us to indicate the direction of the interpolation property.

\begin{remark}\label{rmk:interp-relations}
	Given the square \eqref{ABCDinterp}, we have two order relations from $B$ to $C$, known as the \emph{weakening relations} defined respectively by the span $(f,g)$ and by the cospan $(u, v)$ \cite[Sec.~2.2]{KurMosJun21}:
	\begin{align*}
			R_1 &\coloneqq \setst{ (b,c) \in B^\op \times C}{\exists a \in A \text{ such that } b \leq f(a) \text{ and } g(a) \leq c}\text{,} \\
			R_2 &\coloneqq \setst{(b,c) \in B^\op \times C}{u(b) \leq v(c)}\text{.}
	\end{align*}
	Note that $R_1 \subseteq R_2$ expresses exactly the lax commutativity of the square \eqref{ABCDinterp}, and that $R_2 \subseteq R_1$ is equivalent to the interpolation property for the square \eqref{ABCDinterp}. In other words, a lax square with the interpolation property is precisely one for which $R_1 = R_2$. Lax squares with the interpolation property are called \emph{exact squares} in \cite{KurMosJun21}.
\end{remark}

Our definition of coherent hyperdoctrine will use the following property of a homomorphism between distributive lattices, which is closely related to Frobenius reciprocity, as we will show in Proposition~\ref{prop:frobenius}.

\begin{definition}\label{dfn:frobenius}
	A homomorphism of distributive lattices $h \colon A \to B$ is \emph{Frobenius} if, for every $a \in A$, the following square has the interpolation property:
	\begin{equation}\label{eq:frobenius}
		\begin{tikzcd}
			A \ar[d,"h"'] \ar[r,"p_a"] & {\downarrow}a \ar[d,"\restr{h}{{\downarrow} a}"] \\
			B \ar[r,"p_{h(a)}"'] \ar[ru,phantom,"\leq"{sloped,description,pos=0.4}] & {\downarrow}h(a)
		\end{tikzcd}
	\end{equation}
\end{definition}

\begin{definition}\label{dfn:c-hyp}
	A \emph{coherent hyperdoctrine} is a functor $\hD \colon \cC \to \DL$ satisfying the following three axioms:
	\begin{enumerate}
		\item[\optionaldesc{Int1}{h-axiom:interpol-1}.] The image by $\hD$ of any pushout square in $\cC$ has the interpolation property.
		\item[\optionaldesc{Int2}{h-axiom:interpol-2}.] For any morphism $\sigma \colon n\to m$ in $\cC$, $\hD\sigma$ is Frobenius.
		\item[\optionaldesc{AdjLeft}{h-axiom:adjoint-left}.] For any morphism $\sigma \colon n\to m$ in $\cC$, $\hD\sigma$ has a left adjoint.
	\end{enumerate}
\end{definition}
We emphasize that, in our definition, a coherent hyperdoctrine over a base category $\cC$ is a \emph{covariant} functor from $\cC$ to the category of distributive lattices, while hyperdoctrines are usually presented as contravariant functors in the literature. The covariant way of phrasing the definition fits with our view, to be pursued later in this paper, that $\mathbf{C}$ is thought of as a category of ``small objects,'' rather than as a category of contexts. We will explain in detail the correspondence with the usual definitions after giving a few examples.

\paragraph{Some examples of hyperdoctrines}
We now give three examples of hyperdoctrines. We start with two main examples in the case $\cC = \cFinSet$, one syntactical and the other semantical. The first example shows how coherent hyperdoctrines give an algebraic counterpart to coherent first-order logic, which is defined in detail in, e.g., \cite[Appendix~E]{CoumansThesis}, but we do not need it in the rest of the paper.

\begin{example}\label{exa:logic-hyperdoctrine}
	To a coherent first-order theory $T$, we associate the functor $\hD_T \colon \cFinSet \to \cDL$ defined as follows. For a finite set $n$, $\hD_T(n)$ is the distributive lattice of coherent formulas, modulo equivalence in the theory $T$, whose free variables are taken in the set $n$. Alternatively, it is the lattice of coherent sentences modulo equivalence in the theory $T$ with $n$ constants added; we will not make any distinction between these two points of view. For a function $f \colon n \to m$, the image of $\varphi$ by $\hD_T(f)$ is obtained by substituting in the formula $\varphi$ each variable $x \in n$ by the variable $f(x)$, in a capture-avoiding way. More explicitly,
	\[ \hD_T(f)(\varphi(x_1,{\ldots},x_n)) = \varphi(f(x_1),{\ldots},f(x_n)) \text{.} \]
	This functor $\hD_T$ is a coherent hyperdoctrine. Indeed, in the papers \cite{LawvereAdjoint,LawvereEqHyp} introducing hyperdoctrines, Lawvere remarked that equality and existential quantifiers are given by the left adjoints of the morphisms $\hD(f)$ for appropriate choices of $f$, which naturally leads to the axiom \ref{h-axiom:adjoint-left}. The two interpolation axioms \ref{h-axiom:interpol-1} and \ref{h-axiom:interpol-2} then express how these adjoints must interact with other substitutions and conjunction. For more details on this construction, see \cite[Sec.~5.1.1]{CoumansThesis}.
\end{example}

Although we will not require this fact in the rest of the paper, it turns out that every coherent hyperdoctrine over $\cFinSet^\op$ is isomorphic to one of the form $\hD_T$ for some coherent first-order theory $T$.
From this perspective, a coherent first-order theory is a presentation of a coherent hyperdoctrine in the sense of multi-sorted universal algebra.

\begin{example}\label{exa:powerset-hyperdoc}
	For any set $X$, we define the \emph{hyperdoctrine of predicates on $X$}, $\powerset_X \colon \cFinSet \to \DL$, as the composite of the functor $X^{(-)} \colon \cFinSet \to \cSet^\op$ with the power set functor $\cSet^\op \to \cDL$. This is a coherent hyperdoctrine, which is even Boolean, see Definition~\ref{dfn:int-hyp} below.
\end{example}

A connection between the above two examples will be made when we speak about models in Section~\ref{sec:models}.

\begin{example}\label{exa:subobj-hyperdoc}
	An alternative but closely related categorical approach to coherent logic is provided by coherent categories. To each coherent category $\cE$ is associated a canonical coherent hyperdoctrine of subobjects $\Sub_\cE \colon \cE^\op \to \cDL$ sending each object to its lattice of subobjects. This construction is part of a $2$-adjunction between coherent categories and coherent hyperdoctrines, see \cite{CoumansThesis} for more details. In this paper, this example will only intervene again in Remark~\ref{rmk:link-conceptual-comp} where we explain the link between our statement of conceptual completeness and the one for pretoposes in \cite{MakkaiReyes}.
\end{example}

\paragraph{Comparison with the usual definition}
Let us highlight two differences between the standard definition of coherent hyperdoctrines and the one we gave here, and explain why they are equivalent.

First, a coherent hyperdoctrine is usually presented as a functor $\cC^\op \to \DL$ where $\cC$ is thought of as a category of contexts, the morphisms being substitutions. In this paper, though, we present hyperdoctrines as functors $\cC \to \DL$ where $\cC$ is thought of as a category of ``small objects.'' We find that this gives an easier intuition for what follows (cf. Section~\ref{sec:models}), even if the usual choice is better-behaved with regard to the connection with coherent categories and toposes. So the concept we define in Definition~\ref{dfn:c-hyp} above could be called a ``coherent $\cC^\op$-hyperdoctrine.'' We will often just refer to these objects as ``hyperdoctrines'' when the base category is fixed and the adjective ``coherent'' is clear from the context. Moreover, we note here that the assumption that the base category $\cC$ has pushouts is not essential, but simplifies the definition; in Remark~\ref{rmk:no-pushout}, we will indicate a definition that does not require this assumption.

Second, in the literature, the interpolation axioms \ref{h-axiom:interpol-1} and \ref{h-axiom:interpol-2} are more commonly expressed as the so-called Beck-Chevalley and Frobenius conditions, respectively, see below. We have chosen to phrase the definition so that the existence of a left adjoint is separate from the interpolation properties, as we find this leads to a more transparent and modular duality theory in the next section. On the other hand, the usual definition has the advantage that it shows that coherent hyperdoctrines form a variety of multisorted algebras.
We now explain in some detail why our Definition~\ref{dfn:c-hyp} is equivalent to the one existing in the literature, e.g., \cite[Def.~7]{Coumans12}.

Recall that a lax square of posets \eqref{ABCD} is said to satisfy the \emph{Beck--Chevalley condition} if $f$ and $v$ have left adjoints $\ladj{f}$ and $\ladj{v}$ such that the square \eqref{ABCD2} commutes.
	
\begin{minipage}{0.485\linewidth}
\begin{equation}\label{ABCD}\begin{tikzcd}
	A \ar[r,"g"] \ar[d,"f"'] & C \ar[d,"v"]\\
	B \ar[r,"u"'] \ar[ru,phantom,"\leq"{sloped,description}] & D
\end{tikzcd}\end{equation}
\end{minipage}\begin{minipage}{0.485\linewidth}
\begin{equation}\label{ABCD2}\begin{tikzcd}
	A \ar[r,"g"] & C \\
	B \ar[r,"u"'] \ar[u,"\ladj{f}"] & D \ar[u,"\ladj{v}"']
\end{tikzcd}\end{equation}
\end{minipage}

The following fact was remarked in \cite[p.~156]{PittsAM:amaich}; we give a proof using weakening relations.

\begin{proposition}\label{prop:interp-BC}
	Suppose given a lax square of posets as in \eqref{ABCD}, and suppose that $f$ and $v$ have left adjoints, $\ladj{f}$ and $\ladj{v}$, respectively. Then the lax square \eqref{ABCD} has the interpolation property if, and only if, the square \eqref{ABCD2} is commutative.
\end{proposition}

\begin{proof}
	We use the characterization of the interpolation property of Remark~\ref{rmk:interp-relations}. Note that the relation $R_1$ defined there is equal to the composite of relations
		\[ \setst{(b,a) \in B^\op\times A}{b \leq f(a)} \cdot \setst{(a,c) \in A^\op\times C}{g(a) \leq c} \text{,}\]
		and that, similarly, $R_2$ is equal to the composite
		\[ \setst{(b,d) \in B^\op\times D}{u(b) \leq d} \cdot \setst{(d,c) \in D^\op\times C}{d \leq v(c)} \text{.} \]
	Note also that the relations $\setst{(b,a) \in B^\op\times A}{b \leq f(a)}$ and $\setst{(d,c) \in D^\op\times C}{d \leq v(c)}$ are represented by the left adjoints $\ladj{f}$ and $\ladj{v}$, respectively, and the composite relations $R_1$ and $R_2$ are thus also represented by the composites $\ladj{f} g$ and $u \ladj{v}$, respectively. The stated equivalence now follows.
\end{proof}

In particular, looking back at Definition~\ref{dfn:c-hyp}, Proposition~\ref{prop:interp-BC} implies that if the axiom \ref{h-axiom:adjoint-left} holds, then the axiom \ref{h-axiom:interpol-1} is equivalent to the statement that $\hD$ sends pushout squares to Beck-Chevalley squares, which is how this axiom is usually stated. Viewed like this, the corresponding logical intuition is that existential quantification and equality interact well with substitutions.

We now show the connection between the axiom \ref{h-axiom:interpol-2} and Frobenius reciprocity. Recall that if a homomorphism $f \colon A\to B$ of distributive lattices has a left adjoint $\ladj{f} \colon B\to A$, then the adjoint pair is said to satisfy \emph{Frobenius reciprocity} if, for all $a \in A$ and all $b \in B$,
\[ \ladj{f}(b \land f(a)) = \ladj{f}(b) \land a\text{.} \]

\begin{proposition}\label{prop:frobenius}
	Let $f \colon A \to B$ be a homomorphism of distributive lattices which has a left adjoint, $\ladj{f}$. The homomorphism $f$ is Frobenius in the sense of Definition~\ref{dfn:frobenius} if, and only if, the adjoint pair satisfies Frobenius reciprocity.
\end{proposition}

\begin{proof}
	By Proposition~\ref{prop:interp-BC} applied to \eqref{eq:frobenius}.
\end{proof}

We note in passing that the two axioms \ref{h-axiom:interpol-1} and \ref{h-axiom:interpol-2} may alternatively be combined into an equivalent single axiom which says that $\hD$ sends any pushout square in $\cC$ to a square \eqref{ABCD} with the following \emph{strong interpolation property}: for any $a \in A$, $b \in B$ and $c \in C$ such that $u(b\land f(a)) \leq v(c)$, there is an interpolant $z \in A$ such that $b \leq f(z)$ and $g(z\land a) \leq c$.

The coherent hyperdoctrines over a fixed base category are the objects of a category, under the following notion of morphism.

\begin{definition}\label{dfn:morphism-c-hyp}
	A \emph{morphism} between coherent hyperdoctrines is a natural transformation $\tau \colon \hD_1 \to \hD_2$ such that, for every morphism $\sigma \colon n \to m$ in $\cC$, the naturality square, viewed as a lax square
	\[\begin{tikzcd}
		\hD_1(n) \ar[d, "\hD_1\sigma"'] \ar[r, "\tau_n"] & \hD_2(n) \ar[d, "\hD_2\sigma"] \\
		\hD_1(m) \ar[r, "\tau_m"'] \ar[ru,phantom,"\leq"{description,sloped}] & \hD_2(m),
	\end{tikzcd}\]
	has the interpolation property.
\end{definition}

By Proposition~\ref{prop:interp-BC}, the additional requirement on naturality squares in Definition~\ref{dfn:morphism-c-hyp} is equivalent to saying that, for every morphism $\sigma \colon n \to m$, we have $\ladj{(\hD_1\sigma)} \tau_n = \tau_m \ladj{(\hD_2\sigma)}$. The logical intuition here is that the natural transformation ``preserves existential quantification and equality.'' This is the usual definition of morphism between coherent hyperdoctrines in the literature, and it is also the one given by seeing coherent hyperdoctrines as a variety of multisorted algebras. For a more general notion of morphism that allows for a change of base category, see for example \cite[Def.~7]{Coumans12}.

\paragraph{Intuitionistic hyperdoctrines}
Later in this paper, we will also be concerned with intuitionistic first-order logic, where we add the Heyting implication and universal quantification to coherent logic. Accordingly, intuitionistic hyperdoctrines take values in the category $\cHeyt$ of Heyting algebras instead of distributive lattices, and, in order to account for universal quantifiers, right adjoints must exist in addition to left adjoints.

\begin{definition}\label{dfn:int-hyp}
	An \emph{intuitionistic hyperdoctrine} is a functor $\hD \colon \cC \to \cHeyt$ satisfying the following axioms.
	\begin{enumerate}
		\item[\optionaldesc{Int1}{h-axiom:interpol-1-int}.] The image by $\hD$ of any pushout square in $\cC$ has the interpolation property.
		\item[\optionaldesc{AdjLeft}{h-axiom:adjoint-left-int}.] $\hD\sigma$ has a left adjoint for any map $\sigma \colon n\to m$ in $\cC$.
		\item[\optionaldesc{AdjRight}{h-axiom:adjoint-right-dfn-int}.] $\hD\sigma$ has a right adjoint for any map $\sigma \colon n\to m$ in $\cC$.
	\end{enumerate}

	A \emph{Boolean hyperdoctrine} is an intuitionistic hyperdoctrine such that $\hD(n)$ is a Boolean algebra for every $n$. In this case, \refeq{h-axiom:adjoint-left-int} holds iff \refeq{h-axiom:adjoint-right-dfn-int} holds.
\end{definition}

Note that the axiom \ref{h-axiom:interpol-2} from Definition~\ref{dfn:c-hyp} has disappeared in Definition~\ref{dfn:int-hyp}. The reason is that \ref{h-axiom:interpol-2} is subsumed by the condition that $\hD$ takes value in $\cHeyt$, as we will briefly explain now.
Suppose that $\sigma \colon n \to m$ is a morphism in $\cC$ and that $\hD(n)$ and $\hD(m)$ are Heyting algebras. Let $s \coloneqq \hD(\sigma)$ and let $a \in \hD(n)$. 
Note that a reformulation of Proposition~\ref{prop:interp-BC} with right adjoints instead of left adjoints says that a lax square of posets as in \eqref{ABCD} has the interpolation property if, and only if, $f \circ \radj{g} = \radj{u} \circ v$, where $\radj{g}$ and $\radj{u}$ denote the right adjoints of $g$ and $u$, respectively.
Therefore, 
the square in the axiom \ref{h-axiom:interpol-2} has interpolation if, and only if, for all $b \leq a$, we have $s(a \to b) = s(a) \to s(b)$. So the axiom \ref{h-axiom:interpol-2} is equivalent to the fact that $\hD(f)$ is a morphism of Heyting algebras.

Intuitionistic hyperdoctrines form a non-full subcategory of coherent hyperdoctrines, with morphisms defined as follows.

\begin{definition}\label{dfn:int-hyp-mor}
	A \emph{morphism of intuitionistic hyperdoctrines} is a natural transformation $\tau \colon \hD_1 \to \hD_2$ such that, for every morphism $\sigma \colon n \to m$ in $\cC$, the associated naturality square has the interpolation property in both directions, as depicted below.
	
	\begin{minipage}{0.485\linewidth}
		\[\begin{tikzcd}
			\hD_1(n) \ar[d, "\hD_1\sigma"'] \ar[r, "\tau_n"] & \hD_2(n) \ar[d, "\hD_2\sigma"] \\
			\hD_1(m) \ar[r, "\tau_m"'] \ar[ru,phantom,"\leq"{description,sloped}] & \hD_2(m)\text{,}
		\end{tikzcd}\]
	\end{minipage}\begin{minipage}{0.485\linewidth}
		\[\begin{tikzcd}
			\hD_1(n) \ar[d, "\hD_1\sigma"'] \ar[r, "\tau_n"] & \hD_2(n) \ar[d, "\hD_2\sigma"] \\
			\hD_1(m) \ar[r, "\tau_m"'] \ar[ru,phantom,"\geq"{description,sloped}] & \hD_2(m)\text{,}
		\end{tikzcd}\]
	\end{minipage}
\end{definition}

Compared to Definition~\ref{dfn:morphism-c-hyp}, the additional interpolation property on naturality squares in Definition~\ref{dfn:int-hyp-mor} is equivalent to requiring in addition that $\radj{(\hD_1\sigma)} \tau_n = \tau_m \radj{(\hD_2\sigma)}$, so that the natural transformation $\tau$ also preserves universal quantification, in addition to existential quantification. Note that, in the case of Boolean hyperdoctrines, it suffices to assume that one of the two squares in Definition~\ref{dfn:int-hyp-mor} has the interpolation property, as the other then follows.

\section{Duality between hyperdoctrines and polyadic spaces}
\label{sec:duality}

The aim of this section is to apply Priestley duality for distributive lattices to obtain a duality for coherent hyperdoctrines.
To this end, we express the algebraic axioms introduced in the previous section in topological terms, in order to obtain duality theorems for coherent and intuitionistic hyperdoctrines, Theorem~\ref{thm:duality} and Corollary~\ref{cor:duality-int}.
In the process, we provide two general duality-theoretic propositions: interpolation is self-dual (Proposition~\ref{prop:dual-interpolation}), and existence of adjoints is dual to openness (Proposition~\ref{prop:left-adjoint-dually}).

\paragraph{Priestley duality for $\cat{DL}$-valued presheaves}
Note that, for any category $\cat{C}$, we
immediately obtain from Priestley duality a dual equivalence
between
$\cat{DL}$-valued copresheaves and $\cat{Priestley}$-valued
presheaves on $\cat{C}$,
\begin{equation} \label{eq:priestley-lifted}
[\cat{C},\Spec] \colon [\cat{C}, \cat{DL}] \leftrightarrows [\cat{C}^\op, \cat{Priestley}]^\op \colon [\cat{C}^\op,\ClU]\text{,}
\end{equation}
by applying the Priestley duality functors pointwise.
Concretely, $[\cat{C},\Spec]$ sends
a functor $D \colon \cat{C} \to \cat{DL}$ to the functor $\Spec \circ D^\op \colon
\cat{C}^\op \to \cat{Priestley}$, and a natural transformation $\alpha \colon D
\to E$ to the natural transformation $\beta \colon \Spec \circ E^\op \to \Spec \circ
D^\op$ which is defined for any object $n$ of $\cat{C}$ by $\beta_n \coloneqq
(\alpha_n)_*$. The definition of $[\cat{C}^\op, \ClU]$ is analogous.

\begin{example}
	Let $\hD_T \colon \cFinSet \to \cDL$ be a coherent hyperdoctrine which is presented by some coherent first-order theory $T$, in the sense of Example~\ref{exa:logic-hyperdoctrine}, and denote by $\sS \colon \cFinSet^\op \to \cPriestley$ the dual presheaf $\Spec \circ \hD_T$. Then $\sS(n)$ is the \emph{space of $n$-types} of the coherent theory $T$; that is, its points may be identified with $n$-pointed models of $T$ modulo equivalence in coherent logic. The action of $\sS$ on morphisms is given by composition: if $f \colon n\to m$ is a function and $x \in \sS(m)$ is the $m$-type of a tuple $(x_1,\dots,x_m)$ in some model, then $\sS(f)(x)$ is the $n$-type of $(x_{f(1)},\dots,x_{f(n)})$ in the same model. The claims made in this example will be fully justified by Gödel's completeness theorem (Theorem~\ref{thm:Godel}). 
\end{example}

In the rest of this section, we will characterize the $\cPriestley$-valued presheaves that occur as pointwise duals of a coherent hyperdoctrine. To obtain this dual characterization, we first need to transfer the axioms given in Section~\ref{sec:hyperdoctrines} to the topological side.

\paragraph{The dual of interpolation}
The property for a lax square to have interpolation turns out to be self-dual, as we prove now.

\begin{proposition}\label{prop:dual-interpolation}
	Consider the following lax square \eqref{dia:latticesquare} in $\cat{DL}$ and its dual lax square \eqref{dia:spacesquare} in $\cat{Priestley}$.\\
	\begin{minipage}{0.485\linewidth}
	\begin{equation}\label{dia:latticesquare}
		\begin{tikzcd}
		A \ar[r,"g"] \ar[d,"f"'] & C \ar[d,"v"]\\
		B \ar[ru,phantom,"\leq"{description,sloped}] \ar[r,"u"'] & D
		\end{tikzcd}
	\end{equation}
	\end{minipage}\begin{minipage}{0.485\linewidth}
	\begin{equation}\label{dia:spacesquare}
		\begin{tikzcd}
		D_* \ar[r,"v_*"] \ar[d,"u_*"'] & C_* \ar[d,"g_*"] \\
		B_* \ar[ru,phantom,"\leq"{description,sloped}] \ar[r,"f_*"'] & A_*
		\end{tikzcd}
	\end{equation}
	\end{minipage}\\
	The square \eqref{dia:latticesquare} has the interpolation property if, and only if, the square \eqref{dia:spacesquare} has the interpolation property.
\end{proposition}

\begin{proof}
	First suppose that \eqref{dia:latticesquare} has the interpolation property; we
	show that \eqref{dia:spacesquare} does, as well. Let $y \in B_*$ and $z \in
	C_*$ and suppose that $f_*(y) \leq g_*(z)$, i.e., $y\circ f \leq z\circ g$.
	We claim that the following filter $F$ and ideal $I$ of $D$ are disjoint:
	\[ F \coloneqq {\uparrow} \setst{ u(b)}{b \in B,\ y(b) = \top }\text{,} \quad I \coloneqq
	{\downarrow} \setst{ v(c)}{c \in C,\ z(c) = \bot}\text{.}\]
	Indeed, towards a contradiction,
	suppose that $F \cap I \neq \emptyset$. Pick $b \in B$ with $y(b) = \top$ and $c \in C$ with $z(c) = \bot$, such that $u(b) \leq v(c)$. By the interpolation property of
	\eqref{dia:latticesquare}, pick $a \in A$ such that $b \leq f(a)$ and $g(a) \leq
	c$. Then
	\[ \top = y(b) \leq y(f(a)) \leq z(g(a)) \leq z(c) = \bot \text{,} \]
	which is the desired contradiction. Therefore, by
	the prime filter-ideal theorem, pick $x \in D_*$ such that
	$\restr{x}{I} = \bot$ and $\restr{x}{F} = \top$. It follows from the definitions that $y \leq u_*(x)$ and $v_*(x) \leq z$, as required.
	
	For the converse, we make a very similar argument after exchanging the algebraic and topological sides and by replacing the prime filter-ideal theorem by Lemma~\ref{lem:Priest-sep}. Suppose that \eqref{dia:spacesquare} has the interpolation property; we show that \eqref{dia:latticesquare} does, as well. Let $b \in B$ and $c \in C$ such that $u(b) \leq v(c)$. This means that $u_*^{-1}(\hat{b}) \subseteq v_*^{-1}(\hat{c})$. Consider the closed up-set ${\uparrow} f_*[\hat{b}]$ and the closed down-set ${\downarrow} g_*[C_* \setminus \hat{c}]$. We claim they are disjoint. If not there is some $x \in \hat{b}$ and $y \in C_* \setminus \hat{c}$ such that $f_*(x) \leq g_*(y)$. By the interpolation property of \eqref{dia:spacesquare}, there is some $z \in D_*$ such that $x \leq u_*(z)$ and $v_*(z) \leq y$. But the first inequality implies that $z \in u_*^{-1}(\hat{b})$ and the second one implies that $z \not\in v_*^{-1}(\hat{c})$, which contradicts $u_*^{-1}(\hat{b}) \subseteq v_*^{-1}(\hat{c})$. Therefore, by Lemma~\ref{lem:Priest-sep}, there is some clopen up-set $\hat{a} \subseteq A_*$ containing $f_*[\hat{b}]$ and disjoint from $g_*[C_* \setminus \hat{c}]$. This means that $b \leq f(a)$ and $g(a) \leq c$.
\end{proof}

\begin{remark}
	We outline an alternative proof of Proposition~\ref{prop:dual-interpolation}, which uses Priestley duality for relations. We omit the details since they are not needed in what follows, see e.g. \cite{KurMosJun21} or \cite[Sec.~4.5]{GG22}.
	
	In Remark~\ref{rmk:interp-relations}, we saw that the interpolation property of \eqref{dia:latticesquare} can be interpreted as the equality of the two relations $R_1$ and $R_2$, defined there. Dually, consider the following two relations $S_1, S_2 \subseteq B_* \times C_*$:
	\begin{align*}
	S_1 &\coloneqq \setst{ (y,z) \in B_* \times C_*}{f_*(y) \leq g_*(z)}\text{,} \\
	S_2 &\coloneqq \setst{ (y,z) \in B_* \times C_*}{\exists x \in A_* \text{ such that } y \leq u_*(x) \text{ and } v_*(x) \leq z}\text{.}
	\end{align*}
	Note also here that $S_2 \subseteq S_1$ always holds using commutativity of the square \eqref{dia:spacesquare}, and that the interpolation property for square \eqref{dia:spacesquare} says precisely that $S_1 \subseteq S_2$. Thus, the statement of Proposition~\ref{prop:dual-interpolation} is equivalent to: $R_2 \subseteq R_1$ if, and only if, $S_1 \subseteq S_2$. Now, one may show that, for each $i \in \{1,2\}$, $S_i$ is the ``dual'' relation of $R_i$, meaning that
	\[ S_i = \bigcap \setst[\big]{\hat{b}^c \times \hat{c}}{(b,c) \in R_i}.\]
	This is because $R_i$ and $S_i$ are built by composing dual relations, and that duality for these relations respects composition. Combining this observation with the fact that the operation of taking the dual is an anti-isomorphism between a certain poset of ``filtering'' relations from $B$ to $C$ and a poset of ``closed'' relations from $B_*$ to $C_*$, one may then conclude that $R_2 \subseteq R_1$ iff $S_1 \subseteq S_2$, as required.
\end{remark}

\begin{remark}\label{rmk:conseq-autodual-interp}
	The self-duality of Proposition~\ref{prop:dual-interpolation} can be used to show that pushouts of Heyting algebras have the interpolation property, by translating this statement into a statement in the category of Esakia spaces (see \cite[Thm.~B]{PittsAM:amaich} for a constructive version of this proof). We also mention that \cite[Lemma~5.4]{Gehrke_2018} can be seen as a special case of Proposition~\ref{prop:dual-interpolation}.
\end{remark}

In light of Proposition~\ref{prop:dual-interpolation}, we can now in particular compute dual properties for the axioms \ref{h-axiom:interpol-1} and \ref{h-axiom:interpol-2} of Definition~\ref{dfn:c-hyp}. For \ref{h-axiom:interpol-1}, this will be straight-forward. For \ref{h-axiom:interpol-2}, we need the following consequence of Proposition~\ref{prop:dual-interpolation} and Priestley duality. 

\begin{proposition}\label{prop:frobenius-bounded}
	Let $h \colon A \to B$ be a homomorphism between distributive lattices and write $f$ for the dual function $B_* \to A_*$. The homomorphism $h$ is Frobenius if, and only if, the function $f$ is bounded.
\end{proposition}

\begin{proof}
	By Proposition~\ref{prop:dual-interpolation} and Priestley duality, the square \eqref{eq:frobenius} has interpolation for every $a \in A$ if, and only if, the following condition holds:
	\begin{equation}\label{eq:bdd-proof-cond}
		\forall a \in A, y \in B_*, x \in \hat{a} \colon f(y) \leq x \implies \exists z \in f^{-1}(\hat{a}) \colon y \leq z \text{ and } f(z) \leq x \text{.}
	\end{equation}
	We show that \eqref{eq:bdd-proof-cond} is equivalent to $f$ being bounded.

	First, assume $f$ is bounded, let $a \in A$ and suppose that $f(y) \leq x$ for some $y \in B_*$, $x \in \hat{a}$. Then, since $f[{\uparrow}y]$ is an up-set by the boundedness of $f$, it contains $x$. Pick $z \in {\uparrow} y$ such that $f(z) = x$. Then in particular $z \in f^{-1}(\hat{a})$, as required.

	Conversely, assume that $f$ satisfies \eqref{eq:bdd-proof-cond}.
	Note that, to show that $f$ is bounded, it suffices to prove that $f[{\uparrow} y]$ is an up-set for every $y \in B_*$. Let $x \in A_*$ be such that $f(y) \leq x$. Define
	\[ C \coloneqq {\uparrow} y \cap f^{-1}({\downarrow} x).\]
	Note that $C$ is a closed subset of $B_*$, and property \eqref{eq:bdd-proof-cond} gives that $C \cap f^{-1}(\hat{a})$ is non-empty for every $a \in A$ such that $x \in \hat{a}$. Thus, by compactness, the set
	\[C \cap\bigcap\setst{f^{-1}(\hat{a})}{a \in A, x \in \hat{a}} = C \cap f^{-1}({\uparrow} x)\]
	is non-empty; pick a point $z$ in it. Then $z \geq y$ and $f(z) = x$, so $x \in f[{\uparrow} y]$, as required.
\end{proof}

Let $\hD \colon \cat{C} \to \cat{DL}$ be a functor and write $\sS \colon \cat{C}^\op \to \cat{Priestley}$ for the dual functor $\Spec \circ \hD$. By Proposition~\ref{prop:dual-interpolation}, we get that $\hD$ satisfies the axiom \ref{h-axiom:interpol-1} if, and only if, $\sS$ satisfies the following:
\begin{enumerate}
	\item[\optionaldesc{$\text{Int1}_*$}{p-axiom:interpol-1}.] $\sS$ sends pullback squares to squares with interpolation.
\end{enumerate}
In light of Proposition~\ref{prop:frobenius-bounded}, $\hD$ satisfies \ref{h-axiom:interpol-2} if, and only if, $\sS$ satisfies:
\begin{enumerate}
	\item[\optionaldesc{$\text{Int2}_*$}{p-axiom:interpol-2}.] For any morphism $\sigma \colon n \to m$ in $\cat{C}$, $\sS \sigma$ is bounded.
\end{enumerate}

\begin{remark}\label{rmk:p-axioms-1+2}
	We can combine the two axioms \ref{p-axiom:interpol-1} and \ref{p-axiom:interpol-2} into a strong interpolation property saying that if the square below is the image of a pushout square in $\cC$ by $\sS$, then for any $b \in B$ and $c \in C$ such that $u(b) \leq v(c)$, there is some $a \in A$ such that $b \leq f(a)$ and $g(a) = c$.
	\[\begin{tikzcd}
		A \ar[r,"g"] \ar[d,"f"'] & C \ar[d,"v"]\\
		B \ar[r,"u"] & D
	\end{tikzcd}\]
\end{remark}

\paragraph{The dual of quantification}
As explained in Section~\ref{sec:hyperdoctrines}, first-order quantifiers correspond to the adjoints of substitution maps. In particular, the definition of a coherent hyperdoctrine requires that the image of any map has a left adjoint. We now prove a well-known proposition identifying the dual meaning of left adjoints. This proposition may also be derived, for example, from the corresponding fact about coherent frames, see for example~\cite[Sec~IX.7]{MM1992}. 

\begin{proposition}\label{prop:left-adjoint-dually}
	Let $f \colon X \to Y$ be a morphism between Priestley spaces, write $A \coloneqq \ClU(X), B \coloneqq \ClU(Y)$, and $h \colon B \to A$ for the dual homomorphism.
	The following are equivalent:
	\begin{enumerate}
		\item The homomorphism $h$ has a left adjoint. \label{cond-lad:1}
		\item For any open up-set $U$ of $X$, the set ${\uparrow} f[U]$ is open in $X$. \label{cond-lad:2}
	\end{enumerate}
	Moreover, if these properties are verified, and $g \colon A \to B$ denotes the left adjoint of $h$, then for every $U \in A$, $g(U) = {\uparrow} f[U]$.
\end{proposition}

\begin{proof}
	Since every open up-set $U \subseteq X$ is a union of clopen up-sets and since $U \mapsto {\uparrow}f[U]$ preserves unions, we can restrict $U$ to clopen up-sets in condition \ref{cond-lad:2}. Let $g \colon A \to B$ be the partial left adjoint of $h$, which means that $g(U)$ is, if it exists, the unique element of $B$ verifying $g(U) \leq V \iff U \leq h(V)$ for all $V \in A$. Then $h$ has a left adjoint if and only if $g$ is defined everywhere. By definition, $g(U)$ is the smallest clopen up-set of $Y$ such that $U \subseteq f^{-1}(g(U))$, i.e., $f[U] \subseteq g(U)$. Since $U$ is closed, $f[U]$ is closed by compactness. Hence, ${\uparrow} f[U]$ is the intersection of all the clopen up-sets containing $f[U]$ by Lemma~\ref{lem:adh-spec}. We conclude that if $g(U)$ exists, then it is ${\uparrow}f[U]$, and this is equivalent to ${\uparrow}f[U]$ being open, since it is always a closed up-set.
\end{proof}

Thus, if $\hD \colon \cat{C} \to \cat{DL}$ is a functor and $\sS \colon \cat{C}^\op \to \cat{Priestley}$ is the dual functor $\Spec \circ \hD$, then $\hD$ satisfies the axiom \ref{h-axiom:adjoint-left} if, and only if, $\sS$ satisfies the following:
\begin{enumerate}
	\item[\optionaldesc{$\text{AdjLeft}_*$}{p-axiom:adjoint-left-dual}.] For any map $\sigma \colon n\to m$ in $\cC$ and any open up-set $U$ of $\sS(m)$, ${\uparrow}\sS\sigma[U]$ is open in $\sS(n)$.
\end{enumerate}

\begin{remark}\label{rmk:Heyt-dual}
	Note that, applying order duality to Proposition~\ref{prop:left-adjoint-dually}, we obtain that a morphism $h \colon B\to A$ of distributive lattices has a \emph{right} adjoint if, and only if, its dual $f \colon A_* \to B_*$ satisfies that ${\downarrow}f[U]$ is open for every open down-set $U \subseteq A_*$.
	We briefly explain how this fact can be used to prove that a distributive lattice $A$
	is a Heyting algebra if and only if $A_*$ is an Esakia space, i.e., for every open $U \subseteq A_*$, the set ${\downarrow} U$
	is open. Recall that a distributive lattice $A$ is a Heyting algebra if $p_a \colon A
	\to {\downarrow}a$ has a right adjoint for all $a \in A$.
	By the order-dual of Proposition~\ref{prop:left-adjoint-dually}, this means that for every clopen up-sets $U \subseteq A_*$, the
	inclusion $i \colon U \hookrightarrow A_*$ satisfies that ${\downarrow}i[V]$ is open
	for any clopen down-set $V \subseteq U$. In other words, ${\downarrow}[U \cap V]$ is open
	for any $U,V \subseteq A_*$ with $U$ a clopen up-set and $V$ a clopen down-set.
	Since the sets of the form $U \cap V$ form a basis of opens, we obtain that
	${\downarrow} W$ is open for every open $W \subseteq A_*$. As for morphisms,
	Proposition~\ref{prop:frobenius-bounded} shows that the duals of homomorphisms of Heyting algebras
	are the continuous order-preserving
	functions between Esakia spaces which are bounded.
\end{remark}

\paragraph{Semi-open morphisms}
In light of the results of this section, a functor $\sS \colon \cat{C}^\op \to \cat{Priestley}$ is the pointwise dual of a coherent $\cC^\op$-hyperdoctrine, if, and only if, $\sS$ satisfies the axioms \ref{p-axiom:interpol-1}, \ref{p-axiom:interpol-2} and \ref{p-axiom:adjoint-left-dual}. We now show how to combine the two axioms \ref{p-axiom:interpol-2} and \ref{p-axiom:adjoint-left-dual} into one
natural property of (lower semi-)\emph{openness}, as follows.

Recall from Section~\ref{sec:prelim} that the collection of open up-sets of a compact ordered space $X$ is a subtopology on $X$, and that we denote by $X^{\uparrow}$ the underlying set of $X$ equipped with this topology.
Let $f \colon X \to Y$ be a function between compact ordered spaces. 
We will say that $f$ is \emph{lower semi-open} if it is an open map when viewed as a function from $X^{\uparrow}$ to $Y^{\uparrow}$, that is, if for every open up-set $U \subseteq X$, $f[U]$ is an open up-set. This terminology corresponds with the standard terminology that $f$ is called \emph{lower semi-continuous} if it is continuous as as map from $X^{\uparrow}$ to $Y^{\uparrow}$.
The following proposition shows how lower semi-openness is related to the axioms \ref{h-axiom:interpol-2} and \ref{h-axiom:adjoint-left}, a well known fact in the context of frames.

\begin{proposition}\label{prop:stab-open-bounded}
	Let $f \colon X \to Y$ be a morphism of Priestley spaces with dual lattice homomorphism $h \colon B \to A$. Then $f$ is lower semi-open if, and only if, $h$ is Frobenius and has a left adjoint.
\end{proposition}

\begin{proof}
	In light of Propositions~\ref{prop:frobenius-bounded} and \ref{prop:left-adjoint-dually}, it suffices to show that $f$ is lower semi-open if, and only if, $f$ is bounded and satisfies \ref{cond-lad:2} in Proposition~\ref{prop:left-adjoint-dually}. Note that the condition is clearly sufficient. Now suppose $f$ is lower semi-open. Then clearly $f$ satisfies Prop.~\ref{prop:left-adjoint-dually}\ref{cond-lad:2}. We show that $f$ is bounded. Let $x \in X$ and $y \in Y$ such that $f(x) \leq y$. Write $X_y	\coloneqq f^{-1}(\{y\})$; we need to show that $X_y \cap {\uparrow} x$ is
	non-empty. Since $f$ is continuous, $X_y$ is closed. Moreover, for every $U \in	A$ such that $x \in U$, $f[U]$ contains $f(x)$, and thus, since it
	is an up-set, it contains $y$. Therefore, for every $U \in A$ such that $x \in U$, $X_y \cap U$ is non-empty. By compactness, the set
	\[ X_y \cap {\uparrow} x = X_y \cap \bigcap \setst{U}{U \in A, x \in U}\]
	is non-empty, as required.
\end{proof}

\paragraph{Definition of polyadic spaces} We are now ready to give the general definition of an (open) polyadic compact ordered space, of which an (open) polyadic Priestley space is a special case. The reason for this generalization from Priestley spaces to compact ordered spaces is twofold: first, everything we do in this paper works exactly in the same way for compact ordered spaces; second, there is a good algebraic dual interpretation of polyadic compact ordered spaces, see Remark~\ref{rmk:duality-kord} below.

\begin{remark}\label{rmk:duality-kord}
	It is possible to extend all the propositions of this section from Priestley spaces to compact ordered spaces by using the duality for $\cKOrd$ given in \cite{AbbadiniReggio_2020,AbbadiniThesis2021}, after replacing the strong Priestley separation property (Lemma~\ref{lem:Priest-sep}) by the Kat{\v e}tov-Tong theorem, a generalization of the Tietze extension theorem. The resulting notion is a variation on the continuous syntactic categories defined in \cite{AlbHar16}. We refer to \cite[Sec.~1.3]{JMThesis} for more on this.
\end{remark}

\begin{definition}\label{dfn:polyadic-priestley}
	A functor $\sS \colon \cat{C}^\op \to \cKOrd$ is a \emph{polyadic compact ordered space} if it satisfies the following two axioms.
	\begin{enumerate}
		\item[\optionaldesc{$\text{Int1}_*$}{p-axiom:interpol-1-dfn}.] $\sS$ sends pushout squares to squares with interpolation.
		\item[\optionaldesc{$\text{Int2}_*$}{p-axiom:interpol-2-dfn}.] For any morphism $\sigma\colon n\to m$ in $\cC$, $\sS\sigma$ is bounded.
	\end{enumerate}
	We say that $\sS$ is \emph{open} if it moreover satisfies the following axiom.
	\begin{enumerate}
		\item[\optionaldesc{Open}{p-axiom:open}.] For any map $\sigma \colon n \to m$ in $\cC$, $\sS \sigma$ is lower semi-open.
	\end{enumerate}
	
	A \emph{morphism} of polyadic compact ordered spaces is a natural transformation $\sS_1 \to \sS_2$ whose naturality squares below have the interpolation property.
	\[\begin{tikzcd}
		\sS_1(m) \ar[d] \ar[r] & \sS_1(n) \ar[d] \\
		\sS_2(m) \ar[r] \ar[ru,phantom,"\leq"{description,sloped}] & \sS_2(n)
	\end{tikzcd}\]
	
	A \emph{polyadic Priestley space} is a polyadic compact ordered space taking values in the full subcategory $\cPriestley$ of $\cKOrd$.
\end{definition}

When we want to emphasize the base category $\cC$, we will speak of $\cC$-adic spaces or polyadic spaces over $\cC$.

Note that a functor $\cC^\op \to \cPriestley$ which satisfies \ref{p-axiom:open} is already a polyadic Priestley space as soon as it satisfies \ref{p-axiom:interpol-1-dfn}, since the axiom \ref{p-axiom:interpol-2-dfn} then automatically holds by Proposition~\ref{prop:stab-open-bounded}. Thus, an \emph{open polyadic Priestley space} may be defined more succinctly as a functor $\sS \colon \cC^\op \to \cPriestley$ that sends every map to a lower semi-open map, and pushout squares to squares with the interpolation property. We obtain our duality theorem by combining the previous results of this section.

\begin{theorem}\label{thm:duality}
	The category of coherent hyperdoctrines over $\cC$ is dually equivalent to the category of open polyadic Priestley spaces over $\cC$.
\end{theorem}

\begin{proof}
	Recalling the dual equivalence \eqref{eq:priestley-lifted}, the category of coherent hyperdoctrines is dually equivalent to its image under the functor $[\cC^\op, \Spec]$. Propositions~\ref{prop:dual-interpolation},~\ref{prop:frobenius-bounded},~and~\ref{prop:left-adjoint-dually} together show that on the objects, this image consists of those functors $\sS \colon \cC^\op \to \cat{Priestley}$ such that \ref{p-axiom:interpol-1}, \ref{p-axiom:interpol-2}, and \ref{p-axiom:adjoint-left-dual} hold. By Proposition~\ref{prop:stab-open-bounded}, the conjunction of \ref{p-axiom:interpol-2} and \ref{p-axiom:adjoint-left-dual} is equivalent to \ref{p-axiom:open}. On the morphisms, Proposition~\ref{prop:dual-interpolation} shows that morphisms of open polyadic Priestley spaces are the duals of morphisms of coherent hyperdoctrines.
\end{proof}

\begin{remark}\label{rmk:no-pushout}
	If $\cC$ doesn't have pushouts, then Definition~\ref{dfn:polyadic-priestley} generalizes by replacing pushouts in the axiom \ref{p-axiom:interpol-1-dfn} by ``formal pushouts'' computed in the free completion $[\cC,\cSet]^\op$ of $\cC$. More explicitly, this means that for each span $(f \colon A \to B, g \colon A \to C)$ in $\cC$, the square below has interpolation, where $X$ ranges over all the cocones over the span.
	\begin{equation}\label{dia:generalized-interpol}\begin{tikzcd}
			\colim_X \sS(X) \ar[r] \ar[d] & \sS(C) \ar[d,"\sS(g)"]\\
			\sS(B) \ar[r,"\sS(f)"'] \ar[ru,phantom,"\leq"{description,sloped}] & \sS(A)
	\end{tikzcd}\end{equation}
	While the existence of the colimit requires $\cC$ to be small, we remark that, in the more general case of a not necessarily small base category $\cC$, the colimit can be avoided by formulating the interpolation property more directly, parametric in the cocones over the span. The Priestley duals of the open polyadic Priestley spaces in this generalized sense form an algebraic variety under the condition that for every span in $\cC$, there is a finite set $C$ of cocones such that any cocone factors through at least one of the cocones in $C$. In the terminology explained below in Remark~\ref{rmk:theories-presheaf-type}, this means that $\cC$ has fjw pushouts. Under this condition, $\colim_X \sS(X)$ can be replaced in \eqref{dia:generalized-interpol} by a coproduct ranging over this finite set of cocones. Proposition~\ref{prop:dual-interpolation} implies that the interpolation property of the resulting lax square is equivalent to the fact that for each clopen up-set $U \subseteq \sS(B)$, we have
	\[ \sS(g)^{-1}(\sS(f)[U]) = \bigcup_{(u,v) \in C} \sS(v)[\sS(u)^{-1}(U)] \text{.} \]
	This is equivalent to an equational condition in the language of distributive lattices, which makes it possible to generalize the fact, implicit in the literature, that coherent hyperdoctrines form a multi-sorted algebraic variety. More details on this, and its generalization to the category dual to compact ordered spaces, will be given in a forthcoming paper by the second-named author.
\end{remark}

As remarked in Section~\ref{sec:prelim}, Boolean algebras are dual to Priestley spaces having discrete orders, which we call \emph{Boolean spaces}. In this special case, the interpolation axiom \ref{p-axiom:interpol-1} can be formulated as the following \emph{amalgamation} condition, see also \cite[Dfn.~2.18(5)]{Ben-Yaacov03}, \cite[Sec.~3.1]{Reg21}, \cite[below Dfn.~4.13]{Kamsma_2022} and \cite{MarRamics}.

\begin{definition}\label{dfn:polyadic-set}
	A \emph{polyadic set} is a functor $\sS
	\colon \cC^\op \to \cSet$ with \emph{amalgamation}, by which we mean that every span in $\int \sS$ admits a cocone. A \emph{polyadic Boolean space} is a
	functor $\sS \colon \cC^\op \to \cBoolSp$ with amalgamation.
\end{definition}

A detailed unraveling of the amalgamation condition can be found, for example, in \cite[Sec~3.1]{Reg21}. 
A polyadic set can be understood as the type space functor of a multi-sorted first-order theory in the logic $\mathcal{L}_{\infty,\omega}$, where disjunctions and conjunctions can be taken over sets of any size.

The following corollary is now immediate from Theorem~\ref{thm:duality} and the remarks above.

\begin{corollary}\label{cor:boolean-duality}
	The category of Boolean hyperdoctrines over $\cC$ is dually equivalent to the category of open polyadic Boolean spaces over $\cC$.
\end{corollary}

\begin{example}
	We give an example of a polyadic Priestley space $\sS\colon \cFinSet^\op \to \cPriestley$  which does not have the amalgamation property, meaning that the composite of $\sS$ with the forgetful functor $\cPriestley \to \cBoolSp$ is not a polyadic Boolean space. In logical terms, this shows in particular that it is not straight-forward to Booleanize a coherent theory, since the polyadic Boolean space corresponding to the Booleanization can not always be obtained by simply composing with the forgetful functor.
	
	Our example here is a slight modification of \cite[Example~4.15]{Kamsma_2022}. We start by considering a poset $\set{x,y,z}$ whose order is generated by $x > y$ and $x > z$. As a set, we define $\sS(n) \coloneqq \set{x,y}^n \cup \set{x,z}^n$. The order $(x_1,\dots,x_n) \leq (y_1,\dots,y_n)$ of $\sS(n)$ is defined by the following two conditions:
	\begin{enumerate}
		\item $x_i \leq y_i$ for all $1 \leq i \leq n$.
		\item $x_i = x_j \implies y_i = y_j$ for all $1 \leq i,j \leq n$.
	\end{enumerate}
	If $f \colon n\to m$ is a map in $\cFinSet$, we define $\sS(f)(x_1,\dots,x_m) = (x_{f(1)},\dots,x_{f(n)})$. We leave to the reader the task of checking that this defines a functor satisfying the interpolation properties \ref{p-axiom:interpol-1} and \ref{p-axiom:interpol-2}. On the other hand, it doesn't satisfy the amalgamation property: the elements $y$ and $z$ of $\sS(1)$ are both sent to the unique point of $\sS(0)$ but the unique possible amalgam $(y,z)$ is not in $\sS(2)$. We note that $\sS$ takes its values in finite posets, so that it is automatically a polyadic Esakia space as we will define below in Definition~\ref{dfn:morphism-int}.
	
	This example is actually the polyadic Priestley space associated to the following coherent first-order theory with only three models. The signature has two base unary symbols $P(x)$ and $Q(x)$, and the axioms are:
	\begin{enumerate}
		\item $P(x) \land Q(y) \vdash Q(x) \lor P(y)$
		\item $\vdash \exists x \colon P(x) \land Q(x)$
		\item $P(x) \land Q(x) \land P(y) \land Q(y) \vdash x=y$
		\item $P(x) \land P(y) \vdash x=y \lor Q(x) \lor Q(y)$
		\item $Q(x) \land Q(y) \vdash x=y \lor P(x) \lor P(y)$
	\end{enumerate}
	Only the first axiom is essential, the other ones are included only to simplify the description of $\sS$.
\end{example}

\paragraph{The intuitionistic case}
We now show how Esakia duality can also be extended to a duality for intuitionistic hyperdoctrines; this is again a consequence of Propositions~\ref{prop:dual-interpolation} and \ref{prop:left-adjoint-dually}. The only important notion that we will use in the rest of the paper will be that of \emph{intuitionistic morphism} of polyadic compact ordered spaces.

\begin{definition}\label{dfn:morphism-int}
	A \emph{polyadic Esakia space} is an open polyadic Priestley space $\sS$ that takes values in $\cEsakia$, such that moreover ${\downarrow}\sS(\sigma)[U]$ is open for all $\sigma \colon n \to m$ in $\cC$ and all open down-set $U \subseteq \sS(m)$.

	A morphism $\tau \colon \sS_1 \to \sS_2$ of polyadic compact ordered spaces is \emph{intuitionistic} if, for every object $n$ of $\cC$, $\tau_n$ is bounded, and, in addition to the naturality squares on the left below, the naturality squares on the right below have the interpolation property.
	
	\begin{minipage}{0.485\linewidth}
		\[\begin{tikzcd}
			\sS_1(m) \ar[d] \ar[r] & \sS_1(n) \ar[d] \\
			\sS_2(m) \ar[r] \ar[ru,phantom,"\leq"{description,sloped}] & \sS_2(n)
		\end{tikzcd}\]
	\end{minipage}\begin{minipage}{0.485\linewidth}
		\[\begin{tikzcd}
			\sS_1(m) \ar[d] \ar[r] & \sS_1(n) \ar[d] \\
			\sS_2(m) \ar[r] \ar[ru,phantom,"\geq"{description,sloped}] & \sS_2(n)
		\end{tikzcd}\]
	\end{minipage}
\end{definition}

\begin{corollary}\label{cor:duality-int}
	The category of intuitionistic hyperdoctrines over $\cC$ is dually equivalent to the category of polyadic Esakia spaces over $\cC$ with intuitionistic morphisms between them.
\end{corollary}

Note that a polyadic Esakia space is exactly a polyadic Priestley space $\sS$ respecting three openness conditions, corresponding respectively to Heyting implication, existential quantification, and universal quantification:
\begin{enumerate}
	\item ${\downarrow}W$ is open for all $W \subseteq \sS(n)$ open.
	\item ${\uparrow} \sS(\sigma)[U]$ is open for all $\sigma\colon n\to m$ in $\cC$ and $U \subseteq \sS(m)$ open up-set.
	\item ${\downarrow} \sS(\sigma)[U]$ is open for all $\sigma\colon n\to m$ in $\cC$ and $U \subseteq \sS(m)$ open down-set.
\end{enumerate}

\begin{definition}
	An \emph{intuitionistic polyadic compact ordered space} is a polyadic compact ordered space satisfying the three openness conditions given above.
\end{definition}

\paragraph{Compactification and interpolation}
We conclude this section by explaining how the ordered Stone--Čech compactification interacts with interpolation. In particular, we will see that openness properties can be understood as interpolation properties involving the Stone--Čech compactification. This will be useful in Section~\ref{sec:models} where we define models of polyadic spaces.

The ordered Stone-Čech compactification $\beta \colon \cPoset \to \cKOrd$ is the left adjoint of the forgetful functor $\cKOrd \to \cPoset$. It sends a poset $X$ to $\Up(X)_*$, where $\Up(X)$ denotes the lattice of up-sets of $X$. The functor of up-sets $X \mapsto \Up(X)$ realizes a dual equivalence, called \emph{discrete duality}, between $\cPoset$ and the category of completely distributive complete lattices having enough join-irreducible elements \cite{Raney1952}, also known as \emph{perfect} distributive lattices. The morphisms between these lattices of up-sets preserve arbitrary infima and suprema, so in particular they have left and right adjoints. Moreover, these lattices are Heyting algebras. As a consequence, $\beta \colon \cPoset \to \cKOrd$ takes values in the category of Esakia spaces and morphisms $f \colon A \to B$ such that ${\uparrow}f[U]$ (resp. ${\downarrow}f[U]$) is open for every open up-set (resp. open down-set) $U \subseteq A$, by Proposition~\ref{prop:left-adjoint-dually}.

\begin{proposition}
	The ordered Stone--Čech compactification preserves boundedness of maps and the interpolation property of lax commutative squares.
\end{proposition}

\begin{proof}
	The proofs of Propositions~\ref{prop:dual-interpolation} and \ref{prop:frobenius-bounded} also work in the setting of discrete duality. The interpolation property is thus preserved by both the discrete duality functor $X \mapsto \Up(X)$ and the Priestley duality functor $A \mapsto A_*$. As their composite, the Stone--Čech compactification functor also preserves interpolation. The preservation of boundedness is similar.
\end{proof}

\begin{definition}
	A \emph{polyadic poset} is a functor $\sS \colon \cC^\op \to \cPoset$ sending each morphism to a bounded map and each pushout square to a square with interpolation.
\end{definition}

\begin{corollary}\label{cor:compactification}
	For any polyadic poset $\sS \colon \cC^\op \to \cPoset$, the functor $\beta \circ \sS \colon \cC^\op \to \cEsakia$ is a polyadic Esakia space. If $\sS$ is in particular a polyadic set, then $\beta \circ \sS$ is an open polyadic Boolean space.
\end{corollary}

In the next section, we will consider squares of the form \eqref{eq:beta-XYAB} with the interpolation property.

\begin{minipage}{0.485\linewidth}
\begin{equation}\label{eq:XYAB}\begin{tikzcd}
	X \ar[d,"u"'] \ar[r,"f"] & Y \ar[d,"v"]\\
	A \ar[r,"g"'] \ar[ru,phantom,"\leq"{description,sloped}] & B
\end{tikzcd}\end{equation}
\end{minipage}\begin{minipage}{0.485\linewidth}
\begin{equation}\label{eq:beta-XYAB}\begin{tikzcd}
	\beta X \ar[d,"\ovl{u}"'] \ar[r,"\beta f"] & \beta Y \ar[d,"\ovl{v}"]\\
	A \ar[r,"g"'] \ar[ru,phantom,"\leq"{description,sloped}] & B
\end{tikzcd}\end{equation}
\end{minipage}

Since morphisms of compact ordered spaces $\beta X \to A$ correspond to order-preserving maps $X \to A$, a natural question is to express the interpolation property of square \eqref{eq:beta-XYAB} in terms of square \eqref{eq:XYAB}. This leads to the following definition.

\begin{definition}\label{dfn:weak-interp}
	Let \eqref{eq:XYAB} be a lax commutative square of posets where moreover $A$ and $B$ are compact ordered spaces and $g$ is continuous. Suppose that ${\uparrow}g[U]$ is open for every open up-set $U \subseteq A$. In this situation, we say that the square \eqref{eq:XYAB} has the \emph{weak interpolation property} if for any open up-set $U \subseteq A$ and $y \in Y$ such that $v(y) \in {\uparrow}g[U]$, there exists $x \in X$ such that $f(x) \leq y$ and $u(x) \in U$.
\end{definition}

\begin{proposition}\label{prop:weak-interp}
	In the situation of Definition~\ref{dfn:weak-interp} (in particular, ${\uparrow}g[U]$ is open for every open up-set $U \subseteq A$), if $A$ and $B$ are Priestley spaces, then the square \eqref{eq:beta-XYAB} has the interpolation property if, and only if, the square \eqref{eq:XYAB} has the weak interpolation property.
\end{proposition}

\begin{proof}
	Note that the weak interpolation property of \eqref{eq:XYAB} says that for every open up-set $U \subseteq A$, we have ${\uparrow}f[u^{-1}(U)] = v^{-1}({\uparrow}g[U])$. We can restrict this condition to clopen up-sets, since they form a basis of opens. Applying Proposition~\ref{prop:dual-interpolation} to the square \eqref{eq:beta-XYAB}, we recover exactly the same statement, so the two are equivalent.
\end{proof}

This shows that even though $\beta$ is left adjoint to the forgetful functor $\cKOrd \to \cPoset$, composition with $\beta$ is not left adjoint to the forgetful functor from polyadic compact ordered spaces to polyadic posets: given a polyadic poset $\sP$ and a polyadic space $\sQ$, morphisms of polyadic posets $\sP \to \sQ$ are more restrictive than morphisms of polyadic spaces $\beta \circ \sP \to \sQ$.

We now show how two openness properties can be reformulated as interpolation properties. The first proposition shows that the hypothesis in Proposition~\ref{prop:weak-interp} that ${\uparrow}g[U]$ is open for every open up-set $U \subseteq A$ is necessary. For proofs in the more general setting of compact ordered spaces, see \cite[Sec.~1.3]{JMThesis}.

\begin{proposition}\label{prop:openness-as-interp}
	Let $f \colon A \to B$ be a morphism of Priestley spaces. Then the square \eqref{eq:beta-A-B} below has interpolation if and only if ${\uparrow}f[U]$ is open for every open up-set $U \subseteq A$.
	\begin{equation}\label{eq:beta-A-B}\begin{tikzcd}
		\beta A \ar[d,->>] \ar[r,"\beta f"] & \beta B \ar[d,->>]\\
		A \ar[r,"f"'] \ar[ru,phantom,"\leq"{description,sloped}] & B
	\end{tikzcd}\end{equation}
\end{proposition}

\begin{proof}
	If ${\uparrow}f[U]$ is open for every open up-set $U \subseteq A$, Proposition~\ref{prop:weak-interp} shows that the square has interpolation. Reciprocally, suppose that the square has interpolation. Proposition~\ref{prop:dual-interpolation} shows that for every clopen up-set $U \subseteq A$ and all up-sets $V \subseteq B$, if $U \subseteq f^{-1}(V)$, then there is some clopen up-set $V' \subseteq B$ such that $U \subseteq f^{-1}(V')$ and $V' \subseteq V$. Applying that to $V = {\uparrow}f[U]$, we obtain some clopen up-set $V' \subseteq B$ such that $U \subseteq f^{-1}(V')$ and $V' \subseteq {\uparrow}f[U]$. This implies that $V' = {\uparrow}f[U]$.
\end{proof}

\begin{proposition}\label{prop:esakia-as-interp}
	Let $X$ be a Priestley space. Then $X$ is an Esakia space (i.e., ${\downarrow}U$ is open for every open $U \subseteq X$) if and only if $\beta X \twoheadrightarrow X$ is bounded.
\end{proposition}

\begin{proof}
	In Remark~\ref{rmk:Heyt-dual}, we explained that ${\downarrow}U$ is open for every open $U \subseteq X$ if and only if for any inclusion $i \colon V \into X$ of clopen up-set, ${\downarrow}i[U]$ is open for every open down-set $U \subseteq V$. By (the dual of) Proposition~\ref{prop:openness-as-interp}, this is equivalent to the fact that each square as below has the interpolation property, where $V \subseteq X$ is a clopen up-set.
	
	\begin{equation}\label{eq:beta-V-X}\begin{tikzcd}
		\beta V \ar[d,->>] \ar[r,hook] & \beta X \ar[d,->>]\\
		V \ar[r,hook] \ar[ru,phantom,"\geq"{description,sloped}] & X
	\end{tikzcd}\end{equation}
	
	Since interpolation is auto-dual (Proposition~\ref{prop:dual-interpolation}), this means that the Priestley dual of $\beta X \twoheadrightarrow X$ is Frobenius, and Proposition~\ref{prop:frobenius-bounded} shows that this is equivalent to the boundedness of $\beta X \twoheadrightarrow X$. (The proof of Proposition~\ref{prop:frobenius-bounded} actually shows that directly without going through the dual side.)
\end{proof}

Even though we will not use this point of view there, we note the following consequence of Propositions~\ref{prop:openness-as-interp} and \ref{prop:esakia-as-interp}.

\begin{corollary}\label{cor:caract-open-interp}
	A polyadic Priestley space $\sS$ is:
	\begin{enumerate}
		\item open if and only if the natural transformation $\beta \circ \sS \to \sS$ is a morphism of polyadic compact ordered spaces;
		\item intuitionistic if and only if the natural transformation $\beta \circ \sS \to \sS$ is an intuitionistic morphism of polyadic compact ordered spaces.
	\end{enumerate}
\end{corollary}

\begin{remark}
	Propositions~\ref{prop:weak-interp}, \ref{prop:openness-as-interp}, \ref{prop:esakia-as-interp} and Corollary~\ref{cor:caract-open-interp} above are also true for compact ordered spaces instead of Priestley spaces. These more general statements can be given direct but less transparent proofs, or can be seen in the same way as above by using the duality for compact ordered spaces of \cite{AbbadiniThesis2021, AbbadiniReggio_2020} instead of Priestley duality.
\end{remark}

\begin{remark}\label{rem:canext}
	The construction $\sS \mapsto \beta \circ \sS$ is dual to the \emph{canonical extension of coherent hyperdoctrines} studied in \cite{CoumansThesis,Coumans12}. Indeed, if $\hD \colon \cC \to \cDL$ is a coherent hyperdoctrine and $\sS$ its dual open polyadic Priestley space, then the canonical extension $\hD^\delta$ as defined in \cite[Prop.~9]{Coumans12} has $\beta \circ \sS$ as its dual polyadic space. This is a consequence of the fact that when $L$ is a distributive lattice with dual Priestley space $X$, then $L^\delta$ is isomorphic to $\Up(X)$, so the Priestley dual space of $L^\delta$ is $\beta X$. Note that, combining these observations with Proposition~\ref{prop:frobenius-bounded} and Proposition~\ref{prop:openness-as-interp}, we see that a distributive lattice $L$ is a Heyting algebra if, and only if, the embedding $L \to L^\delta$ is a Frobenius map. One direction is well-known, see for example \cite[Prop.~2]{Gehrke2014}.
\end{remark}

\section{Models of hyperdoctrines and polyadic spaces}
\label{sec:models}

In this section, we explain how to view models and types in the context of hyperdoctrines and polyadic spaces.

\paragraph{Models of hyperdoctrines}
In classical model theory, a model of a theory $T$ is a set $X$ equipped with some extra structure that allows to interpret the predicate symbols occurring in the theory, in such a way that the formulas of the theory $T$ are validated. More specifically, given a model $X$, every formula $\phi$ with $n$ free variables is interpreted as a subset $\sem{\phi}$ of $X^n$. In categorical terms,
if $\hD_T$ is the hyperdoctrine corresponding to $T$, then such a model can be viewed as a natural transformation $\sem{-} \colon \hD_T \to \powerset_X$, where $\powerset_X(n) = \powerset(X^n)$ as in Example~\ref{exa:powerset-hyperdoc}:
\[\begin{tikzcd}
	\cFinSet \ar[r,"X^{-}"] \ar[rd,"\hD_T"',""{name=A,above,near end}] & \cSet^\op \ar[d, "\powerset"]\\
	 & \cDL \ .
	 \ar[from=A,to=1-2,Rightarrow]
\end{tikzcd}\]
This natural transformation is actually a morphism of coherent hyperdoctrines, because it respects the logical connectives of conjunction, disjunction, existential quantification and equality. Such a morphism is called a ``model of $T$ in the hyperdoctrine $\powerset$'' \cite[Def.~5.1.3]{CoumansThesis}; we will simply call it a ``model of $\hD_T$''. In the general setting where we replace the base category $\cFinSet$ by a category $\cC$ with finite colimits, a morphism from a hyperdoctrine $\hD \colon \cC \to \cDL$ to the hyperdoctrine $\powerset$ consists of a functor $X \colon \cC \to \cSet^\op$ which sends finite colimits in $\cC$ to finite limits in $\cSet$ and a morphism of hyperdoctrines $\cD \Rightarrow \powerset \circ X$. The functors $X$ appearing in this definition are ind-objects of $\cC$, which motivates our definition of ``model of $\hD$'', Definition~\ref{dfn:model-alg} below.

\begin{definition}
	Let $X$ be an ind-object of $\cC$. The \emph{hyperdoctrine of predicates on $X$}, $\powerset_X \colon \cC \to \cBA$, is defined as the composite of $X$ seen as a functor $\cC \to \cSet^\op$ with the powerset functor $\cSet^\op \to \cBA$.
\end{definition}

Note that any ind-object $X$ of $\cC$, which we will often view as a functor $\cC^\op\to\cSet$ in this section, is a polyadic set over $\cC$, because $\int X$ is filtered, so in particular every span admits a cocone.
As a consequence of Corollary~\ref{cor:compactification}, the functor $\powerset_X$ defined above is always a Boolean hyperdoctrine.

\begin{definition}\label{dfn:model-alg}
	Let $\hD \colon \cC \to \cDL$ be a coherent hyperdoctrine and let $X \in \cInd(\cC)$. A \emph{$\hD$-structure} on $X$ is a morphism of coherent hyperdoctrines $\hD \to \powerset_X$. A \emph{model of $\hD$} is an ind-object of $\cC$ equipped with a $\hD$-structure.
\end{definition}

Let us give a bit of intuition about how hyperdoctrines are ``coherent theories of ind-objects of $\cC$.''
Let $\sem{-} \colon \hD \to \powerset_X$ be a model of $\hD$.
For any object $n \in \cC$, we write $X^n$ for the set of morphisms $n \to X$, and we will call the elements of $X^n$ the \emph{$n$-points} or \emph{$n$-tuples} of $X$.
Let $f \colon n\to m$ be an arrow in $\cC$, let $\varphi \in \hD(m)$ and suppose $p \colon n \to X$ is some $n$-tuple of $X$. Write $\exists_f \varphi \in \hD(n)$ for the image of $\varphi$ by the left adjoint of $\hD(f)$. Then $p \in \sem{\exists_f \varphi}$ if and only if there exists some commutative diagram as below with $q \in \sem{\varphi}$, according to the definition of $\powerset_X$.

\[\begin{tikzcd}
	m \ar[r,"q",dashed] & X\\
	n \ar[u,"f"] \ar[ur,"p"'] &
\end{tikzcd}\]

In the case of $\cC = \cFinSet$, we point out two special cases: if $f$ is surjective, this condition for $p \in \sem{\exists_f \varphi}$ gives the semantical interpretation of equality; if $f$ is injective, it gives the semantical interpretation of existential quantification. Similarly, conjunctions have to be interpreted as intersections and disjunctions as unions. It is possible to allow only quantification along arrows of a subcategory of $\cC$, for instance only injections to model first-order logic without equality, but we will not deal with this variation.

Suppose that we are given a set of base symbols, each one equipped with an object of $\cC$ that we call its arity. This generates freely a coherent hyperdoctrine $\hD$ in the sense of multisorted universal algebra, the generators in $\hD(n)$ being the symbols of arity $n$. A model of this hyperdoctrine is an ind-object $X$ equipped with an interpretation of each symbol of arity $n$ as a subset of $X^n$. Each expression built from these symbols using conjunction, disjunction, and the $\exists_f$ is then interpreted as another subset of some $X^m$ according to the rules explained above. This explains how, in the same way that coherent $\cFinSet^\op$-hyperdoctrines algebraize mono-sorted coherent first-order logic, coherent $\cC^\op$-hyperdoctrines algebraize some coherent first-order logic whose objects of discourse are the ind-objects of $\cC$.

\begin{remark}\label{rmk:theories-presheaf-type}
	We briefly explain a connection with coherent theories of presheaf type \cite{BekePresheafType}. By a \emph{finite jointly weak} (fjw) \emph{colimit} of a diagram $D$ we mean a finite set $C$ of cocones such that every cocone under $D$ factors through one of the cocones in $C$; this is called an \emph{fc colimit} in \cite{BekePresheafType}. In particular, by an fjw \emph{finite} colimit we mean an fjw colimit of a finite diagram. It is proved in \cite[Thm.~2.1]{BekePresheafType} that, for an arbitrary category $\cC$, the presheaf topos $[\cC,\cSet]$ is coherent if, and only if, $\cC$ has fjw finite colimits. For such a category $\cC$, denote by $T$ the coherent theory classified by $[\cC,\cSet]$. See \cite{BekePresheafType} for a description of $T$: it has one sort for each object of $\cC$ and its set-valued models are the ind-objects of $\cC$ seen as functors $\cC^\op \to \cSet$. Then the extensions of $T$ obtained by adding new symbols and axioms, but no new sort, are classified by the coherent hyperdoctrines over $\cC^\op$, in the sense of the more general definition of Remark~\ref{rmk:no-pushout}. In the extension associated to a hyperdoctrine $\hD\colon \cC \to \cDL$, the formulas with one free variable of sort $c \in \cC$ modulo equivalence correspond to the elements of $\hD(c)$.
\end{remark}

\begin{example}
	For instance, when $\cC = \cFinSetInj$ is the category of finite sets and injections, the variables in our formulas are interpreted as \emph{distinct} elements of the carrier set. Concretely, this implies that the semantics of an existentially quantified statement $\exists x \colon \varphi(x,\ovl{y})$ is that there exists some element $x$ distinct from the $\ovl{y}$ and such that $\varphi(x,\ovl{y})$ holds. In a classical Boolean setting, the expressive power of this modified logic is strictly the same as usual first-order logic, so that Boolean $\cFinSetInj^\op$-hyperdoctrines are equivalent to Boolean $\cFinSet^\op$-hyperdoctrines. In a more general coherent setting, coherent $\cFinSetInj^\op$-hyperdoctrines are equivalent to coherent $\cFinSet^\op$-hyperdoctrines in which equality has a complement. This is a consequence of the fact that $[\cFinSetInj,\cSet]$ is the classifying topos of \emph{decidable} objects, i.e., those whose equality has a complement \cite[Proposition~D~3.2.7]{Elephant}.
	
	As another example, when $\cC$ is the category of finite linear orders, open polyadic Priestley spaces over $\cC$ classify the extensions of the theory of linear orders by new symbols and axioms, as used by the second author in \cite{MarRamics}.
\end{example}

\paragraph{Models of polyadic spaces}
If we apply Priestley duality directly to Definition~\ref{dfn:model-alg}, we obtain the following dual notion, generalized to open polyadic compact ordered spaces. Recall that $\beta X$ denotes the Stone-Čech compactification of the set $X$, obtained as the Stone dual of $\powerset(X)$. Note that, for any $X \in \cInd(\cC)$, the functor $n \mapsto \beta(X^n)$ is an open $\cC$-adic Boolean space (Corollary~\ref{cor:compactification}).

\begin{definition}\label{dfn:model-topo-1}
	Let $\sS \colon \cC^\op \to \cKOrd$ be an open polyadic compact ordered space and let $X \in \cInd(\cC)$. An \emph{$\sS$-structure} on $X$ is a morphism of polyadic compact ordered spaces $\beta \circ X \to \sS$. A \emph{model} of $\sS$ is an ind-object of $\cC$ equipped with some $\sS$-structure.
\end{definition}

We now reformulate this definition in more concrete terms as a property of the restrictions of the continuous maps $\beta(X^n) \to \sS(n)$ to functions $X^n \to \sS(n)$, using Proposition~\ref{prop:weak-interp}.

\begin{proposition}\label{prop:weak-interp-struct}
	For any ind-object $X$, there is a natural bijection between $\sS$-structures on $X$ and natural transformations $\tau \colon X \to \sS$ such that, for every $f \colon n \to m$, the naturality square below has the weak interpolation property.
	\begin{equation}\label{eq:nat-wip}
		\begin{tikzcd}
			X^m \ar[d,"\tau_m"'] \ar[r,"X(f)"] & X^n \ar[d,"\tau_n"]\\
			\sS(m) \ar[r,"\sS(f)"'] \ar[ru,phantom,"\leq"{description,sloped}] & \sS(n)
		\end{tikzcd}
	\end{equation}
\end{proposition}

Proposition~\ref{prop:weak-interp-struct} can be used to show the intuition behind the notion of $\sS$-structure: the natural transformation $\tau$ sends an $n$-tuple of $X$ to its ``$n$-type'' in $\sS(n)$. In the special case where $\sS$ is an open polyadic Priestley space, the open up-sets in the definition of the weak interpolation property can be replaced by clopen up-sets. This property then exactly corresponds to the semantical intuition described above: if $f \colon n\to m$ is a morphism in $\cC$ and if an $n$-tuple $y \in X^n$ satisfies a formula of the form $\exists_f \varphi$, represented here as the direct image under $\sS(f)$ of a clopen up-set, then the $n$-tuple can be extended along $f$ to an $m$-tuple $x$ satisfying $\varphi$.

The notion of $\omega$-saturated model is obtained by strengthening the equivalent definition of $\sS$-structure given in Proposition~\ref{prop:weak-interp-struct}.

\begin{definition}\label{dfn:omega-sat-model}
	Let $\sS \colon \cC^\op \to \cKOrd$ be a polyadic compact ordered space and let $X \in \cInd(\cC)$. An \emph{$\omega$-saturated $\sS$-structure} on $X$ is a natural transformation $\tau \colon X \to \sS$ such that each naturality square as in \eqref{eq:nat-wip} has the interpolation property. An \emph{$\omega$-saturated model} of $\sS$ is an ind-object of $\cC$ equipped with some $\omega$-saturated $\sS$-structure.
\end{definition}

\begin{remark}
	We explain briefly why our definition of $\omega$-saturated model is equivalent to the usual one in the Boolean case and when $\cC = \cFinSet$. Let $\alpha \colon X \to \sS$ be a model. In usual model theory, $X$ is called $\omega$-saturated if for any $n$-tuple $x$ of $X$, each complete $k$-type over $x$ is realized in $X$. A direct reformulation is that the naturality square below has the amalgamation property where $i \colon n \to n+k$ is the canonical injection.
	\[\begin{tikzcd}
		X^{n+k} \ar[d,"{X^i}"'] \ar[r,"\alpha_{n+k}"] & \sS(n+k) \ar[d,"\sS(i)"]\\
		X^n \ar[r,"\alpha_n"] & \sS(n)
	\end{tikzcd}\]
	Moreover, if $\alpha \colon X \to \sS$ is a model, then the naturality squares associated to surjections automatically have the amalgamation property: in Proposition~\ref{prop:weak-interp}, if $f$ is bounded and injective, then the square~\eqref{eq:XYAB} has interpolation if, and only if, the square~\eqref{eq:beta-XYAB} has interpolation.
\end{remark}

\begin{remark}\label{rmk:open-for-structures}
	Since the reformulation of Proposition~\ref{prop:weak-interp-struct} is not possible if the polyadic space $\sS$ is not open, we are not sure what the correct notion of model is for polyadic compact ordered spaces that are not open, which is why we restricted Definition~\ref{dfn:model-topo-1} to open ones. On the other hand, the notion of $\omega$-saturated model makes sense also for non-open polyadic compact ordered spaces. If $\sS$ is open, then any $\omega$-saturated model $X \to \sS$ yields a morphism of polyadic compact ordered spaces $\beta \circ X \to \sS$. If $\sS$ is not open, this is not the case anymore.
\end{remark}

\paragraph{Type spaces over ind-objects} We have indicated above that, for $n \in \cC$, the space $\sS(n)$ can be thought of as a space of $n$-types. We now show how to generalize this to a definition of a space of $X$-types over an arbitrary ind-object $X$.

First, given any $\cSet$-valued presheaf $P \colon \cC^\op \to \cSet$, we write $\ovl{P} \colon \cInd(\cC)^\op \to \cSet$ for the unique extension of $P$ to a $\cSet$-valued presheaf on $\cInd(\cC)$ that sends filtered colimits to cofiltered limits; we call $\ovl{P}$ the \emph{extension by continuity} of $P$. This is coherent with the definition of ``extension by continuity'' that we gave at the end of Section~\ref{sec:prelim}, where we identify functors $\cC^\op \to \cSet$ with functors $\cC \to \cSet^\op$. 

Note that, for any $X \in \cInd(\cC)$, $\ovl{P}(X) \cong \Hom_{[\cC^{\op},\cSet]}(X,P)$, using the Yoneda lemma and the fact that the subcategory $\cInd(\cC) \subseteq [\cC^\op,\cSet]$ is closed under filtered colimits. Now, since the forgetful functor $\cKOrd \to \cSet$ creates limits, if $\sS$ is a $\cKOrd$-valued presheaf on $\cC$, its extension by continuity $\ovl{\sS}$ is again a $\cKOrd$-valued presheaf. For any $X \in \cInd(\cC)$, we may identify the points of $\ovl{\sS}(X)$ with the natural transformations $X \to \sS$, and we call $\ovl{\sS}(X)$ the \emph{space of $X$-types of $\sS$}.

The main result of the next section, Proposition~\ref{prop:main}, will show in particular that the extension by continuity of a polyadic compact ordered space is again a polyadic compact ordered space.

We see in particular that the natural transformations $X \to \sS$ have a natural order. This allows us to speak of lax commutative diagrams involving these arrows. To give an example, we reformulate the fact that a natural transformation $X \to \sS$ from an ind-object to a polyadic compact ordered space is an $\omega$-saturated model: this property says that each diagram as on the left below can be completed as on the right, where $n,m \in \cC$.

\begin{minipage}{0.49\linewidth}
	\[\begin{tikzcd}
		X \ar[r] & \sS\\
		n \ar[u] \ar[r] & m \ar[u] \ar[lu,phantom,"\geq"{sloped,description}]
	\end{tikzcd}\]
\end{minipage}\begin{minipage}{0.49\linewidth}
	\[\begin{tikzcd}
		X \ar[r] & \sS\\
		n \ar[u] \ar[r] & m \ar[u,""{left,name=UL}] \ar[lu]
		\ar[from=UL,to=1-1,phantom,"\geq"{description,sloped,pos=0.3}]
	\end{tikzcd}\]
\end{minipage}
Recall that the objects of $\cC$ are exactly the $\omega$-presentable objects of $\cInd(\cC)$, under the assumption that $\cC$ is Cauchy-complete (which it is when $\cC$ has pushouts). Similarly, the notion of $\kappa$-saturated model can be formulated by making $n$ and $m$ range over the $\kappa$-presentable ind-objects instead of over $\cC$. 

If we replace $X$ by a polyadic compact ordered space $\sP$, the exact same property of diagram-completion as above gives the definition of morphisms of polyadic compact ordered spaces.

\begin{definition}\label{dfn:catego-models}
	A \emph{morphism} from a type $X \to \sS$ to a type $Y \to \sS$ is a morphism $X\to Y$ making the following triangle lax commutative.
	\[\begin{tikzcd}
		Y \ar[r] & \sS\\
		X \ar[u] \ar[ru,""{above,name=A}] &
		\ar[from=A,to=1-1,phantom,"\geq"{description,sloped}]
	\end{tikzcd}\]
	Such a morphism is an \emph{elementary embedding} if the inequality in the triangle above is an equality.
\end{definition}

The intuition behind Definition~\ref{dfn:catego-models} is that a morphism of models should preserve truth, but not necessarily reflect it, while an elementary embedding should also reflect truth. More precisely, let $x \colon X\to \sS$ and $y \colon Y\to \sS$ be models of $\sS$ and let $f \colon X\to Y$ be a morphism. Let $n \in \cC$, and let $p \in X^n$ be an $n$-tuple. Then $f(p) \in Y^n$ and we must have $x(p) \leq y(f(p))$, which means that every open up-set containing $x(p)$ also contains $y(f(p))$: ``everything true about $x(p)$ is also true about $y(f(p))$.'' If the inequality in the triangle in Definition~\ref{dfn:catego-models} is an equality, then the converse holds, as well, that is, $x(p)$ and $y(f(p))$ are in exactly the same open up-sets. This corresponds to the usual notion of ``elementary morphism'' in model theory.

\begin{definition}
	We write $\Tp(\sS)$ for the category of types of $\sS$ with the above notion of morphisms, and $\Tp^e(\sS)$ for the subcategory given by the elementary embeddings. We write $\Mod_\omega^e(\sS)$ and $\Mod_\omega(\sS)$ for the full subcategories of $\Tp^e(\sS)$ and $\Tp(\sS)$ on $\omega$-saturated models. If $\sS$ is open, we define $\Mod^e(\sS)$ and $\Mod(\sS)$ similarly for all models instead of $\omega$-saturated ones.
	
	We say that a type $X \to \sS$ is \emph{realized} by a model $Y \to \sS$ if there is a morphism from $X$ to $Y$ in $\Tp^e(\sS)$.
\end{definition}

\begin{remark}
	A morphism $\sS \to \sP$ of polyadic compact ordered spaces induces functors in the same direction between the categories defined in Definition~\ref{dfn:catego-models}.
\end{remark}

\section{The interpolation extension principle}
\label{sec:main-lemma}

Recall from the previous section that we may naturally extend any $\cKOrd$-valued presheaf $\sS \colon \cC^\op \to \cKOrd$ on $\cC$ to a $\cKOrd$-valued presheaf $\ovl{\sS} \colon \cInd(\cC)^\op \to \cKOrd$, via a Kan extension. 
This yields what we will call here the \emph{extension by continuity functor}
\begin{align*}
	[\cC^{\op},\cKOrd] &\to [\cInd(\cC)^{\op},\cKOrd] \\
	\sS &\mapsto \ovl{\sS}.
\end{align*}
The aim of this section is to prove the following principle, which is central to our approach.

\begin{proposition}[Interpolation extension principle]\label{prop:main}
	The extension by continuity functor preserves the axioms \ref{p-axiom:interpol-1}, \ref{p-axiom:interpol-2}, morphisms of polyadic spaces and intuitionistic morphisms of polyadic spaces. In particular, polyadic compact ordered spaces over $\cC$ are stable under extension by continuity.
\end{proposition}

\begin{remark}\label{rmk:open-not-pres}
	On the other hand, openness is in general not preserved by extension by continuity. For instance, with $\cC = \cFinSetInj$, take $\sS(n) = [\![n,\infty]\!]$ as a subspace of the one-point compactification of $\N$. For any injection $f \colon n \hookrightarrow m$, $\sS(f)$ is the canonical inclusion $[\![m,\infty]\!] \hookrightarrow [\![n,\infty]\!]$. This is actually the terminal $\cFinSetInj$-adic space, all we can speak about being the number of elements of the model. Then $\ovl{\sS}(\omega) = \set{\infty}$ but $\set{\infty} \hookrightarrow [\![n,\infty]\!]$ is not lower semi-open. However, it is still possible to show that some maps are always lower semi-open. For instance, given an open $\cFinSet$-adic space $\sS$, for any map of sets $f \colon X \to Y$, if $\setst{(x_1,x_2)\in X^2}{x_1 \neq x_2 \text{ but } f(x_1) = f(x_2)}$ is finite, then $\sS(f)$ is lower semi-open.
\end{remark}

In order to prove the interpolation extension principle, we will use two lemmas. The first shows that a diagram of ind-objects whose shape is a finite poset $\cI$ can be written as a filtered colimit of diagrams in $\cC$ of the same shape $\cI$. It is a special case of \cite[Prop~8.8.5]{SGA4-1}, also see, e.g., \cite[Corollary~6.4.4]{KS06}.

\begin{lemma}\label{lem:filtered-diagrams}
	Let $\cC$ be an essentially small category and let $\cI$ be a finite poset. Then for every functor $F \colon \cI\to \cInd(\cC)$, there exists a filtered category $\cJ$ and a functor $G \colon \cI\times \cJ \to \cC$ such that $F(i) = \colim_j G(i,j)$. (Thus $\cInd(\cC)^\cI \simeq \cInd(\cC^\cI)$.)
\end{lemma}

The second lemma shows that the required properties are preserved in the target category $\cKOrd$.
\begin{lemma}\label{lem:nonempty-limit}
	In $\cKOrd$, cofiltered limits preserve boundedness and the interpolation property.
\end{lemma}

\begin{proof}
	Suppose that we have a cofiltered diagram of commutative squares, indexed by $i\in I$, each with the interpolation property as below.
	\[\begin{tikzcd}
		A_i \ar[r,"v_i"] \ar[d,"u_i"'] & C_i \ar[d,"g_i"] \\
		B_i \ar[r,"f_i"'] \ar[ur,phantom,"\leq"{description,sloped}] & D_i
	\end{tikzcd}\]
	Let $A = \lim_i A_i$, $u = \lim_i u_i$, etc. Let $b \in B$ and $c \in C$ be such that $f(b) \leq g(c)$. Let $b_i$ and $c_i$ be the $i$-components of $b$ and $c$. Now note that the set $u^{-1}(\uset b) \cap v^{-1}(\dset c) = \lim_i [u_i^{-1}(\uset b_i) \cap v_i^{-1}(\dset c_i)]$ is non-empty, by Lemma~\ref{lem:cofiltered-limit-non-empty}. Hence, cofiltered limits preserve the interpolation property of commutative squares.
	
	For the preservation of boundedness, suppose we have a diagram of bounded arrows $f_i \colon A_i \to B_i$ indexed by $i \in I$. Let $f \colon A \to B$ be the limit. Let $a \in A$ and let $b \geq f(a)$. This means that $b_i \geq f_i(a_i)$ for all $i$, so that $f_i^{-1}(b) \cap \uset a \neq \emptyset$ for all $i$, and thus the limit $f^{-1}(b) \cap \uset a$ is nonempty too.
\end{proof}

We are now ready to prove the main result of this section, the interpolation extension principle.

\begin{proof}[Proof of Proposition~\ref{prop:main}]
	Let us show that the interpolation property of natural transformations is preserved. Let $\alpha \colon P\to Q$ be a natural transformation with the interpolation property in $[\cC^{\op},\cKOrd]$. Let $f \colon X\to Y$ be a morphism in $\cInd(\cC)$. By Lemma~\ref{lem:filtered-diagrams}, we can write $f$ as a filtered colimit of $(f_i \colon X_i \to Y_i)_{i\in I}$. Then the square on the left is the cofiltered limit of the squares on the right, and since each of them has the interpolation property, the square on the left too by Lemma~\ref{lem:nonempty-limit}.
	
	\begin{minipage}{0.485\linewidth}
		\[\begin{tikzcd}
			P(X) \ar[r,"\alpha_X"] \ar[d,"P(f)"'] & Q(X) \ar[d,"Q(f)"]\\
			P(Y) \ar[r,"\alpha_Y"] & Q(Y)
		\end{tikzcd}\]\end{minipage}\begin{minipage}{0.485\linewidth}
		\[\begin{tikzcd}
			P(X_i) \ar[r,"\alpha_{X_i}"] \ar[d,"P(f_i)"'] & Q(X_i) \ar[d,"Q(f_i)"]\\
			P(Y_i) \ar[r,"\alpha_{Y_i}"] & Q(Y_i)
		\end{tikzcd}\]
	\end{minipage}
	
	The proofs of preservation of the other properties are similar.
\end{proof}

\begin{example}\label{exmp:failure-amalg-noncompact}
	Let us give an example to show that, in the proof above, we need compactness of the spaces in the target category. Our example will be a polyadic set $\sS \colon \cFinSet^\op \to \cSet$ such that its extension by continuity $\ovl{\sS} \colon \cSet^\op \to \cSet$ is not a polyadic set. This answers in the negative Question~2.22 in \cite{Ben-Yaacov03}. If we think of $[\cFinSet^\op,\cSet]$ as the free cocompletion of $\cFinSet$, this example works as follows: we take the formal coproduct of two copies of $\N$, we break the amalgamation property of the extended presheaf by identifying the two copies of the empty set and we try to reconstruct this amalgamation property in a free enough way, in such a way that the result admits a simple enough description.
	
	We first define a notion of \emph{pseudo-tree} of depth $n$ recursively. It will be a finite set equipped with some extra structure. We will also define recursively the sub-pseudo-trees of a pseudo-tree.
	\begin{enumerate}
		\item A pseudo-tree of depth $0$ is a finite set $X$ equipped with an injection $X \hookrightarrow \N \times \set{\text{black},\text{white}}$ such that the composite $X \to \N\times \set{\text{black},\text{white}}\to \set{\text{black},\text{white}}$ is constant. The sub-pseudo-trees are all the subsets of $X$ equipped with the restricted inclusions.
		\item A pseudo-tree of depth $n+1$ is a finite set $X$ equipped with two subsets $A,B \subseteq X$ such that $A \cup B = X$, each equipped with a structure of pseudo-tree of depth $n$, and such that $A \cap B$ is a sub-pseudo-tree of both $A$ and $B$, with the same induced structure of pseudo-tree. The sub-pseudo-trees are either $X$ itself, or sub-pseudo-trees of $A$ or of $B$. The compatibility condition on $A \cap B$ ensures that each subset of $X$ corresponds at most to one sub-pseudo-tree.
	\end{enumerate}
	
	For each subset $Y$ of a pseudo-tree $X$, there is a smallest sub-pseudo-tree $\langle Y\rangle_X$ of $X$ containing $Y$. We define $\sS(n)$ as the set of functions $f \colon n \to X$ from $n$ to a pseudo-tree $X$ (modulo isomorphisms) such that $\langle\operatorname{image}(f)\rangle_X = X$. Given $g \colon m \to n$ and $f \colon n \to X$ in $\sS(n)$, we define $\sS(g)(f)$ as the canonical map $m \to \langle\operatorname{image}(gf)\rangle_X$.
	
	We claim that $\sS \colon \cFinSet^\op\to \cSet$ has amalgamation. Suppose we have the following diagram where $n \to X$ and $m\to Y$ are elements of $\sS(n)$ and $\sS(m)$.
	
	\[\begin{tikzcd}
		&X&\\
		n \ar[ru] &&Y\\
		k \ar[u] \ar[r] & m \ar[ru]&
	\end{tikzcd}\]
	
	If the two induced elements of $\sS(k)$ are equal, this means that $\langle\operatorname{image}(k\to X)\rangle_X \cong \langle\operatorname{image}(k\to Y)\rangle_Y$. Let $W$ be this pseudo-tree. It is a common sub-pseudo-tree of $X$ and $Y$, so define $Z = X \sqcup_W Y$ with the structure of pseudo-tree induced by the structures of pseudo-trees on $X,Y \subseteq Z$. To finish, build the map $n \sqcup_k m \to Z$ using the functoriality of the pushout as below.
	
	\[\begin{tikzcd}
		&&& X && Z \\
		\\
		n && {n \sqcup_k m} & W && Y \\
		\\
		k && m
		\arrow[hook', from=3-4, to=1-4]
		\arrow[hook, from=3-4, to=3-6]
		\arrow[hook', from=3-6, to=1-6]
		\arrow[from=5-1, to=3-1]
		\arrow[from=5-1, to=5-3]
		\arrow[from=3-1, to=3-3]
		\arrow[from=5-1, to=3-4]
		\arrow[from=3-1, to=1-4]
		\arrow[from=5-3, to=3-6]
		\arrow[dashed, from=3-3, to=1-6, crossing over]
		\arrow[from=5-3, to=3-3, crossing over]
		\arrow[hook, from=1-4, to=1-6]
	\end{tikzcd}\]
	
	This shows that $\sS \colon \cFinSet^\op \to \cSet$ has amalgamation. Now, we will show that the extension by continuity $\ovl{\sS} \colon \cSet^\op \to \cSet$ doesn't have amalgamation. Let $x$ be the element of $\ovl{\sS}(\N)$ such that for all $X \subseteq \N$ finite, the induced element of $\sS(X)$ is the identity $X\to X$ where $X$ has the structure of pseudo-tree of depth $0$ given by the inclusion $x \in X \mapsto (x,\text{white}) \in \N \times \set{\text{black},\text{white}}$. Define $y \in \ovl{\sS}(\N)$ similarly by replacing $\text{white}$ by $\text{black}$. The restrictions of $x$ and $y$ to $\sS(0)$ are equal but we claim that there is no element of $\sS(\N + \N)$ restricting to $x$ and $y$ along the two inclusions $\N \rightrightarrows \N+\N$. Suppose it was the case and let $z \in \sS(\N+\N)$ be such an element. For all $n \in \N$, let $[n] = \set{0,\dots,n}$, so that $z$ is given by a sequence of elements of $\sS([n]+[n])$ depicted below, where $X_{n-1} \to X_n$ is the smallest sub-pseudo-tree containing the image of $[n-1]+[n-1] \hookrightarrow [n]+[n] \to X_n$.
	
	\[\begin{tikzcd}
		{[0]+[0]} \ar[d] \ar[r,hook] & {[1]+[1]} \ar[d] \ar[r,hook] & {[2]+[2]} \ar[d] \ar[r,hook] & \cdots\\
		X_0 \ar[r,hook] & X_1 \ar[r,hook] & X_2 \ar[r,hook] & \cdots
	\end{tikzcd}\]
	
	Then we actually have $X_n = X_0$ for all $n \in \N$. To show that, we need the following fact: Let $X$ be a pseudo-tree, let $U \subseteq X$ be a sub-pseudo-tree of depth $>0$ and let $Z \subseteq X$ be a sub-pseudo-tree of depth $0$. If $Z \cap U \neq \emptyset$, then $Z \subseteq U$. We show that by induction on the depth of $X$. Suppose $X$ has depth $n>0$ and that its structure of pseudo-tree is given by the two sub-pseudo-trees $A,B \subseteq X$. If $U = X$, nothing has to be proven. Otherwise, we can suppose that $U \subseteq A$ without loss of generality. If $Z \subseteq A$ too, then $Z \subseteq U$ by the induction hypothesis. If $Z \subseteq B$, then $Z \cap U \subseteq Z \cap (A \cap B)$, so $Z \cap (A \cap B) \neq \emptyset$, so $Z \subseteq A \cap B$ by the induction hypothesis and we conclude that $Z \subseteq U$ as before.
	
	We can now show that $X_n = X_0$ for all $n \in \N$. Indeed, $X_0 \subseteq X_n$ is a sub-pseudo-tree of depth $>0$ since it contains two sub-pseudo-trees of depth $0$ of different colors. For both canonical inclusions $[n] \to [n]+[n]$, the image of $[n] \to [n]+[n] \to X_n$ has to be a sub-pseudo-tree of depth $0$. Its intersection with $X_0$ is nonempty since it contains the image of $[0] \to [n] \to X_n$. Hence the image of $[n] \to X_n$ is contained in $X_0$. This proves that $X_0$ contains the image of $[n]+[n] \to X_n$, so that $X_0 = X_n$.
	
	Notice that for each pseudo-tree $X$, there is a canonical map $X \to \N \times \set{\text{black},\text{white}}$ whose restriction to each sub-pseudo-tree $Y \subseteq X$ of depth $0$ is the canonical map $Y \to \N \times \set{\text{black},\text{white}}$. The composite $[n]+[n] \to X_n \to \N \times \set{\text{black},\text{white}}$ has to be injective, hence $[n]+[n] \to X_n$ is injective but it is impossible if $X_n = X_0$ for all $n$ since $X_0$ is finite. This produces a contradiction and we deduce that $\ovl{\sS} \colon \cSet^\op \to \cSet$ does not have amalgamation.
	
	This polyadic set $\mathcal{S}$ is the type space functor of an $\mathcal{L}_{\infty,\omega}$-theory which does not have enough set-valued models. 

	This example also shows that the interpolation extension principle is not true if we do not assume that $\cC$ has pushouts. Note first that a category $\cC$ has amalgamation if and only if the presheaf $\sP \colon \cC^\op \to \cKOrd$ constant to the one-point space is a polyadic compact ordered space. The extension by continuity of $\sP$ is the presheaf $\ovl{\sP} \colon \cInd(\cC)^\op \to \cKOrd$ which is also constant to the one-point space. If the interpolation extension principle is true for $\cC$, it implies that $\ovl{\sP}$ has the interpolation property, which means that $\cInd(\cC)$ has amalgamation. But a counter example is given by $\cC = \int \sS$ with $\sS$ as in the example above: $\sS$ is a polyadic set, so $\int \sS$ has amalgamation, but $\cInd(\cC) = \cInd(\int \sS) \simeq \int \ovl{\sS}$ does not.
\end{example}

\begin{remark}\label{rmk:fw-colim-butterflies}
	In this paper, we make the hypothesis that $\cC$ has pushouts, but what we really need is that it satisfies the conclusion of the interpolation extension principle. The only thing that can fail to be preserved is the axiom \ref{p-axiom:interpol-1}. As we will see in Remark~\ref{rmk:equiv-main-enough-kappa-sat} in the next section, the preservation of this axiom is equivalent to the presence of enough $\kappa$-saturated models, for every $\kappa$.
	
	We can see that the axiom \ref{p-axiom:interpol-1} is also preserved if $\cC$ has fjw finite colimits. Indeed, as stated in Remark~\ref{rmk:theories-presheaf-type}, under this condition, coherent $\cC^\op$-hyperdoctrines correspond to extensions of the coherent theory classified by $[\cC,\cSet]$. Since there exist enough $\kappa$-saturated models for coherent multi-sorted theories, the interpolation extension principle is also valid if $\cC$ has fjw finite colimits.
	
	One could wonder if there is a condition generalizing both the presence of pushouts and of fjw finite colimits that could ensure the validity of the extension interpolation principle. We mention without proof such a condition: the principle is valid if $\cC$ has fjw colimits of diagrams of the form below.
	
	
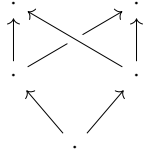
\[\begin{tikzcd}
		\cdot &[-1em]&[-1em] \cdot\\
		\cdot \ar[u] \ar[urr] && \cdot \ar[u] \ar[ull,crossing over]\\
		&\cdot \ar[ul] \ar[ur] &
	\end{tikzcd}\]
\end{remark}

\section{Completeness for coherent logic}
\label{sec:complete-coherent}

In this section, we will see, through a proof of Gödel's completeness theorem, a first example of how the interpolation extension principle can be used. We start with a lemma giving a formulation of what is known as the method of diagrams in model theory. In this lemma, $P(X)$ is thought of as a set of \emph{problems} and $S(X)$ is the set of those problems which have a solution. This lemma allows us to build objects for which every problem has a solution. It will be used again in Section~\ref{sec:variations-int}.

\begin{lemma}\label{lem:build-sol}
	Let $\cD$ be a category admitting all filtered colimits. Let $P \colon \cD \to \cSet$ be a functor preserving filtered colimits and let $S \subseteq P$ be a subfunctor. Suppose that for all $X \in \cD$ and all $p \in P(X)$, there is some $f \colon X \to X'$ such that $P(f)(p) \in S(X')$. Then for all $X$, there is some $f \colon X \to Y$ such that $S(Y) = P(Y)$.
\end{lemma}

\begin{proof}
	Let $X \in \cD$. Choose a well-ordering $(x_\alpha)_{\alpha<\kappa}$ of $P(X)$. We build a sequence $(X_\alpha)_{\alpha \leq \kappa} \subseteq \cD$ as follows.
	\begin{enumerate}
		\item $X_0 = X$.
		\item $f_\alpha \colon X_\alpha \to X_{\alpha+1}$ is chosen such that if $g \colon X_0 \to X_{\alpha}$ is the canonical map, then
		\[ P(g f_\alpha)(x_\alpha) \in S(X_{\alpha+1}) \text{.} \]
		\item For $\alpha$ a limit ordinal, $X_\alpha = \colim_{\lambda < \alpha} X_\lambda$.
	\end{enumerate}
	We put $X^1 = X_{\kappa}$ and we iterate this construction $\omega$ times so as to obtain a sequence
	 \[ X \to X^1 \to X^2 \to \cdots \text{.}\]
	 The image of $P(X^i)$ in $P(X^{i+1})$ is included in $S(X^{i+1})$. Finally, let $Y = \colim_i X^i$. Then $P(Y) = \colim_i P(X^i) \subseteq \colim_i S(X^{i+1}) \subseteq S(Y)$.
\end{proof}

We now prove a version of Gödel's completeness theorem in our setting. The proof can be thought of as a small object argument. It generalizes the usual Henkin proof of Gödel's theorem to our setting.

\begin{theorem}[Gödel's completeness theorem]\label{thm:Godel}
	Let $\sS$ be a polyadic compact ordered space and let $X \in \cInd(\cC)$. Then every $X$-type is realized by an $\omega$-saturated model.
\end{theorem}

\begin{proof}
	Recall that an $X$-type $x \colon X \to \sS$ is an $\omega$-saturated model if and only if each lax diagram as below on the left (where $n,m \in \cC$) can be completed as on the right.
	
	\begin{minipage}{0.48\linewidth}
	\[\begin{tikzcd}
		X \ar[r,"x"] & \sS\\
		n \ar[u,"f"] \ar[r,"g"'] & m \ar[u,"t"'] \ar[lu,phantom,"\geq"{sloped,description}]
	\end{tikzcd}\]
	\end{minipage}\begin{minipage}{0.48\linewidth}
	\[\begin{tikzcd}
		X \ar[r,"x"] & \sS\\
		n \ar[u,"f"] \ar[r,"g"'] & m \ar[u,""{left,name=UL},"t"'] \ar[lu,"h"]
		\ar[from=UL,to=1-1,phantom,"\geq"{description,sloped,pos=0.3}]
	\end{tikzcd}\]
	\end{minipage}
	
	For any $X$-type $x \colon X \to \sS$, let us call a tuple $r = (n,m,f,g,t)$ as in the lax diagram on the left a \emph{request} for $x$, and a morphism $h \colon m \to X$ like in the diagram on the right an \emph{answer} for $r$. Thus, an $X$-type is an $\omega$-saturated model if, and only if, every request has an answer.

	With this terminology in place, we will now construct a functor $P \colon \Tp^e(\sS) \to \cSet$ and a subfunctor $S$ of $P$, to which we will apply Lemma~\ref{lem:build-sol}. On objects, for any $x \in \Tp^e(\sS)$, define $P(x)$ to be the set of requests for $x$. For any morphism $h \colon (x,X) \to (x',X')$ in $\Tp^e(\sS)$, define $Ph \colon P(x) \to P(x')$ to be the function sending any request $r = (n,m,f,g,t)$ for $x$ to the request $(Ph)(r) \coloneqq (n,m,fh,g,t)$ for $x'$.
	Finally, for any $x \in \Tp^e(\sS)$, write $S(x)$ for the subset of $P(x)$ consisting of those requests that admit an answer. We now show that $P$ preserves filtered colimits and that $S$ is a subfunctor of $P$ satisfying the assumption of Lemma~\ref{lem:build-sol}.

	Let $r = (n,m,f,g,t)$ be a request for an $X$-type $x \colon X \to \sS$. Thanks to the interpolation extension principle, we may construct a lax diagram \eqref{eq:Xr-IEP} with $X_r \in \cInd(\cC)$.
	
	\begin{equation}\label{eq:Xr-IEP}
	\begin{tikzcd}
		&&\sS\\
		X \ar[rru,bend left=20] \ar[r,"\iota_1"'] & X_r \ar[ru,"x_r" pos=0.4] &\\
		n \ar[u] \ar[r] & m \ar[u,"\iota_2"] \ar[ruu,bend right=20,""{left,name=UL,pos=0.4}]
		\ar[from=UL,to=2-2,phantom,"\geq"{description,sloped,pos=0.45}]
	\end{tikzcd}
	\end{equation}
	
	\begin{enumerate}
		\item $P$ preserves filtered colimits: if $X = \colim_i X_i$ is a filtered colimit in $\Tp^e(\sS)$, then every request for $x \colon X \to \sS$ is of the form $(P\iota_i)(r)$, where $\iota_i \colon X_i \to X$ denotes the canonical injection and $r$ a request for $\ovl{\sS}(\iota_i)(x) \in \ovl{\sS}(X_i)$.\label{colim-sol}
		\item $S$ is a subfunctor: if $r$ has an answer, then $(Ph)(r)$ has an answer too.\label{transfer-sol}
		\item $S$ satisfies the assumption of Lemma~\ref{lem:build-sol}: with the notations $\iota_1$ and $\iota_2$ of the diagram \eqref{eq:Xr-IEP}, $\iota_2$ is an answer for $(P\iota_1)(r)$.\label{adding-sol}
	\end{enumerate}
	
	Thus, by Lemma~\ref{lem:build-sol}, for any $X$-type $x \colon X \to \sS$, there exists a morphism $f \colon (x, X) \to (y, Y)$ such that $S(y) = P(y)$, and this is exactly an $\omega$-saturated model realizing $x$.
\end{proof}

\begin{remark}\label{rmk:equiv-main-enough-kappa-sat}
	The same reasoning as above shows that for each regular cardinal $\kappa$, each type is realized by a $\kappa$-saturated model, by extending the sets of requests. Actually, the interpolation extension principle for a general category $\cC$ is equivalent to the presence of enough $\kappa$-saturated models, for all $\kappa$. Indeed, suppose that each type is realized by a $\kappa$-saturated model for $\kappa$ arbitrarily large. Let the following be an interpolation problem for $\sS$, with $X,Y,Z$ arbitrary ind-objects of $\cC$.
	
	\[\begin{tikzcd}
		X \ar[r] & \sS\\
		Y \ar[u] \ar[r] & Z \ar[u] \ar[ul,phantom,"\geq"{description,sloped}]
	\end{tikzcd}\]

	Let $\kappa$ be large enough such that $Y$ and $Z$ are $\kappa$-presentable (such a $\kappa$ always exists, see \cite[Rmk. below Thm.~1.20]{AdaRos94}). Let $X' \to \sS$ be a realization of $X \to \sS$ by a $\kappa$-saturated model. Then there is an arrow $Z \to X'$ giving a solution to the interpolation problem.
\end{remark}

\begin{remark}\label{rmk:pos-closed-models}
	We explain how the same reasoning as above also shows the existence of enough \emph{positively closed models}.
	
	Given a compact ordered space $A$, we denote by $\max A$ the set of maximal points of $A$. This defines a functor from compact ordered spaces and bounded morphisms to sets: if $f \colon A\to B$ is bounded, then $f[\max A] \subseteq \max B$. If $\sS$ is a $\cC$-adic compact ordered space, let $\max \sS \colon \cC^\op \to \cSet$ be the composite of $\sS$ with this functor. A model $X \to \sS$ is called \emph{positively closed} if it takes values in $\max \sS \subseteq \sS$. This is the terminology used in \cite{HAYKAZYAN_2019}, but it is also called \emph{existentially} closed in \cite{Ben-Yaacov03,Kamsma_2022}.
	
	If $\ovl{\sS}$ satisfies \ref{p-axiom:interpol-1} and \ref{p-axiom:interpol-2}, then $\max \ovl{\sS}$ satisfies \ref{p-axiom:interpol-1}, using Remark~\ref{rmk:p-axioms-1+2}. Moreover, for any object $X$ of $\cInd(\cC)$, an element $x$ of $\ovl{\sS}(X)$ is in the image of the inclusion $\ovl{\max \sS} \hookrightarrow \ovl{\sS}$ if, and only if, for every $f \colon n \to X$ with $n \in \cC$, the element $\sS(f)(x) \in \sS(n)$ is maximal. From this and \ref{p-axiom:interpol-2}, one can derive that the inclusion $\ovl{\max \sS} \hookrightarrow \ovl{\sS}$ identifies $\ovl{\max \sS}$ with $\max \ovl{\sS}$. 
	Hence $\ovl{\max \sS}$ also satisfies \ref{p-axiom:interpol-1}. Thus, the proof of Gödel's completeness theorem above works also for $\max \sS$ instead of $\sS$: every maximal $X$-type is realized by a positively closed model. In particular, every model $\sS$ admits a morphism to a positively closed one. This implies another definition for positively closed models: they are the models such that every morphism to another model is an elementary embedding.
\end{remark}

\section{Completeness for Kripke models}
\label{sec:complete-intuitionistic}

We will now see the completeness of Kripke models for polyadic compact ordered spaces. To make sense of Kripke models on the algebraic side, we need to make the additional assumption that there is a Heyting implication and universal quantification. On the topological side, it translates as openness conditions, as explained in Section~\ref{sec:duality}. However, similarly to what we did for coherent logic in Sections~\ref{sec:models} and \ref{sec:complete-coherent}, we will work with an $\omega$-saturated notion of Kripke model that doesn't need this openness hypothesis.

\begin{definition}
	Let $F \colon \cI \to \cInd(\cC)$ be a diagram indexed by a small category $\cI$. We write $\tilde{F} \in [\cC^\op,\cPoset]$ for its oplax colimit when considering $\cInd(\cC)$ as a subcategory of $[\cC^\op,\cPoset]$.
\end{definition}

We recall that the elements of $\tilde{F}(c)$ are the pairs $(i,x)$ with $i \in \cI$ and $x \in F(i)$.

\begin{lemma}
	Boundedness and the interpolation property are preserved by oplax colimits in $\cPoset$. In particular, polyadic posets are stable under oplax colimits in $[\cC^\op,\cPoset]$.
\end{lemma}

\begin{proof}
	Let $(X_i)_{i\in\cI}, (Y_i)_{i\in\cI} \subseteq \cPoset$ be two diagrams and let $(F_i \colon X_i \to Y_i)_i$ be a natural transformation. Suppose that each $F_i$ is bounded. Let $X$, $Y$ and $F$ be the respective oplax colimits of $(X_i)_i$, $(Y_i)_i$ and $(F_i)_i$. Suppose $F(i,x) \leq (j,y)$. This means that there is some map $f \colon i\to j$ in $\cI$ with $Y_f(F_i(x)) \leq y$. Since $Y_f(F_i(x)) = F_j(X_f(x))$ and since $F_j$ is bounded, we get some $x' \geq X_f(x)$ with $F_j(x') = y$. This implies that $(i,x) \leq (j,x')$ and $F(j,x') = (j,y)$.
	
	The proof of preservation of the interpolation property is similar.
\end{proof}

Consequently, $\tilde{F}$ is a polyadic poset for any diagram $F \colon \cI \to \cInd(\cC)$.

\begin{definition}
	Let $\sS$ be an intuitionistic polyadic compact ordered space. A \emph{Kripke model} of $\sS$ based on a diagram $F$ of ind-objects is an intuitionistic morphism $\beta \circ \tilde{F} \to \sS$.
\end{definition}

Strengthening this definition, we get the notion of $\omega$-saturated Kripke model below.

\begin{definition}
	Let $\sS$ be a polyadic compact ordered space. An \emph{$\omega$-saturated Kripke model} of $\sS$ based on a diagram $F$ of ind-objects is an intuitionistic morphism $\tilde{F} \to \sS$ (of polyadic posets).
\end{definition}

Note that if $\sS$ is an polyadic Esakia space, then any $\omega$-saturated Kripke model is a Kripke model: the hypothesis on $\sS$ implies that $\beta \circ \sS \to \sS$ is an intuitionistic morphism, as was proved when $\sS$ is a polyadic Esakia space in Corollary~\ref{cor:caract-open-interp}. If moreover $\tilde{F} \to \sS$ is an intuitionistic morphism of polyadic posets, then $\beta \circ \tilde{F} \to \beta \circ \sS$ is also intuitionistic and the composite $\beta \circ \tilde{F} \to \beta \circ \sS \to \sS$ too.

A natural transformations $\tilde{F} \to \sS$ can also be viewed as an oplax cocone $(F(i) \to \sS)_{i\in\cI}$. If $\tilde{F} \to \sS$ is an $\omega$-saturated Kripke model, then each component $F(i) \to \sS$ is an $\omega$-saturated model.
This allows us to view $\omega$-saturated Kripke models as diagrams $\cI \to \Mod_\omega(\sS)$, where we defined $\Mod_\omega(\sS)$ as the category of $\omega$-saturated models of $\sS$. That is how we will think of them from now on. In the same way, Kripke models will be viewed as special functors $\cI \to \Mod(\sS)$.

Here is a more explicit way of stating that $(X_i)_{i \in \cI} \subseteq \Mod_\omega(\sS)$ is an $\omega$-saturated Kripke model:
\begin{enumerate}[wide, labelwidth=!, labelindent=0pt]
	\item[(Implication.)] Suppose that we have the diagram below on the left where $i \in \cI$ and $n\in\cC$. Then there is some morphism $i \to j$ completing the diagram as on the right.
	
	\begin{minipage}{0.48\linewidth}
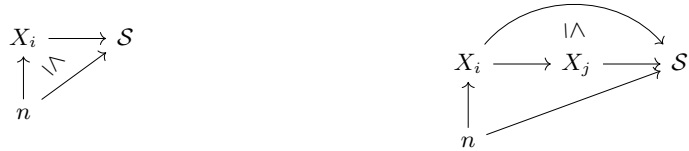
\[\begin{tikzcd}
			X_i \ar[r] & \sS\\
			n \ar[u] \ar[ur,""{left,name=A}] &
			\ar[from=A,to=1-1,phantom,"\leq"{description,sloped}]
		\end{tikzcd}\]\end{minipage}\begin{minipage}{0.48\linewidth}\[\begin{tikzcd}
			X_i \ar[r] \ar[rr,bend left=50,""{below,name=A}] & X_j \ar[r] & \sS\\
			n \ar[u] \ar[urr] & &
			\ar[from=A,to=1-2,phantom,"\leq"{description,sloped}]
		\end{tikzcd}\]\end{minipage}
	
	\item[(Universal quantification.)] Suppose we have the diagram below on the left where $i \in \cI$ and $n,m \in \cC$. Then there is a morphism $i \to j$ and a way of completing the diagram as on the right.
	
	\begin{minipage}{0.48\linewidth}
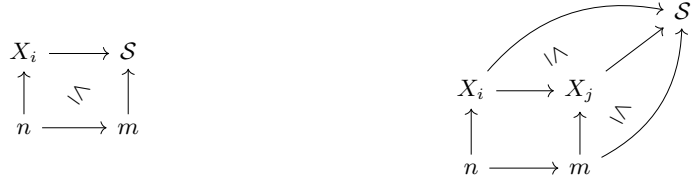
\[\begin{tikzcd}
			X_i \ar[r] & \sS\\
			n \ar[u] \ar[r] & m \ar[u] \ar[lu,phantom,"\leq"{description,sloped}]
		\end{tikzcd}\]\end{minipage}\begin{minipage}{0.48\linewidth}\[\begin{tikzcd}
			&& \sS \\
			X_i \ar[r] \ar[urr,bend left=30,""{below,name=A,pos=0.35}] & X_j \ar[ur] &\\
			n \ar[u] \ar[r] & m \ar[u] \ar[ruu,bend right=30,""{left,name=B,pos=0.35}] &
			\ar[from=A,to=2-2,phantom,"\leq"{description,sloped}] \ar[from=2-2,to=B,phantom,"\leq"{description,sloped}]
		\end{tikzcd}\]\end{minipage}
\end{enumerate}

As in Remark~\ref{rmk:p-axioms-1+2}, the two conditions above can be merged into one by requiring in the ``universal quantification'' part that $m \to X_j \to \sS = m \to \sS$ instead of $m \to X_j \to \sS \leq m \to \sS$.

Intuitively, the inequality $n \to X_i \to \sS \leq n \to m \to \sS$ represents the statement that some universally quantified sentence is \emph{not} satisfied, and we must provide a counter-example at a later stage in the Kripke model.

\begin{remark}
	In the usual notion of Kripke model, the set of worlds is a poset, replaced here by a more general category $\cI$, as is commonly done in the context of, for instance, Kripke--Joyal semantics, see, e.g., \cite{MarquisReyes2011} for more on the rich history of the topic. From any Kripke model in this generalized sense, we can extract Kripke models indexed by trees, as we explain now. Suppose that $F \colon \cI \to \Mod_\omega(\sS)$ is an $\omega$-saturated Kripke model, with $\cI$ a small category. Let $i \in \cI$. We define $\cJ$ to be the poset whose elements are finite paths in $\cI$ starting at $i$, with the extension order. The minimal element of $\cJ$ is the constant path at $i$. Then the composite of the canonical projection $\cJ \to \cI$ and $F$ is also an $\omega$-saturated Kripke model. 
	If $\cI$ is not set-sized, one can still extract set-sized tree-shaped models from it with a similar construction. For this, we select recursively, for each node in the tree starting with the root $i$, a \emph{set} of children solving each of the diagram-completion problems explained above, instead of all possible children.
	In topos-theoretic terms, the above construction is an instance of the Diaconescu cover applied to a presheaf topos, see, e.g., \cite[Sec.~IX.9]{MM1992}.
\end{remark}
We now derive the following theorem which can be seen as a version of a theorem due to Joyal, see, e.g., \cite[p.~75]{MarquisReyes2011}.
\begin{theorem}[Joyal's completeness theorem]
	Let $\sS$ be a polyadic compact ordered space. Then the identity $\Mod_\omega(\sS) \to \Mod_\omega(\sS)$ is an $\omega$-saturated Kripke model.
\end{theorem}

\begin{remark}
	There is a strong analogy with Esakia duality. In Esakia duality, one looks at the space of \emph{coherent} models of a Heyting algebra, i.e., the Priestley dual of the algebra as a distributive lattice. If the distributive lattice happens to be a Heyting algebra, one gets a Kripke model (the canonical one). Here, we do the same thing: the collection of all coherent models of an intuitionistic $\cC$-adic compact ordered space forms a Kripke model.
\end{remark}

\begin{proof}
	Suppose we have the following lax square where $x \colon X \to \sS$ is a model.
	
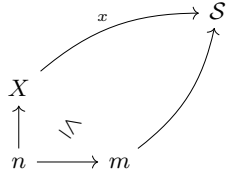
\[\begin{tikzcd}
		& & \sS\\
		X \ar[rru,"x",bend left=20] & &\\
		n \ar[u] \ar[r] & m \ar[uur,bend right=20] \ar[lu,phantom,"\leq"{description,sloped}] &
	\end{tikzcd}\]
	Then thanks to the interpolation extension principle, we can complete it as below where $Y \to \sS$ is a $Y$-type.
	
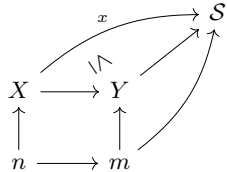
\[\begin{tikzcd}
		& & \sS\\
		X \ar[rru,"x",bend left=20,""'{name=UL,pos=0.3}] \ar[r] & Y \ar[ru] &\\
		n \ar[u] \ar[r] & m \ar[uur,bend right=20] \ar[u] &
		\ar[from=UL,to=2-2,phantom,"\leq"{description,sloped}]
	\end{tikzcd}\]
	Gödel's completeness theorem allows us to factor $Y \to \sS$ through a model and we are done.
\end{proof}

\section{Omitting types}
\label{sec:omitting-types}

In this section, we will see an omitting types theorem applicable to open polyadic compact ordered spaces. This is similar to \cite{EAGLE2021102907,HAYKAZYAN_2019,RasiowaSikorski1950}, where omitting types theorems are put in connection with the Baire property. However, in our context, a formulation as an \emph{application} of the Baire property doesn't seem natural, so we essentially mix its proof with that of Gödel's completeness theorem. We do not assume that the base category $\cC$ has pushouts in this section. See Remark~\ref{rmk:no-pushout} for the definition of a $\cC$-adic compact ordered space when $\cC$ doesn't have pushouts.

Let $\sS$ be a $\cC$-adic compact ordered space. We say that a model $X \to \sS$ \emph{omits} a type $t \in \sS(n)$ with $n \in \cC$ if there is no arrow $n \to X$ such that the composite $n \to X \to \sS$ is $t$. Omitting types theorems give conditions for the existence of models avoiding a given set of types.

In this section, we work with the stably compact topology of compact ordered spaces. The \emph{stably compact interior} of a subset $A \subseteq X$ of a compact ordered space is its interior in the stably compact topology, i.e., the largest open up-set contained in it. We will say that $A$ is \emph{meager} if it is meager in the stably compact topology, i.e., if it is contained in a countable union of closed down-sets containing no nonempty open up-set.

An essentially small category $\cC$ is \emph{essentially countable} if it is equivalent to a small category with countably many arrows. If $\sS$ is a $\cC$-adic space, we say that it has a countable basis of opens if every $\sS(n)$ for $n\in\cC$ admits a countable basis of opens (either in the stably compact topology or in the compact ordered topology, it is equivalent). When $\sS$ is a polyadic Priestley space, the dual of this condition is that the Priestley dual of each $\sS(n)$ is countable.

In the following statement, we will use the convention that when an arrow in a diagram is labeled by a \emph{set} of morphisms, the commutativity of the diagram means that there is \emph{some} arrow in the set making the diagram commute.

\begin{proposition}[Omitting types]
	Let $\cC$ be an essentially countable category. Let $\sS$ be an open $\cC$-adic compact ordered space with a countable basis of opens. For each $n \in \cC$, let $A_n \subseteq \sS(n)$ be a meager subset. Let $n_0 \in \cC$ and let $R \subseteq \sS(n_0)$ be a nonempty open up-set. Then there exist $X \in \cInd(\cC)$ of presentability rank at most $\omega_1$, a morphism $n_0 \to X$, and a model $X\to\sS$ which omits $t$ for every $t \in A_n$ and $n \in \cC$, such that the diagram below commutes.
	\[\begin{tikzcd}
		X \ar[r] & \sS\\
		n_0 \ar[u] \ar[ur,"R"'] &
	\end{tikzcd}\]
\end{proposition}

\begin{proof}
	Without loss of generality, suppose that for each $f \colon n \to m$ in $\cC$, we have $\sS(f)^{-1}(A_n) \subseteq A_m$. Otherwise, we can take the closure of the $A_n$ under these conditions: since $\sS(f)$ is lower semi-open, $\sS(f)^{-1}(A_n)$ is meager and the countable union of meager subsets stays meager. This step is the reason for which we need to use the stably compact topology. We will not need that $\sS$ is open anymore in this proof.
	
	To organize the induction, we choose a bijection $e = (e_0,e_1) \colon \N \to \N^2$ such that $e_0(n) \leq n$ for all $n \in \N$. We also fix, for each $n \in \cC$, a basis of opens of the stably compact topology of $\sS(n)$.
	
	We will build inductively a sequence $n_0 \to n_1 \to \cdots$ of objects of $\cC$, and a sequence $(R_i \subseteq \sS(n_i))_{i\in\N}$ of closed up-sets with nonempty stably compact interior such that $\sS(n_i \to n_{i+1})(R_{i+1}) \subseteq R_i$. During the induction, we choose for each $i \in \N$:
	\begin{enumerate}
		\item a sequence $(Q_i^k)_{k\in\N}$ of closed meager subsets of $\sS(n_i)$ such that $A_{n_i} \subseteq \bigcup_k Q_i^k$ and such that $\sS(n_{i-1}\to n_i)^{-1}(Q_{i-1}^k) \subseteq Q_i^k$ for all $k$;
		\item an enumeration $(p_i^k,q_i^k,U_i^k)_{k\in\N}$ of all the configurations of the following form, where $U \subseteq \sS(q)$ is an open up-set of the fixed basis of $\sS(q)$.
		\[\begin{tikzcd}
			n_i & \sS\\
			p \ar[u] \ar[r] & q \ar[u,"U"']
		\end{tikzcd}\]
	\end{enumerate}
	The hypothesis of the proposition gives $n_0$, and we can take a closed up-set $R_0 \subseteq R$ with nonempty stably compact interior. We explain how to build $(n_{i+1},R_{i+1})$ from $(n_i,R_i)$. First, we choose a closed up-set $R_i' \subseteq R_i$ with nonempty stably compact interior and disjoint from each $Q_i^k$ for $k \leq i$. Let $(a,b) = e(i)$ and consider the following (possibly non-commutative) diagram.
	
	\[\begin{tikzcd}
		n_a \ar[r] & n_i \ar[r,"R_i'"] & \sS\\
		p_q^b \ar[rr] \ar[u] & & q_a^b \ar[u,"U_a^b"']
	\end{tikzcd}\]
	
	If this diagram doesn't commute, we take $n_{i+1} = n_i$ and $R_{i+1} = R_i'$. Otherwise, let $x \in R_i'$ and $u \in U_a^b$ witnessing the commutativity. Thanks to the interpolation property of $\sS$, we can complete the diagram as follows.
	
	\[\begin{tikzcd}
		&&& \sS\\
		n_a \ar[r] & n_i \ar[r,"\alpha_i"'] \ar[urr,bend left=20,"x"] & n_{i+1} \ar[ur,"x'"] &\\
		p_a^b \ar[u] \ar[rr] & & q_a^b \ar[u,"\beta_i"] \ar[uur,bend right=20,"u"',""{name=UL}]&
		\ar[from=UL,to=2-3,phantom,"\geq"{sloped,description}]
	\end{tikzcd}\]
	
	Then $\sS(\alpha_i)^{-1}(R_i') \cap \sS(\beta_i)^{-1}(U_a^b)$ is a nonempty open up-set since it contains $x'$. We choose $R_{i+1}$ to be any closed up-set contained in $\sS(\alpha_i)^{-1}(R_i') \cap \sS(\beta_i)^{-1}(U_a^b)$ and with nonempty stably compact interior.
	
	Once the induction is finished, we take $X = \colim_i n_i$, and our model is any point $x \in \lim_i R_i \subseteq \sS(X)$.
	
	To check that $x \in \sS(X)$ is a model, consider the following commutative square with $U \subseteq \sS(q)$ an open up-set in the chosen basis.
	
	\[\begin{tikzcd}
		X \ar[r,"x"] & \sS\\
		p \ar[u] \ar[r] & q \ar[u,"U"']
	\end{tikzcd}\]
	
	We can factor $p \to X$ through some $n_i \to X$ and find $j \geq i$ such that $e_0(j) = i$ and $e_1(j)$ is the index of the situation $(p,q,U)$ associated to $n_i$. The arrow $n_{j+1} \to X \to P$ is in $R_{j+1}$ and the situation is solved at the step $j+1$ as illustrated below, by construction of $n_{j+1}$.
	
	\[\begin{tikzcd}
		n_i \ar[r] &[-1em] n_j \ar[r] &[-1em] n_{j+1} \ar[r] &[-1em] X \ar[r,"x"] & \sS\\
		&&&p \ar[ulll,bend left=20] \ar[r] & q \ar[u,"U"'] \ar[ull,dashed,"\beta_j" pos=0.65]
	\end{tikzcd}\]
	
	To finish, we check that $x \in \sS(X)$ avoids all the $A_n$. For all $i \in \N$, we know that $n_i \to X \to \sS$ is in $R_i$, so it is not in any of the $Q_i^k$ for $k \leq i$. For $k > i$, we have $\sS(n_i \to n_k)^{-1}(Q_i^k) \subseteq Q_k^k$, and since $n_k \to X \to P$ is not in $Q_k^k$, we also know that $n_i \to n_k \to X \to P$ is not in $Q_i^k$. Hence $n_i \to X \to \sS$ is not in $A_{n_i}$. Let $n \in \cC$ and let $n \to X$ be any arrow. It factors through one of the $n_i \to X$. Since $\sS(n\to n_i)^{-1}(A_n) \subseteq A_{n_i}$ and since $n_i \to X \to P$ is not in $A_{n_i}$, we also have $n \to n_i \to X \to \sS$ not in $A_n$.
\end{proof}

\section{Colimits in categories of models}
\label{sec:access}

In this section, we explain how to compute filtered colimits in $\Mod(\sS)$ and $\Mod_\omega^e(\sS)$. This will be used in Section~\ref{sec:FO-interpol}. We will treat two cases: a strict construction for $\omega$-saturated models, which doesn't need the openness hypothesis, and a lax construction for general models (non-$\omega$-saturated ones), but which needs openness. These categories are even accessible, as can be shown by writing them as categories of models of some $L_{\infty,\infty}$ theory in the sense of \cite[Sect.~3.2]{MakkaiPare}, but we will not need that.

\begin{lemma}\label{lem:Tarski-Vaught}
	Let $\sS$ be a $\cC$-adic compact ordered space. Then filtered colimits in $\Mod^e_\omega(\sS)$ exist and are preserved by the forgetful functor $\Mod^e_\omega(\sS) \to \cInd(\cC)$.
\end{lemma}

\begin{proof}
	Let $\sS$ be a $\cC$-adic compact ordered space and let $(X_i \to \sS)_i$ be a filtered cocone with apex $\sS$ such that every $X_i \to \sS$ is an $\omega$-saturated model for each $i$. We want to show that $X = \colim_i X_i \to \sS$ is still an $\omega$-saturated model. Suppose we have a lax commutative square such as below with $n,m \in \cC$.
	\[\begin{tikzcd}
		X \ar[r] & \sS\\
		n \ar[u] \ar[r] & m \ar[u] \ar[lu,phantom,"\geq"{sloped,description}]
	\end{tikzcd}\]
	We want to show that there is some morphism $m \to X$ making the two triangles lax commutative. But since $n \in \cC$ and since $X = \colim_i X_i$ is a filtered colimit, we can factor $n \to X$ through some $X_i \to X$. Then because $X_i \to \sS$ is a model, there is some $m \to X_i$ making the two triangles lax commutative and the composite $m \to X_i \to X$ gives the desired morphism.
\end{proof}

Models too are stable by filtered colimits of elementary embeddings, but they also admit a stronger stability property if $\sS$ is open. This generalizes the well-known fact (see, e.g., \cite[Thm.~5.23]{AdaRos94}) that directed colimits in a category of models of a first-order theory are computed as in $\cSet$.

\begin{lemma}\label{lem:lax-colim-models}
	Let $\sS$ be an open $\cC$-adic compact ordered space. Then filtered colimits in $\Mod(\sS)$ exist and are preserved by the forgetful functor $\Mod(\sS) \to \cInd(\cC)$.
\end{lemma}

\begin{proof}
	Let $\sS$ be an open $\cC$-adic compact ordered space. We will use Proposition~\ref{prop:weak-interp-struct} in order to manipulate models of $\sS$. A filtered diagram in $\Mod(\sS)$ is given by an oplax cocone $(h_i \colon X_i \to \sS)_i$ where each $h_i$ is a $\sS$-structure. Let $X = \colim_i X_i$, let $u_{i,j} \colon X_i \to X_j$ be the connecting morphisms and let $u_i \colon X_i \to X$ be the canonical injections. We define the natural transformation $X \to \sS$ as the one sending $pu_i \colon n \to X_i \to X$ to the increasing limit $\lim_{j\geq i} p u_{i,j} h_j$. Given any other model $Y \to \sS$ equipped with a cocone $(X_i \to Y)_i$ in $\Mod(\sS)$, there is a unique morphism $X \to Y$ factoring the morphisms $X_i \to Y$.
	
	To conclude the proof, we need to show that $X \to \sS$ is a model of $\sS$. Suppose we have a commutative diagram such as below on the left with $n,m \in \cC$, and where $U$ is an open up-set of $\sS(m)$. Our goal is to complete it as in the commutative diagram on the right, where we reuse the convention of Section~\ref{sec:omitting-types} concerning arrows indexed by sets of morphisms.
	
	\begin{minipage}{0.48\linewidth}
		\[\begin{tikzcd}
			X \ar[r,"h"] & \sS\\
			n \ar[r,"f"'] \ar[u,"p"] & m \ar[u,"U"']
		\end{tikzcd}\]
	\end{minipage}\begin{minipage}{0.48\linewidth}
		\[\begin{tikzcd}
			X \ar[r,"h"] & \sS\\
			n \ar[r,"f"'] \ar[u,"p"] & m \ar[u,"U"'] \ar[lu,"w"]
		\end{tikzcd}\]
	\end{minipage}
	
	Since $n \in \cC$, we can write $p \colon n \to X$ as $p_i u_i$ for some $p_i \colon n \to X_i$. By definition of $h$, we have $ph = \lim_{j\geq i} p_i u_{i,j} h_j$. Since $fU$ is open, we can suppose that $p_i h_i \in fU$, replacing $i$ by some $j \geq i$ if needed. Since $h_i \colon X_i \to \sS$ is a model, there is some $w_i \colon m \to X_i$ such that $fw_i = p_i$ and $w_i h_i \in U$. Let $w = w_i u_i$. Then $f w = f w_i u_i = p_i u_i = p$ and $w h \in U$ since $w h \geq w_i h_i$.
\end{proof}

\begin{remark}
	If $\cC$ has fjw finite colimits, then $\cInd(\cC)$ has ultraproducts because it is the category of models of some coherent first-order theory. Given a family $(X_i)_{i\in I} \subseteq \cInd(\cC)$, and some $\lambda \in \beta I$, we denote by $\int_i X_i \mathrm{d}\lambda$ the corresponding ultraproduct. If $n \in \cC$, we have $\Hom(n,\int_i X_i \mathrm{d}\lambda) = \int_i \Hom(n,X_i) \mathrm{d}\lambda$, where the latter ultraproduct is the usual one in the category of sets. If $\sS$ is an open polyadic compact ordered space on $\cC$, then a family of models $(m_i \colon X_i \to \sS)_{i\in I}$ can be turned into a model $m \colon \int_i X_i \mathrm{d}\lambda \to \sS$, using a similar argument as the previous proof. To describe the transformation $\int_i X_i \mathrm{d}\lambda \to \sS$ is described as follows, pick $x \in \big(\int_i X_i \mathrm{d}\lambda\big)(n)$. Then $x$ is the equivalence class of some family $(x_i \in X_i(n))_{i\in I}$, and we define $m(x) = \lim_{i \to \lambda} m_i(x_i)$.
\end{remark}

\section{First-order interpolation}
\label{sec:FO-interpol}

Recall Robinson's consistency theorem from classical (Boolean) model theory.

\begin{theorem}[Robinson]
	Let $\mathcal{L}_1$ and $\mathcal{L}_2$ be two first-order signatures with a possibly non-empty intersection. Let $T_1$ and $T_2$ be two Boolean first-order theories on respectively $\mathcal{L}_1$ and $\mathcal{L}_2$. If $T_1 \cap T_2$ is consistent, then $T_1 \cup T_2$ is also consistent.
\end{theorem}

In terms of hyperdoctrines, this can be reformulated as follows. Given a signature $\mathcal{L}$, let $\sS_{\mathcal{L}} \colon \cFinSet^\op \to \cBoolSp$ be the polyadic space associated to the empty theory on $\mathcal{L}$. A theory $T$ on signature $\mathcal{L}$ can be seen as a closed subset of $\sS_{\mathcal{L}}(0)$. If $\mathcal{L}' \subseteq \mathcal{L}$, then $T$ can be restricted to $\mathcal{L}'$ by taking its direct image under the canonical map $\sS_{\mathcal{L}}(0) \to \sS_{\mathcal{L}'}(0)$. Keeping that in mind, Robinson's consistency theorem says that the square below has the interpolation property (or, equivalently, the amalgamation property in this Boolean setting).

\[\begin{tikzcd}
	\sS_{\mathcal{L}_1 \cup \mathcal{L}_2}(0) \ar[r] \ar[d] & \sS_{\mathcal{L}_1}(0) \ar[d]\\
	\sS_{\mathcal{L}_2}(0) \ar[r] & \sS_{\mathcal{L}_1 \cap \mathcal{L}_2}(0)
\end{tikzcd}\]

The first proposition of this section is an adaptation of Robinson's consistency theorem, and its usual proof, to an ordered and non-zero dimensional setting.

\begin{proposition}\label{prop:intuitionistic-interpolation}
	Suppose we have a square like below on the left with $\sS,\sP,\sQ$ three $\cC$-adic compact ordered spaces and $X \in \cInd(\cC)$. Suppose that $\sS$ is open and that $f \colon \sS \to \sP$ is an intuitionistic morphism. Then there is some way of completing the square like below on the right with $X' \to \sS$ and $X' \to \sQ$ models.
	
	\begin{minipage}{0.48\linewidth}
		\[\begin{tikzcd}
			\sP & \sS \ar[l,"f"'] \ar[ld,"\geq"{sloped,description}, phantom]\\
			\sQ \ar[u,"g"] & X \ar[l,"u"] \ar[u,"v"']
		\end{tikzcd}\]
	\end{minipage}\begin{minipage}{0.48\linewidth}
		\[\begin{tikzcd}
			\sP & \sS \ar[l,"f"'] &\\
			\sQ \ar[u,"g"] & X' \ar[l,dashed] \ar[u,dashed] &\\
			& & X \ar[lu,dashed] \ar[llu,bend left=25,"u"] \ar[luu,bend right=25,"v"',""{name=V,below}]
			\ar[from=V,to=\tikzcdmatrixname-2-2,phantom,"\geq"{sloped,description}]
		\end{tikzcd}\]
	\end{minipage}
\end{proposition}

\begin{proof}
	We will apply the following two constructions a countable number of times. These two constructions take as input a lax commutative diagram as below, where $\sS,\sP,\sQ$ are $\cC$-adic spaces with $\sS$ open, $\sS \to \sP$ intuitionistic, and where $X \in \cInd(\cC)$.
	
	\[\begin{tikzcd}
		\sP & \sS \ar[l,"f"'] \ar[ld,"\geq"{sloped,description}, phantom]\\
		\sQ \ar[u,"g"] & X \ar[l,"u"] \ar[u,"v"']
	\end{tikzcd}\]
	
	\paragraph{First construction} Thanks to Gödel's completeness theorem, we can write $X \to \sS$ as a composite $X \to X' \to \sS$ where $X' \to \sS$ is a model of $\sS$. We obtain the lax commutative square below on the left. Thanks to the interpolation extension principle, we can find an arrow $X' \to \sQ$ as in the diagram on the right.
	
	\begin{minipage}{0.48\linewidth}
		\[\begin{tikzcd}
			\sP & X' \ar[l] \ar[ld,phantom,"\geq"{description,sloped}]\\
			\sQ \ar[u] & X \ar[l] \ar[u]
		\end{tikzcd}\]
	\end{minipage}\begin{minipage}{0.48\linewidth}
		\[\begin{tikzcd}
			\sP & X' \ar[l,""{name=LU}] \ar[ld]\\
			\sQ \ar[u] & X \ar[l] \ar[u]
			\ar[from=LU,to=2-1,phantom,near start,"\geq"{sloped,description}]
		\end{tikzcd}\]
	\end{minipage}
	
	In terms of our initial diagram, this means that we can complete it as below with $X' \to \sS$ a model.
	
	\[\begin{tikzcd}
		\sP & \sS \ar[l] \ar[ld,phantom,"\geq"{sloped,description}] &[-1em]\\
		\sQ \ar[u] & X' \ar[l] \ar[u] &\\[-1em]
		& & X \ar[lu] \ar[llu,bend left=20] \ar[luu,bend right=20]
	\end{tikzcd}\]
	
	\paragraph{Second construction} In the second construction, we use Gödel's completeness theorem to write $X \to \sQ$ as a composite $X \to X' \to \sQ$ with $X' \to \sQ$ a model. We obtain the square below on the left. Thanks to the fact that $\sS \to \sP$ is an intuitionistic morphism and thanks to the interpolation extension principle, we obtain an arrow $X' \to \sS$ as in the diagram on the right.
	
	\begin{minipage}{0.48\linewidth}
		\[\begin{tikzcd}
			\sP & \sS \ar[l] \ar[ld,phantom,"\geq"{sloped,description}]\\
			X' \ar[u] & X \ar[u] \ar[l]
		\end{tikzcd}\]
	\end{minipage}\begin{minipage}{0.48\linewidth}
		\[\begin{tikzcd}
			\sP & \sS \ar[l]\\
			X' \ar[u] \ar[ru] & X \ar[u,""{name=UL,pos=0.4}] \ar[l]
			\ar[from=UL,to=2-1,phantom,"\geq"{sloped,description,pos=0.3}]
		\end{tikzcd}\]
	\end{minipage}
	
	In terms of our initial diagram, this means that we can complete it as below with $X' \to \sQ$ a model.
	
	\[\begin{tikzcd}
		\sP & \sS \ar[l] &[-1em]\\
		\sQ \ar[u] & X' \ar[l] \ar[u] &\\[-1em]
		& & X \ar[lu] \ar[llu,bend left=20] \ar[luu,bend right=20,""{left,name=LU}]
		\ar[from=LU,to=\tikzcdmatrixname-2-2,phantom,"\geq"{sloped,description}]
	\end{tikzcd}\]
	
	\paragraph{Iterating the constructions} We now iterate our constructions, alternating the two. We get a sequence $X = X_0 \to X_1 \to X_2 \to \cdots$ where $X_{2n} \to X_{2n+1}$ is obtained with the first construction and $X_{2n+1} \to X_{2n+2}$ is obtained with the second construction. Our final model is $X_{\omega} = \colim_i X_i$. The arrow $X_{\omega} \to \sQ$ is built using the cocone $(X_i \to \sQ)_i$. The arrow $X_{\omega} \to \sS$ is built using the oplax cocone $(X_i \to \sS)_i$. As a consequence of Lemma~\ref{lem:lax-colim-models}, these two arrows are both models, because each $X_{2n} \to \sQ$ is a model for $X_{\omega} \to \sQ$, and because each $X_{2n+1} \to \sS$ is a model for $X_{\omega} \to \sS$. We also have $X \to \sQ = X \to X_{\omega} \to \sQ$ and $X \to \sS \leq X \to X_{\omega} \to \sS$.
	
	The last thing we have to show is that $X_{\omega} \to \sQ \to \sP = X_{\omega} \to \sS \to \sP$. This is indeed the case since $X_{\omega} \to \sS \to \sP$ is obtained from the oplax cocone $(X_i \to \sS \to \sP)_i$ which is equal to $(X_i \to \sQ \to \sP)_i$ on even indices, which itself produces $X_{\omega} \to \sQ \to \sP$.
\end{proof}

\begin{remark}
	In the Boolean case, we don't need openness in the proof above: the orders being discrete, the oplax cocones are actually just cocones and we can use Lemma~\ref{lem:Tarski-Vaught} instead of Lemma~\ref{lem:lax-colim-models}.
\end{remark}

As a consequence of this, we obtain interpolation for first-order intuitionistic logic.

\begin{proposition}
	Let $\sS \to \sP \from \sQ$ be a cospan of intuitionistic morphisms between $\cC$-adic compact ordered spaces such as below, where $\sS$ and $\sQ$ are open. Let $\cat{B}$ be the category of pairs $(X \to \sS, X \to \sQ)$ of models of $\sS$ and $\sQ$ making the square commute, where morphisms from $(X \to \sS, X \to \sQ)$ to $(Y \to \sS, Y \to \sQ)$ are morphisms $X \to Y$ such that $X \to \sS \leq X \to Y \to \sS$ and $X \to \sQ \leq X \to Y \to \sQ$. Then the canonical projections $\cat{B} \to \Mod(\sS)$ and $\cat{B} \to \Mod(\sQ)$ are Kripke models of $\sS$ and $\sQ$.
\end{proposition}

We thus obtain a commutative square of intuitionistic morphisms
\[\begin{tikzcd}
	\sP & \sS \ar[l]\\
	\sQ \ar[u] & \tilde{F} \ar[l] \ar[u]
\end{tikzcd}\]
which has pointwise the interpolation property, where $F \colon \cat{B} \to \cInd(\cC)$ is the projection sending $(X \to \sS, X \to \sQ)$ to $X$. Since the dual of interpolation is interpolation, the dual square of intuitionistic hyperdoctrines also has pointwise interpolation: this implies Craig interpolation for first-order intuitionistic logic.

\begin{proof}
	Suppose we have a diagram as below where $X \to \sS$ and $X \to \sQ$ are models and where $n,m \in \cC$.
	
	\[\begin{tikzcd}
		\sP & \sS \ar[l] & m \ar[l] \ar[ld,phantom,"\leq"{sloped,description}]\\
		\sQ \ar[u] & X \ar[l] \ar[u] & n \ar[l] \ar[u]
	\end{tikzcd}\]
	
	Like in Joyal's completeness theorem, we can complete the diagram as below with $X' \in \cInd(\cC)$.
	
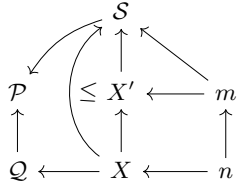
	\[\begin{tikzcd}
		& \sS \ar[ld,bend right=20] & \\
		\sP & X' \ar[u] & m \ar[l] \ar[lu]\\
		\sQ \ar[u] & X \ar[l] \ar[u] \ar[uu,bend left=55,""{right,name=LU}] & n \ar[l] \ar[u]
		\ar[from=LU,to=\tikzcdmatrixname-2-2,phantom,"\leq"{sloped,description}]
	\end{tikzcd}\]
	
	Let us consider the part of the diagram below on the left. Since $\sQ \to \sP$ is an intuitionistic morphism, we can complete our diagram like on the right.
	
	\begin{minipage}{0.48\linewidth}
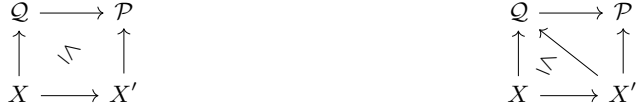
		\[\begin{tikzcd}
			\sQ \ar[r] \ar[rd,phantom,"\leq"{sloped,description}] & \sP\\
			X \ar[u] \ar[r] & X' \ar[u]
		\end{tikzcd}\]
	\end{minipage}\begin{minipage}{0.48\linewidth}
		\[\begin{tikzcd}
			\sQ \ar[r] & \sP\\
			X \ar[u,""{right,name=LU}] \ar[r] & X' \ar[u] \ar[lu]
			\ar[from=\tikzcdmatrixname-2-2,to=LU,phantom,near end,"\leq"{sloped,description}]
		\end{tikzcd}\]
	\end{minipage}
	
	The last step is to extend the span $(X' \to \sS,X' \to \sQ)$ using Proposition~\ref{prop:intuitionistic-interpolation} as below, where $X'' \to \sQ$ and $X'' \to \sS$ are models.
	
	
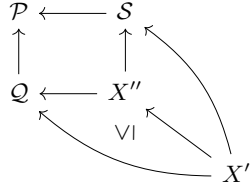
\[\begin{tikzcd}
		\sP & \sS \ar[l] & \\
		\sQ \ar[u] & X'' \ar[l] \ar[u] & \\
		& & X' \ar[lu] \ar[uul,bend right=20] \ar[ull,bend left=20,""{right,name=LU}]
		\ar[from=LU,to=\tikzcdmatrixname-2-2,phantom,"\leq"{sloped,description}]
	\end{tikzcd}\]
	
	Gluing everything together, we get an extension as desired. Note that we must use Proposition~\ref{prop:intuitionistic-interpolation} in such a way as to have $X' \to \sQ \leq X' \to X'' \to \sQ$ and $X' \to \sS = X' \to X'' \to \sS$, not the other way around, because we need to get in the end
	\[ m \to \sS = m \to X' \to \sS = m \to X' \to X'' \to \sS = m \to X'' \to \sS \text{.} \]
\end{proof}

As in \cite{PittsOpenMap}, we also get Beth definability as a consequence.

\begin{proposition}\label{prop:Beth-def}
	Let $\sS,\sP$ be two polyadic compact ordered spaces with $\sS$ open and let $f \colon \sS \to \sP$ be an intuitionistic morphism. If $f$ is injective on models, then it is a pointwise embedding.
\end{proposition}

\begin{proof}
	Suppose that $x \colon n \to \sS$ and $y \colon n \to \sS$ are two $n$-types of $\sS$ with $f(x) \leq f(y)$. Then we can complete the diagram as follows by Proposition~\ref{prop:intuitionistic-interpolation}, with the two morphisms $X \to \sS$ being models of $\sS$.
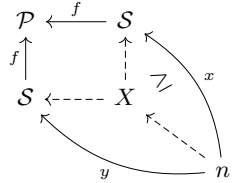
	\[\begin{tikzcd}
		\sP & \sS \ar[l,"f"'] &\\
		\sS \ar[u,"f"] & X \ar[l,dashed] \ar[u,dashed] &\\
		& & n \ar[lu,dashed] \ar[llu,bend left=25,"y"] \ar[luu,bend right=25,"x"',""{name=V,below,pos=0.7}]
		\ar[from=V,to=\tikzcdmatrixname-2-2,phantom,"\geq"{sloped,description}]
	\end{tikzcd}\]
	Because of the hypothesis, the two morphisms $X \to \sS$ are actually equal, which shows that $x \leq y$.
\end{proof}

The dual interpretation is the Beth definability theorem. Suppose $\hD_1 \to \hD_2$ is a morphism of intuitionistic hyperdoctrines. Suppose that each model of $\hD_1$ can be extended to a model of $\hD_2$ in at most one way. Then the dual of Proposition~\ref{prop:Beth-def} says that $\hD_1(n) \to \hD_2(n)$ is surjective for each $n \in \cC$, which means that $\hD_2$ is obtained from $\hD_1$ by adding axioms only, which in turn trivially implies that the models of $\hD_2$ are a subset of the models of $\hD_1$. In terms of first-order signatures, each $n$-ary symbol in the signature of $\hD_2$ can be defined by a $n$-ary predicate in the signature of $\hD_1$.

Conceptual completeness \cite[Thm.~7.1.8]{MakkaiReyes} says that if $F \colon \cT \to \cE$ is a logical functor between pretoposes that induces an equivalence $F^* \colon \Mod(\cE) \to \Mod(\cT)$ at the level of the categories of models, then $F$ itself is an equivalence. Below is a version of conceptual completeness in our context. However, it is equivalent to only half of conceptual completeness in \cite{MakkaiReyes}, namely \cite[Thm.~7.1.4]{MakkaiReyes}, because the base category $\cC$ of the polyadic spaces we consider is fixed.

\begin{proposition}\label{prop:concept-compl}
	Let $\sS,\sP$ be two polyadic compact ordered spaces and let $f \colon \sS \to \sP$ be any morphism. Let $f^* \colon \Mod(\sS) \to \Mod(\sP)$ be the induced functor.
	\begin{enumerate}
		\item $f^*$ is always faithful.
		\item If $f^*$ is essentially surjective, then $f$ is pointwise surjective.
		\item If $f^*$ is an embedding (fully faithful), then $f$ is a pointwise embedding.
		\item If $f^*$ is an equivalence, then $f$ is an isomorphism.
	\end{enumerate}
\end{proposition}

\begin{proof}
	The functor $f^*$ is necessarily faithful since composing it with the forgetful functor $\Mod(\sP)  \to  \cInd(\cC)$ gives the faithful forgetful functor $\Mod(\sS)  \to  \cInd(\cC)$.
	
	If $f^*$ is essentially surjective and if $x  \in   \sP(n)$, then $x$ is realized by some model of $\sP$, which can be pulled back to $\sS$ and this gives an antecedent of $x$ by $f$.
	
	Suppose that $f^*$ is an embedding. We want to show that for all (non commutative) triangles
	\[\begin{tikzcd}
		& \sS\\
		X \ar[ur] \ar[r] & Y \ar[u]
	\end{tikzcd}\]
	such that $X \to Y \to \sS \to \sP  \geq  X \to \sS \to \sP$, we have $X \to Y \to \sS  \geq  X \to \sS$. The fact that $f^*$ is an embedding tells us that this is true if $X \to \sS$ and $Y \to \sS$ are models. If only $X \to \sS$ is a model, we can factorize $Y \to \sS$ through a model $Y' \to \sS$ thanks to Gödel's completeness theorem and we obtain the result. If $X \to \sS$ is not a model, we factorize it through a model $X' \to \sS$. Since $X \to Y \to \sP  \geq  X \to X' \to \sP$, we obtain an interpolant $Z \to \sP$ as follows.
	
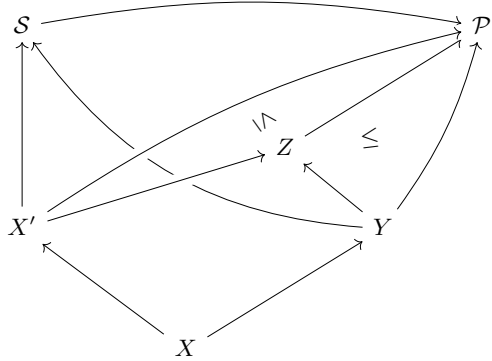
\[\begin{tikzcd}
		&\sS &&&&& \sP \\
		\\
		&&&& Z \\
		& X' &&&& Y \\
		\\
		&&& X
		\arrow[from=6-4, to=4-6]
		\arrow[from=6-4, to=4-2]
		\arrow[from=4-2, to=1-2]
		\arrow[bend left=25, from=4-6, to=1-2]
		\arrow[bend left=10, from=1-2, to=1-7]
		\arrow[from=3-5, to=1-7]
		\arrow[bend right=10, from=4-6, to=1-7,""{name=LU,left}]
		\arrow[from=4-2, to=3-5, crossing over]
		\arrow[from=4-6, to=3-5]
		\arrow[bend left=10, from=4-2, to=1-7, crossing over, ""{name=UL,right}]
		\ar[from=UL,to=3-5,phantom," \leq "{description,sloped,pos=0.6}]
		\ar[from=3-5,to=LU,phantom," \leq "{description,sloped}]
	\end{tikzcd}\]
	We complete the diagram on the left below like on the right.
	
	\begin{minipage}{0.48\linewidth}
		\[\begin{tikzcd}
			\sS \ar[r] & \sP\\
			Y \ar[u] \ar[r] & Z \ar[u] \ar[lu,phantom," \geq "{sloped,description}]
		\end{tikzcd}\]
	\end{minipage}\begin{minipage}{0.48\linewidth}
		\[\begin{tikzcd}
			\sS \ar[r] & \sP\\
			Y \ar[u] \ar[r] & Z \ar[u,""{left,name=A}] \ar[lu]
			\ar[from=1-1,to=A,phantom," \geq "{description,sloped,pos=0.7}]
		\end{tikzcd}\]
	\end{minipage}
	
	We see that $X'  \to  Z  \to  \sS  \to  \sP  \geq  X'  \to  Z  \to  \sP  \geq  X'  \to  \sP = X'  \to  \sS  \to  \sP$. Hence $X'  \to  Z  \to  \sS  \geq  X'  \to  \sS$ since $X' \to \sS$ is a model. We deduce that $X \to Y \to \sS  \geq  X \to \sS$ with the diagram below.
	
	
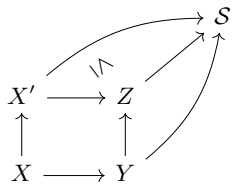
\[\begin{tikzcd}
		&&\sS\\
		X' \ar[r] \ar[urr,bend left=20,""{name=WW,right,near start}] & Z \ar[ru] & \\
		X \ar[u] \ar[r] & Y \ar[u] \ar[uur,bend right=20]
		\ar[from=WW,to=2-2,phantom," \leq "{description,sloped}]
	\end{tikzcd}\]
\end{proof}

\begin{remark}
	Note that the same situation as the two previous propositions appears already in the propositional case. The category of models is replaced by the underlying poset of the Priestley space.
	\begin{enumerate}
		\item A morphism of Esakia spaces which is injective is an embedding. (Proposition~\ref{prop:Beth-def}.)
		\item A morphism of Priestley spaces which gives an isomorphism at the level of posets is an isomorphism of Priestley spaces. (Proposition~\ref{prop:concept-compl}.)
	\end{enumerate}
	We cannot omit in Proposition~\ref{prop:Beth-def} that the morphism is intuitionistic, since the resulting proposition is already false at the propositional level, where intuitionistic morphisms become bounded morphisms of Priestley spaces.
\end{remark}

\begin{remark}\label{rmk:link-conceptual-comp}
	Theorem~7.1.4 in \cite{MakkaiReyes} can be obtained from Proposition~\ref{prop:concept-compl}, but for that we would need to develop a bit more the link between hyperdoctrines and pretoposes. We only sketch here how this can be done, leaving detailed proofs to future work. Let $f \colon \cE \to \cC$ be a morphism of coherent categories and let $f^* \colon \Mod(\cC) \to \Mod(\cE)$ be the induced functor between categories of models.
	
	Given a coherent category $\cE$, an \emph{$\cE$-hyperdoctrine} is a hyperdoctrine $\hD \colon \cE^{\op} \to \cDL$ such that whenever $(X_i \to Y)_i$ is a finite collection of morphisms covering $Y$, the canonical map $\hD(Y) \to \prod_i \hD(X_i)$ is injective. The initial $\cE$-hyperdoctrine is $\Sub_\cE \colon \cE^{\op} \to \cDL$, sending an object to its lattice of subobjects.
	
	The adjunction between coherent categories and hyperdoctrines described in \cite{Coumans12} specializes to an adjunction between $\cE$-hyperdoctrines and coherent categories equipped with a morphism from $\cE$. This adjunction actually gives $\cE$-hyperdoctrines as a reflective subcategory of coherent categories equipped with a morphism from $\cE$. The morphism $f \colon \cE \to \cC$ is sent to $\Sub_\cC \circ f$, and coming back into coherent categories produces the factorization $\cE \to \cC' \to \cC$ where $\cC'$ is the full subcategory of $\cC$ on objects which are subobjects of some $f(x)$ with $x \in \cE$. Moreover, models of $\cC'$ are equivalent to models of $\Sub_\cC \circ f$ (and models of $\cE$ are equivalent to models of $\Sub_\cE$).
	
	The functor $\Mod(\cC') \to \Mod(\cE)$ is always faithful, as said earlier, and $\Mod(\cC) \to \Mod(\cC')$ is always essentially surjective, since a model $\cC' \to \cSet$ can be extended to a model $\cC \to \cSet$ with a left Kan extension.
	\begin{enumerate}
		\item If $f^*$ is full, then $\Mod(\cC') \to \Mod(\cE)$ too and Proposition~\ref{prop:concept-compl} says that $\Sub_\cE \to \Sub_\cC \circ f$ is pointwise surjective, which means that $f$ is ``full on subobjects'' in the terminology of \cite{MakkaiReyes}.
		\item If $f^*$ is essentially surjective, then $\Mod(\cC') \to \Mod(\cE)$ too and $\Sub_\cE \to \Sub_\cC \circ f$ is pointwise injective, which means that $\cE \to \cC'$ is conservative, hence so is $f$.
		\item If $f^*$ is full and essentially surjective, then $\Sub_\cE \to \Sub_\cC \circ f$ is an isomorphism, hence $\cE \to \cC'$ is an equivalence.
	\end{enumerate}
\end{remark}

\section{Variations on intuitionistic logic}
\label{sec:variations-int}

We will now prove completeness theorems for variations of first-order intuitionistic logic: with a linearity axiom \cite{Corsi1992}, on constant domains \cite{Grzegorczyk64}, and both (Gödel-Dummett logic) \cite{Takano87}. We will work in this section with $\omega$-saturated Kripke models exclusively, and in order to shorten our notation, we will just call them Kripke models.

\subsection{Linearity}

A Heyting algebra satisfies the linearity axiom if for all elements $\phi$ and $\psi$, we have
\[ (\phi \to \psi) \lor (\psi \to \phi) = \top \text{.} \]
If $X$ is the dual compact ordered space, this is equivalent to the fact that $X$ is a forest with branches pointing downward, meaning that ${\uparrow}z$ must be linearly ordered for each $z \in X$. If it is the case, we will say that $X$ is \emph{locally linear}. We also say that a $\cC$-adic compact ordered space $\sS$ is \emph{locally linear} if $\sS(n)$ is locally linear for all $n \in \cC$. A Kripke model $F \colon \cI \to \Mod_\omega(\sS)$ is \emph{linear} if $\cI$ is a linear order, and it implies that $\tilde{F}$ is locally linear. We now show that linear Kripke models are a complete semantics for locally linear polyadic compact ordered spaces.

\begin{proposition}
	Let $\sS$ be a locally linear $\cC$-adic compact ordered space. Let $\cI$ be a linearly ordered set, let $(X_i)_i$ be an $\cI$-indexed diagram in $\cInd(\cC)$ and let $X_i \to \sS$ be an oplax cocone. Then we can complete $X_i \to \sS$ into a linear Kripke model in the sense that there is an embedding of linear orders $\cI \subseteq \cI'$, a diagram $(X'_i)_{i \in \cI'} \subseteq \cInd(\cC)$ equipped with a natural transformation $X_i \to X'_i$ and a Kripke model $(X'_i \to \sS)_i$ such that $X_i \to X'_i \to \sS = X_i \to \sS$.
	Moreover, if $0 \in \cI$ is initial and $X_0 \to \sS$ is a model, then we can take $X_0 = X'_0$.
\end{proposition}

\begin{proof}
	As a preliminary remark, we note that filtered limits of locally linear compact ordered spaces are also locally linear, so that $\ovl{\sS}\colon \cInd(\cC)^\op \to \cKOrd$ is locally linear if $\sS$ is.
	
	In order to ensure that $(X_i \to \sS)_{i\in\cI}$ is a Kripke model, we need to be able to solve two kinds of problems:
	\begin{enumerate}[wide, labelwidth=!, labelindent=0pt]
		\item[(Existential problems.)] For every $i \in \cI$ and for every lax commutative square as on the left (with $n,m \in \cC$), there is a way of completing it as on the right.\\
		\begin{minipage}{0.485\linewidth}
			\[\begin{tikzcd}
				X_i \ar[r] & \sS\\
				n \ar[u] \ar[r] & m \ar[u] \ar[lu,phantom,"\geq"{sloped,description}]
			\end{tikzcd}\]
		\end{minipage}\begin{minipage}{0.485\linewidth}
			\[\begin{tikzcd}
				X_i \ar[r] & \sS\\
				n \ar[u] \ar[r] & m \ar[u,""{left,name=LU}] \ar[lu]
				\ar[from=LU,to=\tikzcdmatrixname-1-1,phantom,near start,"\geq"{sloped,description}]
			\end{tikzcd}\]
		\end{minipage}
		\item[(Universal problems.)] For every $i \in \cI$ and for every lax commutative square as on the left (with $n,m \in \cC$), there is some $j \geq i$ and a way of completing the diagram as on the right.\\
		\begin{minipage}{0.485\linewidth}
			\[\begin{tikzcd}
				X_i \ar[r] & \sS\\
				n \ar[u] \ar[r] & m \ar[u] \ar[lu,phantom,"\leq"{sloped,description}]
			\end{tikzcd}\]
		\end{minipage}\begin{minipage}{0.485\linewidth}
			\[\begin{tikzcd}
				& & \sS \\
				X_i \ar[r] \ar[rru,bend left=25,""{right,name=LU}] & X_j \ar[ru] &\\
				n \ar[u] \ar[r] & m \ar[u] \ar[uur,bend right=25] &
				\ar[from=LU,to=\tikzcdmatrixname-2-2,phantom,"\leq"{sloped,description}]
			\end{tikzcd}\]
		\end{minipage}
	\end{enumerate}
	
	We will apply Lemma~\ref{lem:build-sol} as in the proof of Gödel's completeness theorem. The objects of our category $\cD$ are the linear orders $\cI$ equipped with an $\cI$-indexed diagram $(X_i)_i \subseteq \cInd(\cC)$ and an oplax cocone $(X_i \to \sS)_i$. The morphisms from $(\cI,(X_i)_i)$ to $(\cJ,(Y_j)_j)$ are the order-preserving maps $f \colon \cI \to \cJ$ equipped with a natural transformation $X_i \to Y_{f(i)}$ such that $X_i \to \sS = X_i \to Y_{f(i)} \to \sS$.
	
	The linear colimits in $\cD$ are computed as follows. Let $(\cI_k, X_k)_k$ be a linear diagram in $\cD$. Let $\cJ = \colim_k \cI_k$. Each $j \in \cJ$ can be written as $j = (\cI_k \to \cJ)(i)$ for some $k$ and some $i \in \cI_k$, and we define $Y_j = \colim_{k' \geq k} (X_{k'})_{(\cI_k \to \cI_{k'})(i)}$. Then $\colim_k (\cI_k, X_k)_k = (\cJ, Y)$.
	
	Given $(\cI,X) \in \cD$, define $\operatorname{EU}(\cI,X)$ to be the set of existential and universal problems as presented above. This produces a functor $\operatorname{EU} \colon \cD \to \cSet$ and it preserves linear colimits. Let $S(\cI,X) \subseteq \operatorname{EU}(\cI,X)$ be the subset of existential and universal problems having a solution. We must prove that this is a subfunctor. Let $(\cI,X) \to (\cJ,Y)$ be a morphism in $\cD$ with underlying order-preserving map $f \colon \cI \to \cJ$. Given a solvable existential problem for $(\cI,X)$ as below on the left, its image is a solvable existential problem as illustrated below on the right.
	
	\begin{minipage}{0.485\linewidth}
		\[\begin{tikzcd}
			X_i \ar[r] & \sS\\[1em]
			n \ar[u] \ar[r] & m \ar[u,""{left,name=LU}] \ar[lu]
			\ar[from=LU,to=\tikzcdmatrixname-1-1,phantom,near start,"\geq"{sloped,description}]
		\end{tikzcd}\]
	\end{minipage}\begin{minipage}{0.485\linewidth}
		\[\begin{tikzcd}
			Y_{f(i)} \ar[r] & \sS\\[-1em]
			X_i \ar[u] &\\[-1em]
			n \ar[u] \ar[r] & m \ar[uu,""{left,name=LU}] \ar[lu] &
			\ar[from=LU,to=\tikzcdmatrixname-1-1,phantom,bend left=15,near start,"\geq"{sloped,description}]
		\end{tikzcd}\]
	\end{minipage}
	
	The same goes for universal problems: the image of a solvable universal problem illustrated below on the left is also a solvable universal problem.
	
	\begin{minipage}{0.485\linewidth}
		\[\begin{tikzcd}
			& & \sS \\
			X_i \ar[r] \ar[rru,bend left=25,""{right,name=LU}] & X_j \ar[ru] &\\[1em]
			n \ar[u] \ar[r] & m \ar[u] &
			\ar[from=LU,to=\tikzcdmatrixname-2-2,phantom,"\leq"{sloped,description}]
		\end{tikzcd}\]
	\end{minipage}\begin{minipage}{0.485\linewidth}
		\[\begin{tikzcd}
			& & \sS \\
			Y_{f(i)} \ar[r] \ar[rru,bend left=25,""{right,name=LU}] & Y_{f(j)} \ar[ru] &\\[-1em]
			X_i \ar[u] \ar[r] & X_j \ar[u] &\\[-1em]
			n \ar[u] \ar[r] & m \ar[u] &
			\ar[from=LU,to=\tikzcdmatrixname-2-2,phantom,"\leq"{sloped,description}]
		\end{tikzcd}\]
	\end{minipage}
	
	In order to apply Lemma~\ref{lem:build-sol}, we need to show that for each $(\cI,X) \in \cD$ and for each existential or universal problem, there is some morphism $(\cI,X) \to (\cJ,Y)$ sending the problem to one with a solution.
	
	\paragraph{Existential problems} Suppose we have an existential problem as below.
	\[\begin{tikzcd}
		X_i \ar[r] & \sS\\
		n \ar[u] \ar[r] & m \ar[u] \ar[lu,phantom,"\geq"{sloped,description}]
	\end{tikzcd}\]
	Take $\cJ = \cI$, $Y_k = X_k$ if $k < i$ and $Y_k = X_k \sqcup_n m$ if $k \geq i$, with the obvious connecting morphism $(\cI,X) \to (\cJ,Y)$. We want to build an oplax cocone $Y_k \to \sS$ for $k \geq i$ such that $X_k \to Y_k \to \sS = X_k \to \sS$ and such that $m \to \sS \leq m \to Y_i \to \sS$. In order to do so, we will use compactness: we are looking for a point in the compact space $\prod_{k\geq i} \sS(Y_k)$ satisfying a collection of closed conditions. So we must show that for each finite family of these conditions, there is a point satisfying them. Each finite family of conditions impacts only finitely many of the $Y_k$, so we are reduced to showing the statement when $\cI$ is finite and $i$ is the minimal element. Suppose $\cI = \{0,1,\dots,t\}$ with $i=0$. We will define $Y_k \to \sS$ inductively. The strong interpolation property of $\sS$ and the interpolation extension principle allow us to find some $Y_0 \to \sS$ factoring $X_0 \to \sS$ and such that $m \to \sS \leq m \to Y_0 \to \sS$. We now want to define $Y_1 \to \sS$, and we find ourselves in a similar situation, having replaced the cocartesian square $(n,m,X_0,Y_0)$ by the cocartesian square $(X_0,Y_0,X_1,Y_1)$. We can continue the induction.
	
	\[\begin{tikzcd}
		\vdots & \vdots &\\
		X_1 \ar[u] \ar[r] & Y_1 \ar[r] \ar[u] & \sS\\
		X_0 \ar[u] \ar[r] & Y_0 \ar[u] \ar[ru,""{left,name=LU}] &\\
		n \ar[u] \ar[r] & m \ar[u] \ar[ruu,bend right=15,""{left,name=LU1}] &
		\ar[from=LU1,to=\tikzcdmatrixname-3-2,phantom,"\geq"{sloped,description}]
		\ar[from=LU,to=\tikzcdmatrixname-2-2,phantom,"\geq"{sloped,description}]
	\end{tikzcd}\]
	
	\paragraph{Universal problems} Suppose we have a universal problem as below.
	
	\[\begin{tikzcd}
		X_i \ar[r] & \sS\\
		n \ar[u] \ar[r] & m \ar[u] \ar[lu,phantom,"\leq"{sloped,description}]
	\end{tikzcd}\]
	
	We will first reduce the problem to the case where $n \to m \to \sS \not\geq n \to X_k \to \sS$ for all $k > i$. In order to do so, consider the set $A$ of all $j \in \cI$ which are either below $i$ or such that $n \to m \to \sS \geq n \to X_j \to \sS$. Define $X_A = \colim_{j \in A} X_j$ and build the map $X_A \to \sS$ as in Lemma~\ref{lem:lax-colim-models}. We then insert $X_A$ between $A$ and $\cI \setminus A$ in $(X_k)_{k \in \cI}$. We have $n \to m \to \sS \geq n \to X_A \to \sS$ and we can replace $i$ by $A$ (a solution to the universal problem for $A$ implies a solution for $i$). But for all $k \not\in A$, we don't have $n \to m \to \sS \geq n \to X_k \to \sS$.
	
	Now, we will suppose that $n \to m \to \sS \not\geq n \to X_k \to \sS$ for all $k > i$. (Actually, it implies that $n \to m \to \sS < n \to X_k \to \sS$ by local linearity of $\sS$, but we don't need that now.) First, define $\cJ$ as $\cI \uplus \{i'\}$, where $i'$ is an element added just after $i$. Define
	
	\[Y_k = \begin{cases}X_k & \text{if $k \leq i$,}\\X_i \sqcup_n m & \text{if $k = i'$,}\\X_k \sqcup_n m & \text{if $k > i'$.}\end{cases}\]
	
	We leave out the description of the morphism $(\cI,X) \to (\cJ,Y)$, it is the obvious one associated to the canonical inclusion $\cI \hookrightarrow \cJ$. We still need to define the morphisms $Y_k \to \sS$ for $k \geq i'$. The morphism $Y_{i'} = X_i \sqcup_n m \to \sS$ is chosen using the strong interpolation property of $\sS$ and the interpolation extension principle, as in the following diagram.
	
	\[\begin{tikzcd}
		& & \sS\\
		X_i \ar[r] \ar[rru,bend left=10,""{right,name=LU}] & X_i \sqcup_n m \ar[ru] &\\
		n \ar[u] \ar[r] & m \ar[u] \ar[ruu,bend right=10]
		\ar[from=LU,to=\tikzcdmatrixname-2-2,phantom,"\geq"{sloped,description}]
	\end{tikzcd}\]
	
	This will ensure that the image of our universal problem in $(\cJ,Y)$ has a solution. After that, we use the same technique as for existential problems to show that we can choose the morphisms $Y_k \to \sS$ for $k > i'$. Thanks to compactness, we can suppose that the set of $k > i'$ is finite. Let $j$ be the successor of $i'$. Since $X_i \to X_j \to \sS \geq X_i \to \sS$ and $X_i \to Y_{i'} \to \sS \geq X_i \to \sS$, and thanks to the local linearity of $\sS$, the morphisms $X_i \to X_j \to \sS$ and $X_i \to Y_{i'} \to \sS$ are comparable. But we cannot have $X_i \to Y_{i'} \to \sS \geq X_i \to X_j \to \sS$. Indeed, precomposing with $n \to X_i$, we would find that $n \to X_i \to Y_{i'} \to \sS = n \to m \to Y_{i'} \to \sS = n \to m \to \sS$ is greater than $n \to X_i \to X_j \to \sS = n \to X_j \to \sS$. But we supposed that $n \to m \to \sS \not\geq n \to X_j \to \sS$ for all $j > i$. We conclude that $X_i \to Y_{i'} \to \sS < X_i \to X_j \to \sS$.
	
	\[\begin{tikzcd}
		X_j \ar[r] & \sS\\
		X_i \ar[u] \ar[r] & Y_{i'} \ar[u] \ar[lu,phantom,"\geq"{sloped,description}]
	\end{tikzcd}\]
	
	Thanks to the strong interpolation property of $\sS$ and the interpolation extension principle, we can complete the diagram as below.
	
	\[\begin{tikzcd}
		& & \sS\\
		X_j \ar[r] \ar[rru,bend left=10] & Y_j \ar[ru] &\\
		X_i \ar[u] \ar[r] & Y_{i'} \ar[u] \ar[ruu,bend right=10,""{left,name=LU}]
		\ar[from=LU,to=\tikzcdmatrixname-2-2,phantom,"\geq"{sloped,description}]
	\end{tikzcd}\]
	
	At this point, we have $X_j \to Y_j \to \sS = X_j \to \sS \leq X_j \to X_{j+1} \to \sS$, where $j+1$ is the successor of $j$, so that we can continue to factor $X_{j+1} \to \sS$ as $X_{j+1} \to Y_{j+1} \to \sS$ inductively.
	
	Finally, we can apply Lemma~\ref{lem:build-sol} to the functors $S \subseteq \operatorname{EU}$ and conclude the proof.
\end{proof}

Notice that we used only once in the proof above the hypothesis that $\sS$ is locally linear.

\subsection{Constant domain}

In this subsection, we add the co-Frobenius rule to intuitionistic logic, which goes as follows:
\[ \forall x \colon (\phi(y) \lor \psi(x,y)) = \phi(y) \lor \forall x:\psi(x,y) \text{.} \]
This is the order-dual of the Frobenius law, and on the topological side, the corresponding axiom is that of co-boundedness. A $\cC$-adic compact ordered space $\sS$ is \emph{co-bounded} if all the diagrams like the one on the left can be completed as on the right, where $n,m \in \cC$.

\begin{minipage}{0.485\linewidth}
	\[\begin{tikzcd}
		& \sS\\
		n \ar[ur,""{name=UL,right}] \ar[from=2-2,to=UL,phantom,"\leq"{description,sloped}] \ar[r] & m \ar[u,"x"']
	\end{tikzcd}\]
\end{minipage}\begin{minipage}{0.485\linewidth}
	\[\begin{tikzcd}
		& \sS\\
		n \ar[ur] \ar[r] & m \ar[u,bend right=50,"x"',""{name=L,left}] \ar[u,bend left=35,""{name=R,right}]
		\ar[from=L,to=R,phantom,"\leq"{description,sloped}]
	\end{tikzcd}\]
\end{minipage}

The boundedness axiom \ref{p-axiom:interpol-1} of $\cC$-adic compact ordered spaces is obtained by reversing the direction of the inequalities above.

A Kripke model $F \colon \cI \to \Mod_\omega(\sS)$ of a $\cC$-adic space $\sS$ has \emph{constant domain} if the composite $\cI \to \Mod_\omega(\sS) \to \cInd(\cC)$ is constant. This means that we have an ind-object $X \in \cInd(\cC)$ and a family of elements $(x_i)_{i\in\cI} \subseteq \sS(X)$ with $x_i \leq x_j$ if $i \leq j$, such that each time we have a diagram as one in the left column below, we can complete it as in the diagram on its right.

\begin{minipage}{0.485\linewidth}
	\[\begin{tikzcd}
		X \ar[r,"x_i"] & \sS\\
		n \ar[r] \ar[u] & m \ar[u] \ar[lu,phantom,"\geq"{description,sloped}]
	\end{tikzcd}\]
\end{minipage}\begin{minipage}{0.485\linewidth}
	\[\begin{tikzcd}
		X \ar[r,"x_i"] & \sS\\
		n \ar[r] \ar[u] & m \ar[u,""{name=R,left}] \ar[lu]
		\ar[from=1-1,to=R,phantom,"\geq"{description,sloped,pos=0.7}]
	\end{tikzcd}\]
\end{minipage}

\begin{minipage}{0.485\linewidth}
	\[\begin{tikzcd}
		X \ar[r,"x_i"] & \sS\\
		n \ar[r] \ar[u] & m \ar[u] \ar[lu,phantom,"\leq"{description,sloped}]
	\end{tikzcd}\]
\end{minipage}\begin{minipage}{0.485\linewidth}
	\[\begin{tikzcd}
		X \ar[r,"x_i"{above},bend left=50,""{name=A,below}] \ar[r,"x_j"{below},""{name=B,above}] & \sS\\
		n \ar[r] \ar[u] & m \ar[u,""{name=R,left}] \ar[lu]
		\ar[from=A,to=B,phantom,"\leq"{description,sloped}]
	\end{tikzcd}\]
\end{minipage}

Co-boundedness and constant domain models are closely related: if $F$ has constant domain, then $\tilde{F} \colon \cC^\op \to \cPoset$ is co-bounded; in fact, one may even show that, if $\cC$ has finite coproducts, then the converse is true, but we will not need this fact. We will now show that constant domain models give a complete semantics for co-bounded compact ordered spaces.

In the following proposition and in its proof, a (rooted) \emph{tree} will be a poset $\cI$ such that ${\downarrow}i$ is linearly ordered for each $i \in \cI$ and such that $\cI$ has a minimal element that we call its \emph{root}. A \emph{sub-tree} will be a non-empty sub-poset $\cJ \subseteq \cI$ such that ${\downarrow}j \subseteq \cJ$ for all $j \in \cJ$. Below, when $\cI$ is a tree and $S$ is a poset, by an \emph{order-respecting family} $(x_i)_{i \in \cI} \subseteq S$, we mean an order preserving function $\cI \to S$ whose value at $i$ is $x_i$.

\begin{proposition}
	Let $\sS$ be a co-bounded $\cC$-adic compact ordered space. Let $X \in \cInd(\cC)$, let $\cI$ be a tree and let $(x_i)_{i \in \cI} \subseteq \sS(X)$ be an order-respecting family. Then we can complete $(x_i)_{i\in\cI}$ into a model in the following sense. There is a morphism $X \to Y$ in $\cInd(\cC)$, a tree $\cJ$ containing $\cI$ as a sub-tree, and a constant domain model $(y_j)_{j\in\cJ} \subseteq \sS(Y)$ such that for all $i \in \cI \subseteq \cJ$, we have $X \to^{x_i} \sS = X \to Y \to^{y_i} \sS$.
\end{proposition}

\begin{proof}
	We will apply Lemma~\ref{lem:build-sol} again. We have two kinds of problems to solve.
	
	Suppose we have the following configuration.
	
	\[\begin{tikzcd}
		X \ar[r,"x_i"] & \sS\\
		n \ar[r] \ar[u] & m \ar[u] \ar[lu,phantom,"\geq"{description,sloped}]
	\end{tikzcd}\]
	
	We can complete the diagram as below, thanks to the interpolation extension principle.
	
	\[\begin{tikzcd}
		&&\sS\\
		X \ar[rru,bend left=20,"x_i"] \ar[r] & X' \ar[ru,"a"{pos=0.35}]\\
		n \ar[r] \ar[u] & m \ar[u] \ar[uur,bend right=20,""{left,name=B,pos=0.35}]
		\ar[from=2-2,to=B,phantom,"\geq"{description,sloped}]
	\end{tikzcd}\]
	
	We then look for an order-respecting family $(x'_j)_{j\in\cI} \subseteq \sS(X')$ such that $x'_i = a$ and such that $X \to^{x_j} \sS = X \to X' \to^{x'_j} \sS$ for all $j \in \cI$. Thanks to compactness, it is enough to show that it is possible for all finite subsets of $\cI$. We can define $x'_j$ for $j < i$ inductively by using the co-boundedness hypothesis and the linearity of ${\downarrow}i$. After that, we use the boundedness axiom of $\sS$ (\ref{p-axiom:interpol-2}) to define $x'_j$ for all $j$, again thanks to the tree structure of $\cI$. This solves problems of the first kind.
	
	For the problems of the second kind, suppose we have a diagram as below.
	
	\[\begin{tikzcd}
		X \ar[r,"x_i"] & \sS\\
		n \ar[r] \ar[u] & m \ar[u] \ar[lu,phantom,"\leq"{description,sloped}]
	\end{tikzcd}\]
	
	Thanks to the interpolation extension principle, we complete our diagram as below.
	
	
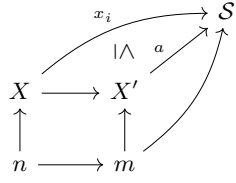
\[\begin{tikzcd}
		&&\sS\\
		X \ar[rru,bend left=20,"x_i",""{right,name=A}] \ar[r] & X' \ar[ru,"a"{pos=0.35}]\\
		n \ar[r] \ar[u] & m \ar[u] \ar[uur,bend right=20]
		\ar[from=A,to=2-2,phantom,"\leq"{description,sloped}]
	\end{tikzcd}\]
	
	We define $\cJ$ as $\cI$ with an additional point $i'$ above $i$ and incomparable with any point not below $i$. We define $x'_{i'} = a$ and using the same method as above, we extend this to an order-respecting family $(x'_j)_{j\in\cJ} \subseteq \sS(X')$ such that $X \to^{x_j} \sS = X \to X' \to^{x'_j} \sS$ for all $j \in \cI$. This solves problems of the second kind, and we can apply Lemma~\ref{lem:build-sol} to conclude.
\end{proof}

\subsection{Gödel-Dummett logic}

We now combine the two previous axioms. A polyadic compact ordered space $\sS$ is \emph{Gödel-Dummett} if it is both locally linear and co-bounded. On the algebraic side, we add the linearity and coFrobenius axioms. A \emph{Gödel-Dummett model} of $\sS$ is a linearly ordered Kripke model with constant domain. More explicitly, it is an ind-object $X$ of $\cC$ and an increasing linear sequence $(x_i)_{i\in\cI} \subseteq \sS(X)$ such that:

\begin{enumerate}
	\item Each lax diagram as below on the left can be completed as below on the right.
	
	\begin{minipage}{0.485\linewidth}
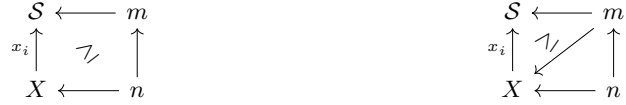
\[\begin{tikzcd}
		\sS & m \ar[l] \ar[dl,phantom,"\geq"{description,sloped}] \\
		X \ar[u,"x_i"] & n \ar[l] \ar[u]
	\end{tikzcd}\]\end{minipage}\begin{minipage}{0.485\linewidth}\[\begin{tikzcd}
		\sS & m \ar[l,""{below,name=UL}] \ar[dl] \\
		X \ar[u,"x_i"] & n \ar[l] \ar[u]
		\ar[from=UL,to=2-1,phantom,"\geq"{description,sloped,pos=0.35}]
	\end{tikzcd}\]\end{minipage}
	\item Each lax diagram as below on the left can be completed as below on the right.
	
	\begin{minipage}{0.485\linewidth}
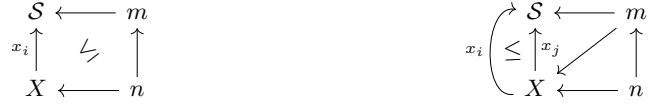
\[\begin{tikzcd}
		\sS & m \ar[l] \ar[dl,phantom,"\leq"{description,sloped}] \\
		X \ar[u,"x_i"] & n \ar[l] \ar[u]
	\end{tikzcd}\]\end{minipage}\begin{minipage}{0.485\linewidth}\[\begin{tikzcd}
		\sS & m \ar[l,""{below,name=UL}] \ar[dl] \\
		X \ar[u,"x_i",bend left=100,""{right,name=A}] \ar[u,bend right=0,"x_j"{right},""{left,name=B}] & n \ar[l] \ar[u]
		\ar[from=A,to=B,phantom,"\leq"{description,sloped}]
	\end{tikzcd}\]\end{minipage}
\end{enumerate}

This semantics is complete for Gödel-Dummett polyadic compact ordered spaces.

\begin{proposition}\label{prop:completeness-GD}
	Let $\sS$ be a Gödel-Dummett $\cC$-adic compact ordered space. Let $X \in \cInd(\cC)$ and let $(x_i)_{i\in\cI} \subseteq \sS(X)$ be an increasing sequence. Then there exists a Gödel-Dummett model $(y_i)_{i\in\cI'} \subseteq \sS(Y)$ with $\cI \subseteq \cI'$ and a morphism $X \to Y$ such that $x_i$ is $X \to Y \to^{y_i} \sS$ for all $i \in \cI$. If $\cI$ has a minimal element $0$, it is possible to preserve it in $\cI'$.
\end{proposition}

\begin{proof}
	We will again use the method of diagrams and apply Lemma~\ref{lem:build-sol}. Suppose we have a configuration as below.
	
	
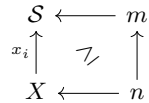
\[\begin{tikzcd}
		\sS & m \ar[l] \ar[dl,phantom,"\geq"{description,sloped}] \\
		X \ar[u,"x_i"] & n \ar[l] \ar[u]
	\end{tikzcd}\]
	
	Thanks to the interpolation extension principle, we can complete the diagram as follows for some $X'$ and $x'_i \in \sS(X')$.
	
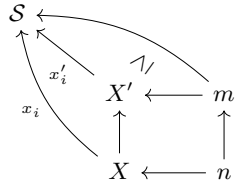
	\[\begin{tikzcd}
		\sS&&\\
		&X' \ar[lu,"x'_i"{pos=0.35}] & m \ar[l] \ar[llu,bend right=20,""{below,name=A,pos=0.35}]\\
		&X \ar[uul,"x_i",bend left=20] \ar[u] & n \ar[l] \ar[u]
		\ar[from=A,to=2-2,phantom,"\geq"{description,sloped}]
	\end{tikzcd}\]
	
	Tanks to the boundedness of $\sS$, and thanks to the interpolation extension principle, we can define $x'_j \in \sS(X')$ for $j > i$ such that $x_j = X \to X' \to^{x'_j} \sS$. Symmetrically for $j < i$. (We also use compactness to reduce to the case where $\cI$ is finite.) Thus situations of the first kind can be solved.
	
	Suppose now we have a configuration as below.
	
	\[\begin{tikzcd}
		\sS & m \ar[l] \ar[dl,phantom,"\leq"{description,sloped}] \\
		X \ar[u,"x_i"] & n \ar[l] \ar[u]
	\end{tikzcd}\]
	
	We can complete it as below thanks to the interpolation property and boundedness of $\sS$.
	
	\[\begin{tikzcd}
		\sS&&\\
		&X' \ar[lu,"x'"'{pos=0.35}] & m \ar[l] \ar[llu,bend right=20]\\
		&X \ar[uul,"x_i",bend left=20,""'{name=A,pos=0.35}] \ar[u] & n \ar[l] \ar[u]
		\ar[from=A,to=2-2,phantom,"\leq"{description,sloped}]
	\end{tikzcd}\]
	
	Once again, we use compactness to reduce to the case where $\cI$ is finite. We define $\cI'$ as $\cI$ with one point $p$ added above all the $j \in \cI$ such that $x_j \leq X \to X' \to^{x'} \sS$. We define $x'_p = x'$. Using the co-boundedness of $\sS$, we define $x'_j \in \sS(X')$ for $j < p$ such that $x_j = X\to X' \to^{x'_j} \sS$. For all $j > p$, we have $x_j > X \to X' \to^{x'} \sS$ since $\sS$ is locally linear. So we can define $x'_j \in \sS(X')$ such that $x_j = X\to X' \to^{x'_j} \sS$ thanks to the boundedness of $\sS$. This solves situations of the second kind.
	
	We can now apply Lemma~\ref{lem:build-sol} and it proves the proposition.
\end{proof}

\section*{Acknowledgments}

We are grateful to Mai Gehrke for her generous advice and guidance, and to André Joyal for providing the inspiration for this paper. We also thank Pino Rosolini and Josh Wrigley for helpful discussions. Finally, we would like to thank the anonymous reviewer for their careful reading of the paper and for their many thoughtful remarks. The research reported here has been supported financially by the European Research Council (ERC) under the European Union's Horizon 2020 research and innovation program, grant agreement \#670624.

\AtNextBibliography{\small}
{\printbibliography[
heading=bibintoc,
title={References}
]}

\end{document}